\documentclass{amsart}
\usepackage{amsmath ,amssymb ,amsthm, mathrsfs, stmaryrd, dsfont}
\usepackage{bm}
\usepackage[all]{xy}
\usepackage[pdftex]{graphicx} 
\usepackage[pdftex ,pdfborder={0 0 0}]{hyperref} 
\usepackage{tikz, tikz-cd}
\usepackage{shadethm}
\usepackage{graphicx, xcolor}
\usepackage{url}
\DeclareMathAlphabet\mathbfcal{OMS}{cmsy}{b}{n}
\DeclareMathAlphabet{\mathpzc}{OT1}{pzc}{m}{it}

\theoremstyle{definition}
\newtheorem{defin}{Definition}[section]

\newtheorem{thm}[defin]{Theorem}

\newtheorem{lem}[defin]{Lemma}
\newtheorem{prop}[defin]{Proposition}
\newtheorem{cor}[defin]{Corollary}
\newtheorem*{cor*}{Corollary}

\newtheorem*{thmA}{Theorem A}
\newtheorem*{thmB}{Theorem B}

\theoremstyle{remark}
\newtheorem{rem}[defin]{Remark}
\newtheorem{eg}[defin]{Example}

\newcommand{\Ric}{\mathrm{Ric}}
\newcommand{\cFut}{\check{\mathrm{F}}\mathrm{ut}}
\newcommand{\Fut}{\mathrm{Fut}}
\newcommand{\cmu}{\bm{\check{\mu}}}

\newcommand{\id}{\mathrm{id}}

\newcommand{\D}{\mathfrak{D}}
\newcommand{\Vol}{\mathrm{Vol}}

\title[Equivariant calculus and $\mu$K-stability]{Equivariant calculus on $\mu$-character and $\mu$K-stability of polarized schemes}

\author{Eiji Inoue}

\address{RIKEN, iTHEMS, 2-1 Hirosawa, Wako, Saitama 351-0198, Japan. }

\email{eiji.inoue@riken.jp}

\begin{document}

\begin{abstract}
We introduce and study $\mu$K-stability of polarized schemes with respect to general test configurations as an algebro-geometric aspect of the existence of $\mu$-cscK metrics which is introduced in the paper \cite{Ino2} as a framework unifying the frameworks of K\"ahler--Ricci solitons and cscK metrics. 
This article consists of two ingredients. 

On one hand, we develop a foundational framework concerning `derivative of relative equivariant intersection', which we call \textit{equivariant calculus}. 
A core claim in equivariant calculus is a convergence result for some infinite series in equivariant cohomology, which is obtained by \textit{relative equivariant intersections}. 
Our proof is based on some observations on deRham--Cartan model of equivariant locally finite homology. 
This framework furnishes a language to describe $\mu$K-stability. 
We in particular conclude the $\mu$K-semistability of $\mu$-cscK manifolds with respect to general test configurations. 

On the other hand, we introduce an equivariant character $\cmu^\lambda$ called \textit{$\mu$-character} for equivariant family of polarized schemes, motivated by the $\mu$-volume functional introduced in the paper \cite{Ino2} as a generalization of Tian--Zhu's functional in the theory of K\"ahler--Ricci soliton. 
The equivariant derivative of the $\mu$-character derive $\mu$-Futaki invariant for general test configuration, and furthermore, it also produces an analogue of the equivariant first Chern class of CM line bundle for family of polarized schemes, which is irrational and hence cannot be realized as a $\mathbb{Q}$-line bundle in our general $\mu$K-stability setup. 
\end{abstract}

\maketitle

\tableofcontents

\section{Introduction}

K-stability is a central theme in recent studies of algebraic geometry, which was introduced by \cite{Tia1, Don} as an algebro-geometric aspect of the existence of cscK metrics on K\"ahler manifolds. 
Odaka--Wang's intersection formula \cite{Oda1, Wang1} of Donaldson--Futaki invariant 
\begin{equation}
\mathrm{DF} (\mathcal{X}, \mathcal{L}) = (K_{\bar{\mathcal{X}}/\mathbb{P}^1}. \bar{\mathcal{L}}^{\cdot n}) - \frac{n}{n+1} \frac{(K_X. L^{\cdot (n-1)})}{(L^{\cdot n})} (\bar{\mathcal{L}}^{\cdot (n+1)}) 
\end{equation}
is a foundation of these studies. 

The author introduced in \cite{Ino2} a variant of cscK metrics, called $\mu$-cscK metrics, which encloses both cscK metrics and K\"ahler--Ricci solitons. 
The notion is motivated by the moment map picture for K\"ahler--Ricci soliton \cite{Ino1} and hence admits a variant of K-stability, which we call $\mu$K-stability. 
The aim of this paper is to introduce an equivariant intersection formula of $\mu$-Futaki invariant and to establish a foundation of $\mu$K-stability, as an algebro-geometric aspect of the existence of $\mu$-cscK metrics. 

We exhibit two main theorems: the $\mu$K-semistability of $\mu$-cscK manifolds with respect to general test configurations (Theorem A) and the construction of an equivariant characteristic class for equivariant family of polarized schemes which generalizes the equivariant Chern class of CM line bundle for the usual K-stability (Theorem B). 
The construction is applied in \cite{Ino5} to show the analytic moduli space of Fano manifolds with K\"ahler--Ricci soliton constructed in \cite{Ino1} is algebraic. 
We will pursue this in the article \cite{Ino6} in preparation. 

While Theorem A and B are presented as main theorems, the main achievement of this article is rather the establishment of equivariant calculus. 
Indeed, Theorem A is just reduced to Lahdili's result \cite{Lah} on the boundedness of weighted Mabuchi functional, which is the key analysis of the result. 
In the subsequent studies \cite{Ino3, Ino4}, we propose another approach for Theorem A which does not rely on the boundedness of weighted Mabuchi functional. 
This article provides a basic language for these studies. 

Before going on, we caution the sign $\pm$ in our equivariant calculus is quite sensitive to various conventions on group action and equivariant cohomology from the very beginning. 
We learned from \textit{H\"ansel and Gretel} that the simplest way not to get lost in the deep forest (of sign convention) is to drop pebbles along the way you come from your home. 
This is the reason we begin with clarifying all these basic `trivial' conventions in section \ref{sign convention}. 

Throughout this article, we use equivariant cohomology theory, especially the Cartan model of equivariant deRham cohomology / current homology. 
We arranged Appendix for the readers who are not familiar with this material, but it also plays a role in the sign clarification. 
The readers familiar to this material are also recommended to look over this section. 

\subsection{Backgrounds and key ingredients}

\subsubsection{Convention on spaces}

We only consider schemes of locally finite type over $\mathbb{C}$ and often identify these with complex spaces (not necessarily reduced nor irreducible), for which we have natural candidates for cohomology and its dual theory with \textit{$\mathbb{R}$-coefficient}: singular cohomology and locally finite homology (or sheaf cohomology and Borel--Moore homology). 
In the construction of (relative) exponential equivariant intersection, we discuss some convergence of equivariant cohomology classes in the real topology. 
To show the convergence, we reduce the problem to smooth manifolds by the resolution of singularities and make use of the Cartan model of equivariant deRham cohomology \& current homology. 
This is the reason for our current assumption on schemes. 
Our arguments work also for complex spaces, however, we are rather interested in algebraic aspects of $\mu$K-stability, so we would formulate things in an algebraic way and explain the complex analytic case in remarks. 

We denote by $\mathbb{G}_m$ the multiplicative group $\mathbb{C}^\times$ and by $\mathbb{A}^1$ the affine space $\mathbb{C}$ endowed with the right $\mathbb{G}_m$-action $z.t = t z$. 
To distinguish, we denote by $\mathbb{A}^1_-$ the affine space $\mathbb{C}$ endowed with the reverse right $\mathbb{G}_m$-action $z. t = t^{-1} z$ and by $0_-$ its origin $0 \in \mathbb{C}$, which we use when compactifying test configuration. 
We identify $\mathbb{P}^1$ with the quotient of $\mathbb{A}^1 \sqcup \mathbb{A}^1_-$ by the identification map $r: \mathbb{A}^1_- \setminus \{ 0_- \} \to \mathbb{A}^1 \setminus \{ 0 \}: z \mapsto z^{-1}$, so that we have a $\mathbb{G}_m$-action on $\mathbb{P}^1$ and natural $\mathbb{G}_m$-equivariant inclusions $\mathbb{A}^1, \mathbb{A}^1_- \subset \mathbb{P}^1$. 

For an algebraic torus $T = N \otimes \mathbb{G}_m$, we put $\mathfrak{t} := N \otimes \mathbb{R}$ and identify it with the Lie algebra of the maximal compactum $T_{\mathrm{cpt}} = N \otimes U (1) \subset T$ via $\xi \mapsto \frac{d}{dt}|_{t=0} \exp (t 2\pi \sqrt{-1} \xi)$, where $\exp: N \otimes \mathbb{C} \to T$ is induced by $\mathbb{C} \to \mathbb{G}_m: z \mapsto e^z$. 
In particular, for $\mathbb{G}_m = \mathbb{Z} \otimes \mathbb{G}_m$, the positive generator $1 \in \mathbb{Z}$ assigns $\eta \in \mathfrak{u} (1)$: $\eta = \frac{d}{dt}|_{t=0} \exp 2\pi \sqrt{-1} t$. 
We denote by $\eta^\vee \in \mathfrak{u} (1)^\vee$ the dual basis: $\langle \eta^\vee, a \eta \rangle = a$. 

\subsubsection{Weighted scalar curvature and $\mu$-cscK metric}

We briefly explain some basic notions around $\mu$-cscK metric. 
The author introduced $\mu$-cscK metric in the last paper \cite{Ino2} to establish an inclusive framework enclosing both cscK metric and K\"ahler--Ricci soliton. 
The notion of $\mu$-cscK metric is motivated by the moment map picture for K\"ahler--Ricci soliton observed in \cite{Ino1}. 
Meanwhile, Lahdili \cite{Lah} also considered a generalization of Donaldson--Fujiki type moment map picture and introduced weighted scalar curvature as a far extensive framework, which includes $\mu$-scalar curvature as a special case in the framework. 
The moment map picture on weighted scalar curvature yields a version of Yau--Tian--Donaldson conjecture which states that the existence of weighted scalar curvature on a given manifold must be equivalent to a proper notion of `weighted K-stability'. 
Lahdili introduced a differential geometric weighted Futaki invariant for test configurations with \textit{smooth} total spaces and proved that every weighted cscK manifold has non-negative weighted Futaki invariants for all smooth test configurations. 
Since we reduce our Theorem A to his result \cite{Lah}, we begin with his framework. 

Let $X$ be a compact K\"ahler manifold with a Hamiltonian action of a closed torus $T_{\mathrm{cpt}}$. 
Let $\omega$ be a $T_{\mathrm{cpt}}$-invariant K\"ahler metric and $\mu: X \to \mathfrak{t}^\vee$ be a moment map. 
Since $X$ is compact, the moment polytope $P = \mu (X)$ (and even the measure $\mu_* (\omega^n/n!)$ on $\mathfrak{t}^\vee$ supported on $P$) depends only on the equivariant cohomology class $L_T := {^\hbar \Phi}^{-1} [\omega + \hbar \mu] \in H^2_T (X, \mathbb{R})$ (cf. \cite[Section 2.3--2.4]{GGK}). 
Here ${^\hbar \Phi}^{-1}: {^\hbar H^2_{\mathrm{dR}, T_{\mathrm{cpt}} }} (X) \to H^2_T (X, \mathbb{R})$ is the isomorphism extending the Chern--Weil isomorphism $S \mathfrak{t}^\vee \to H^2_T (\mathrm{pt}, \mathbb{R})$ scaled by $\hbar \in \mathbb{R}^\times$. 
See Appendix for the construction of these equivariant cohomologies and section \ref{sign convention} for the sign convention on the Chern--Weil isomorphism. 

For a smooth positive function $v$ on $P$, Lahdili \cite{Lah} defines the \textit{weighted scalar curvature} $s_v (\omega)$ by 
\begin{equation} 
s_v (\omega) := s (\omega) \cdot (v \circ \mu^\omega) + \Delta_\omega (v \circ \mu^\omega) - \frac{1}{2} \sum_{1 \le i, j \le k} J\xi_i (\mu^\omega_{\xi_j}) \cdot (\frac{\partial^2 v}{\partial x^i \partial x^j} \circ \mu^\omega). 
\end{equation}
Note here we follow K\"ahlerian convention on the scalar curvature $s (\omega) = g^{i \bar{\jmath}} R_{i \bar{\jmath}} = n \frac{\Ric (\omega) \wedge \omega^{n-1}}{\omega^n}$, so it is the half of the Riemannian scalar curvature. 

When $v$ is of the form $v (x) = \tilde{v} (\langle x, \xi \rangle)$ with some smooth positive function $\tilde{v}$ on $\mathbb{R}$ and $\xi \in \mathfrak{t}$, we can simplify it as 
\begin{align*} 
s_v (\omega) 
&= s (\omega) \cdot (\tilde{v} \circ \mu^\omega_\xi) + \big{(} \Delta_\omega \mu^\omega_\xi \cdot (\tilde{v}' \circ \mu^\omega_\xi) - (\nabla \mu^\omega_\xi, \nabla \mu^\omega_\xi) \cdot (\tilde{v}'' \circ \mu^\omega_\xi) \big{)} - \frac{1}{2} J\xi (\mu^\omega_{\xi}) \cdot (\tilde{v}'' \circ \mu^\omega_\xi) 
\\
&= s (\omega) \cdot (\tilde{v} \circ \mu^\omega_\xi) + \Delta_\omega \mu^\omega_\xi \cdot (\tilde{v}' \circ \mu^\omega_\xi) + \frac{1}{2} J\xi (\mu^\omega_{\xi}) \cdot (\tilde{v}'' \circ \mu^\omega_\xi). 
\end{align*}
For $\tilde{v} = e^{\hbar t}$, we derive the \textit{${^\hbar \check{\mu}}$-scalar curvature} ${^\hbar s}_\xi (\omega)$: 
\begin{equation} 
s_v (\omega)  = \big{(} s (\omega) + \hbar \Delta \mu_\xi + \frac{\hbar^2}{2} J \xi (\mu_\xi) \big{)} \cdot e^{\hbar \mu_\xi} =: {^\hbar s_\xi} (\omega) \cdot e^{\hbar \mu_\xi} 
\end{equation}
The $\mu$-scalar curvature introduced in \cite{Ino2} is nothing but the ${^\hbar \check{\mu}}$-scalar curvature for $\hbar = -2$: for $\theta_\xi = -2 \mu_\xi$ and $\xi^J = J \xi + \sqrt{-1} \xi$, we have $\sqrt{-1} \bar{\partial} \theta_\xi = i_{\xi^J} \omega$ and 
\[ {^{-2} s_\xi^\lambda} (\omega) = s (\omega) + \bar{\Box} \theta_\xi + \bar{\Box} \theta_\xi - \xi^J (\theta_\xi) - \lambda \theta_\xi. \]
We introduce the parameter $\hbar \in \mathbb{R}^\times$ as it helps to clarify the sign convention in our equivariant calculus. 

For a real number $\lambda \in \mathbb{R}$, we introduce the following variant 
\begin{equation} 
{^\hbar s^\lambda_\xi} (\omega) := {^\hbar s_\xi} (\omega) - \lambda \mu_\xi 
\end{equation}
and call it \textit{${^\hbar \check{\mu}^\lambda_\xi}$-scalar curvature}. 
We call a K\"ahler metric $\omega$ a \textit{${^\hbar \check{\mu}^\lambda_\xi}$-cscK metric} if ${^\hbar s^\lambda_\xi} (\omega)$ is a constant. 
Since moment maps are unique modulo constant, the notion of ${^\hbar \check{\mu}^\lambda_\xi}$-cscK metric is independent of the choice of the moment maps. 

The constant 
\begin{equation} 
{^\hbar \bar{s}^\lambda_\xi} := \int_X {^\hbar s^\lambda_\xi} (\omega) e^{\hbar \mu_\xi} \omega^n \Big{/} \int_X e^{\hbar \mu_\xi} \omega^n 
\end{equation}
depends only on $\lambda, \xi \in \mathfrak{t}$ and the equivariant cohomology class $L_T = {^\hbar \Phi^{-1}} [\omega + \hbar \mu]$. 
For $\lambda = 0$, it depends only on $\lambda, \xi \in \mathfrak{t}$ and the cohomology class $L = [\omega] \in H^2 (X, \mathbb{R})$ rather than $L_T$. 

It is observed in \cite{Ino1} that the ${^{-2} \check{\mu}^\lambda_\xi}$-cscK metric in a K\"ahler class $L$ satisfying $\lambda L = -2\pi K_X$ is equivalent to K\"ahler--Ricci soliton: $\Ric (\omega) - L_{\xi^J} \omega = \lambda \omega$. 

\subsubsection{$\mu$-Futaki invariant and $\mu$-volume functional for vectors}
\label{mu-Futaki and mu-volume for vectors}

Based on the moment map picture for ${^\hbar \check{\mu}}$-scalar curvature \cite{Ino1, Ino2, Lah}, we consider the following \textit{${^\hbar \check{\mu}}^\lambda_\xi$-Futaki invariant}: 
\begin{equation} 
\label{mu-Futaki invariant}
{^\hbar \cFut^\lambda_\xi} (\zeta) = \int_X ({^\hbar s^\lambda_\xi} (\omega) - {^\hbar \bar{s}^\lambda_\xi}) \cdot \hbar \mu_\zeta e^{\hbar \mu_\xi} \omega^n \Big{/} \int_X e^{\hbar \mu_\xi} \omega^n. 
\end{equation}
for $\eta \in \mathfrak{t}$. 
It depends only on the K\"ahler class $L = [\omega]$ and hence vanishes if there exists a ${^\hbar \check{\mu}^\lambda_\xi}$-cscK metric in $L$. 

In \cite{Ino2}, we studied the \textit{${^\hbar \mu^\lambda}$-volume functional} ${^\hbar \Vol}^\lambda: \mathfrak{t} \to \mathbb{R}$ 
\begin{equation} 
\log {^\hbar \Vol^\lambda} (\xi) := \int_X {^\hbar s^\lambda_\xi} (\omega) e^{\hbar \mu_\xi} \omega^n \Big{/} \int_X e^{\hbar \mu_\xi} \omega^n +\lambda \log \int_X e^{\hbar \mu_\xi} \omega^n. 
\end{equation}
This depends only on the cohomology class $L$ and is characterized up to constant by the property 
\begin{equation}
D_\xi (\log {^\hbar \Vol^\lambda}) (\zeta) = \frac{d}{dt}\Big{|}_{t=0} \log {^\hbar \Vol^\lambda} (\xi + t\zeta) = {^\hbar \cFut^\lambda_\xi} (\zeta). 
\end{equation}
If there is a ${^\hbar \mu^\lambda_\xi}$-cscK metric, then the vector $\xi$ must be a critical point of ${^\hbar \mathrm{Vol}^\lambda}$. 

The functional generalizes Tian--Zhu's volume functional $\int_X e^{\theta_\xi} \omega^n$ introduced in the study of K\"ahler--Ricci soliton \cite{TZ}. 
Indeed, since 
\[ {^\hbar \bar{s}^\lambda_\xi} = \frac{{^\hbar (- 2\pi K_X^T - \lambda L_T. e^{L_T}; \xi)}}{{^\hbar (e^{L_T}; \xi)}} + \lambda n, \]
we have ${^\hbar \bar{s}^\lambda_\xi} = \lambda n$ when $\lambda L_T = 2\pi c_{T, 1} (X)$, which is equivalent to say that $\lambda L = -2 \pi K_X$ and $\mu$ is normalized as $\int_X \mu e^h \omega^n = 0$ by the Ricci potential $h: \mathrm{Ric} (\omega) - \lambda \omega = \sqrt{-1} \partial \bar{\partial} h$ (see Example \ref{Equivariant Chern class of canonical bundle}). 

It is proved in \cite{Ino2} that the $\mu$-volume functional always admits a minimizer and the critical points are unique for $\lambda \ll 0$, while always \textit{not} unique for $\lambda \gg 0$. 
Studying the second variation of the functional, we can also prove the set of vectors associated to some ${^\hbar \check{\mu}^\lambda}$-cscK metric is (empty or) finite for each $\lambda \le 0$. 
The author speculates the critical points are unique whenever $\lambda \le 0$. 

We can express the $\mu$-volume functional using the following equivariant intersections: 
\begin{gather*}
\int_X (s (\omega) + \hbar \bar{\Box} \mu_\xi) e^{\hbar \mu_\xi} \omega^n/n! = \int_X (\Ric_\omega + \hbar \bar{\Box} \mu_\xi) e^{\omega + \hbar \mu_\xi} = - 2 \pi \cdot {^\hbar (K_X^T. e^{L_T}; \xi)}, 
\\
\int_X (n + \hbar \mu_\xi) e^{\hbar \mu_\xi} \omega^n/n! = \int_X (\omega + \hbar \mu_\xi) e^{\omega + \hbar \mu_\xi} = {^\hbar (L_T. e^{L_T}; \xi)}, 
\\
\int_X e^{\hbar \mu_\xi} \omega^n/n! = \int_X e^{\omega + \hbar \mu_\xi} = {^\hbar (e^{L_T}; \xi)}. 
\end{gather*}
This is the key observation for this study. 
These equalities follow from the isomorphism (\ref{singular to Cartan isomorphism}) and Example \ref{Equivariant Chern class of canonical bundle}. 
Now we explain the notations. 

\subsubsection{Equivariant intersection}

Let $G$ be an almost connected Lie group, for which we have a maximal compact subgroup. 
Let $X$ be a locally compact Hausdorff topological space with a continuous $G$-action. 
We denote by $H^*_G (X)$ the equivariant cohomology and by $H^{\mathrm{lf}, G}_* (X)$ the equivariant locally finite homology, which are constructed as in section \ref{section: Equivariant singular cohomology and locally finite homology}. 
When we would emphasize that $\alpha \in H^{\mathrm{lf}, G}_* (X, \mathbb{R})$ and $\beta \in H^*_G (X, \mathbb{R})$ denote $G$-equivariant classes, we write it as $\alpha^G, \beta_G$, respectively. 
We denote by $\alpha_{\langle p \rangle} \in H^{\mathrm{lf}, G}_{2p} (X, \mathbb{R})$ and $\beta^{\langle q \rangle} \in H^{2q}_G (X, \mathbb{R})$ the projection to the even degree. 

Let $\varpi: \mathcal{X} \to B$ be a proper continuous $G$-equivariant map of $G$-spaces. 
For $\alpha \in H^{\mathrm{lf}, G}_p (\mathcal{X}, \mathbb{R})$ and $\beta_i \in H^{q_i}_G (\mathcal{X}, \mathbb{R})$ ($i = 1, \ldots, k$), we define its \textit{relative equivariant intersection} by the equivariant pushforward of the cap product: 
\begin{equation}
{_B (\alpha. \beta_1. \dotsb. \beta_k)} := \varpi_* (\alpha \frown \beta_1 \smile \dotsb \smile \beta_k) \in H^{\mathrm{lf}, G}_{p - q_1 - \dotsb - q_k} (B, \mathbb{R}). 
\end{equation}
When $B$ is a topological manifold with an orientation preserving $G$-action, we denote its Poincare dual by 
\begin{equation} 
(\alpha. \beta_1. \dotsb. \beta_k)_B \in H_G^{\dim_{\mathbb{R}} B + q_1 + \dotsb + q_k - p} (B, \mathbb{R}). 
\end{equation}

More generally, we can consider the equivariant intersection for an element $\beta = (\beta^{\langle q/2 \rangle})_{q=0}^\infty$ in the formal completion of equivariant cohomology $\hat{H}_G (\mathcal{X}, \mathbb{R}) = \prod_{q=0}^\infty H^q_G (\mathcal{X}, \mathbb{R})$: 
\[ (\alpha. \beta)_B := ((\alpha. \beta^{\langle q/2 \rangle})_B)_{q=0}^\infty \in \hat{H}_G (B, \mathbb{R}). \]
In our study of $\mu$K-stability, we are especially interested in the equivariant intersection for the exponential class 
\[ \beta = e^{\mathcal{L}} = (\frac{1}{k!} \mathcal{L}^{\smile k})_{k=0}^\infty \in \hat{H}^{\mathrm{even}}_G (\mathcal{X}, \mathbb{R}) := \prod_{k=0}^\infty H^{2k}_G (\mathcal{X}, \mathbb{R}) \]
with $\mathcal{L} \in H^2_G (\mathcal{X}, \mathbb{R})$. 

\subsubsection{The Chern--Weil isomorphism and equivariant intersection as function}

Now let $\hbar \in \mathbb{R}^\times$ be a non-zero real number. 
Depending on the context, it is convenient to select different $\hbar$. 
For instance, it is convenient to choose $\hbar = -2$ in the context of K\"ahler--Ricci soliton and $\mu$-cscK metric since the notions are described in a K\"ahlerian convention: we use the $\bar{\partial}$-Hamiltonian potential $\theta_\xi = -2 \mu_\xi$ rather than $\mu_\xi$. 
In the subsequent studies \cite{Ino3, Ino4}, we choose $\hbar = 1$, whose sign is crucial for the results. 
In this article, we do not choose any specific $\hbar$. 
See section \ref{sign convention} for various sign convention. 

Fix a maximal compact Lie subgroup $K \subset G$. 
Consider a smooth manifold $M$ with a smooth $K$-action. 
As explained in section \ref{section: Cartan model of equivariant cohomology and locally finite homology}, we can construct the Cartan model of equivariant cohomology ${^\hbar H^*_{\mathrm{dR}, K}} (M)$ from a chain complex $(\Omega^*_K (M), {^\hbar d_K})$ consisting of $K$-equivariant differential forms. 
Here the boundary operator ${^\hbar d_K}$ (and hence its kernel and image) depends on $\hbar \in \mathbb{R}^\times$. 
For another $\hbar' \in \mathbb{R}^\times$, there is a natural isomorphism of functors: 
\[ (\hbar'/\hbar)_*: {^\hbar H^*_{\mathrm{dR}, K}} \xrightarrow{\sim} {^{\hbar'} H^*_{\mathrm{dR}, K}}. \]

Beware equivariant form is not a $K$-invariant form nor a $K$-basic form. 
For instance, for a $K$-invariant $2$-form $\omega$ and a $K$-equivariant smooth map $\mu: M \to \mathfrak{k}^\vee$ satisfying $- d\mu_\xi = i_{\xi^\#} \omega$ for every $\xi \in \mathfrak{k}$, the formal sum $\omega + \hbar \mu$ gives a ${^\hbar d_K}$-closed $K$-equivariant $2$-form: $({^\hbar d_K} (\omega + \hbar \mu)) (\xi) = d\omega + \hbar (d\mu_\xi + i_{\xi^\#} \omega) = 0$. 
Thus we can assign its $K$-equivariant cohomology class $[\omega + \hbar \mu] \in {^\hbar H^2_{\mathrm{dR}, K}} (M)$. 

When $K$ acts on $M$ via a $G$-action on $M$, we have isomorphisms 
\begin{align}
\label{singular to Cartan isomorphism}
{^\hbar \Phi_K^M} 
&: H^*_G (M, \mathbb{R}) \to {^\hbar H_{\mathrm{cur}, K}^*} (M), 
\\
{^\hbar \Phi^K_M}
&: H^{\mathrm{lf}, G}_* (M, \mathbb{R}) \to {^\hbar H^{\mathrm{cur}, K}_*} (M)
\end{align} 
from the Borel construction of equivariant cohomology / locally finite homology to the Cartan models of equivariant deRham cohomology / current homology. 
We often abbreviate these as ${^\hbar \Phi}$. 
For $\alpha \in H^*_G (M, \mathbb{R})$, we put 
\[ {^\hbar \alpha} := {^\hbar \Phi} (\alpha) \in {^\hbar H^*_{\mathrm{dR}, K}} (M). \]
In particular, we put 
\begin{align*}
{^\hbar (\alpha. \beta_1. \dotsb. \beta_k)_B} := {^\hbar \Phi} (\alpha. \beta_1. \dotsb . \beta_k)_B \in {^\hbar H^{\mathrm{cur}, K}_*} (B) 
\end{align*} 
when $B$ is a smooth manifold. 

When $B$ is a point, we abbreviate these as $(\alpha. \beta_1. \dotsb. \beta_k), {^\hbar (\alpha. \beta_1. \dotsb. \beta_k)}$, respectively. 
We have a natural identification ${^\hbar H_{\mathrm{dR}, K}^{2k}} (\mathrm{pt}) = (S^k \mathfrak{k}^\vee)^K$ by the very construction of the Cartan model (see section \ref{section: Cartan model of equivariant cohomology and locally finite homology}), so that we can identify the equivariant intersection ${^\hbar (\alpha. \beta_1. \dotsb. \beta_k)}$ for $\alpha \in H^{2p}_G (X, \mathbb{R}), \beta_i \in H^{\mathrm{lf}, G}_{2q_i} (X, \mathbb{R})$ with a $K$-invariant degree $q_1 + \dotsb +q_k -p$ polynomial function on $\mathfrak{k}$. 
Under this identification, the isomorphism $(\hbar'/\hbar)_*: (S^* \mathfrak{k}^\vee)^K \to (S^* \mathfrak{k}^\vee)^K$ is given by $((\hbar'/\hbar)_* f) (\xi) = f (\hbar'/\hbar \cdot \xi)$ for $\xi \in \mathfrak{k}$. 
We denote by ${^\hbar (\alpha. \beta_1. \dotsb. \beta_k; \xi)} \in \mathbb{R}$ the evaluation of ${^\hbar (\alpha. \beta_1. \dotsb. \beta_k)}$ at $\xi \in \mathfrak{k}$, for which we have ${^\hbar (\alpha. \beta_1. \dotsb. \beta_k; \xi)} = {^{\hbar'} (\alpha. \beta_1. \dotsb. \beta_k; (\hbar'/\hbar) \xi)}$. 

The sign convention of the Chern--Weil isomorphism ${^\hbar \Phi}^{-1}: (S \mathfrak{k}^\vee)^K \to H^*_G (\mathrm{pt}, \mathbb{R})$, which depends on $\hbar$, is crucial for the output ${^\hbar (\alpha. \beta_1. \dotsb. \beta_k; \xi)}$. 
We review this in section \ref{sign convention} by comparing two constructions of equivariant Chern classes. 

\subsubsection{The beauty of the Cartan model}
\label{The beauty of the Cartan model}

Consider a smooth $T_{\mathrm{cpt}} \times K$-action on a smooth manifold $M$. 
If $T_{\mathrm{cpt}}$ acts trivially on $M$, we have the decomposition
\begin{equation} 
\label{Cartan Kunneth}
{^\hbar H^*_{\mathrm{dR}, T_{\mathrm{cpt}} \times K}} (M) = S^* \mathfrak{t}^\vee \otimes {^\hbar H^*_{\mathrm{dR}, K}}, 
\end{equation}
which corresponds to the K\"unneth decomposition 
\begin{equation}
\label{Kunneth}
H^*_{T \times G} (X) = H^* (BT \times (X \times_G EG)) = H^*_T (\mathrm{pt}) \otimes H^*_G (X). 
\end{equation}
The Cartan model of equivariant cohomology has such decomposition at chain level, even when the $T$-action is non-trivial: 
\begin{equation} 
\Omega_{T_{\mathrm{cpt}} \times K}^* (M) = (S^* (\mathfrak{t} \times \mathfrak{k})^\vee \otimes \Omega^* (M))^{T_{\mathrm{cpt}} \times K} = S^* \mathfrak{t}^\vee \otimes \Omega_K^* (M)^{T_{\mathrm{cpt}}}, 
\end{equation}
where we put $\Omega_K^* (M)^{T_{\mathrm{cpt}}} := (S^* \mathfrak{k} \otimes \Omega^* (X)^{T_{\mathrm{cpt}}})^K \subset \Omega_K^* (M)$. 
We note the boundary map ${^\hbar d_{T_{\mathrm{cpt}} \times K}}$ is not compatible with this decomposition when the $T_{\mathrm{cpt}}$-action is non-trivial. 
We have a similar decomposition for equivariant currents which is compatible with the chain level proper pushforward map. 

In particular, we have the following \textit{chain-level evaluation map}. 
For $\xi \in \mathfrak{t}$, we have the following (non-linear) map
\begin{align}
\label{rho evaluation}
\rho^T_\xi: \Omega_{T_{\mathrm{cpt}} \times K}^k (M) = \bigoplus_{2p+q = k} S^p \mathfrak{t}^\vee \otimes \Omega_K^q (M)^{T_{\mathrm{cpt}}} 
&\to \bigoplus_{0 \le q \le k} \Omega^q_K (M)
\\ \notag
 \sum_{2p + q = k} \varphi_q 
&\mapsto \sum_{2p+q=k} \varphi_{q, \xi}, 
\end{align}
where we identify $\varphi_q \in S^p \mathfrak{t}^\vee \otimes \Omega_K^q (M)^{T_{\mathrm{cpt}}}$ with a degree $p$ homogeneous polynomial map $\mathfrak{t} \to \Omega_K^q (M)^{T_{\mathrm{cpt}}}$. 
The map is $T_{\mathrm{cpt}} \times K$-equivariant and is compatible with the chain-level proper pushforward map for equivariant currents (See section \ref{section: Cartan model of equivariant locally finite homology (current homology)}). 

When the $T_{\mathrm{cpt}}$-action is trivial, the ${^\hbar d}$-closedness/exactness is preserved by the evaluation map. 
In this case, the evaluation map descends to the evaluation map on the equivariant cohomology 
\begin{equation} 
{^\hbar \rho^T_\xi}: {^\hbar H^k_{\mathrm{dR}, T_{\mathrm{cpt}} \times K}} (M) \to \bigoplus_{0 \le q \le k} {^\hbar H^q_{\mathrm{dR}, K}} (M)
\end{equation} 
defined in a similar way using the decomposition (\ref{Cartan Kunneth}). 
We are especially interested in the following truncation: 
\begin{equation} 
\label{derivation}
\D^q_{\hbar. \xi} := q! \cdot t_q \circ {^\hbar \Phi_G^{-1}} \circ {^\hbar \rho^T_\xi} \circ {^\hbar \Phi_{T \times G}}: H^*_{T \times G} (M, \mathbb{R}) \to H^{2q}_G (M, \mathbb{R}), 
\end{equation}
where $t_q: H^*_G (M) \to H^{2q}_G (M)$ denotes the projection. 
We put $\D_{\hbar. \xi} := \D^1_{\hbar. \xi}$. 
We have $\D_{\hbar'. \xi}^q \alpha = \D_{\hbar. (\hbar'/\hbar) \xi}^q \alpha$, so the notation $\D_{\hbar. \xi}^q \alpha$ is not confusing. 
See also section \ref{Equivariant cohomology class of class epsilon ell}. 

When the $T_{\mathrm{cpt}}$-action is non-trivial, the map $\rho^T_\xi$ \textit{does not preserve the ${^\hbar d}$-closedness of equivariant forms}. 
This is the reason we cannot express the characteristic class $\D_\xi \cmu^\lambda$ in Theorem B by a simple proper pushforward of an equivariant cohomology class, different from the case of the equivariant cohomology class of CM line bundle.

\subsubsection{Key ingredient: Convergence result for equivariant intersection}

A scheme $X$ is called \textit{pure $n$-dimensional} if $\dim_{\mathbb{C}} Z = \dim_{\mathbb{C}} Z^\circ = n$ for every irreducible component $Z$ of $X$, where $Z^\circ$ denotes the smooth part of the reduction of $Z$. 

Let $G$ be an algebraic group and $X$ be a scheme with an algebraic $G$-action, which we call $G$-scheme. 
For a pure $n$-dimensional $G$-scheme $X$, we denote by $H^{\mathrm{alg}, G}_{2k} (X, \mathbb{R})$ the subspace of $H^{\mathrm{lf}, G}_{2k} (X, \mathbb{R}) = H^{\mathrm{lf}}_{2 \dim_\mathbb{C} B_{2n-2k} G + 2k} (E_{2n-2k} G \times_G X, \mathbb{R})$ spanned by $(\dim_\mathbb{C} B_{2n-2k} G + k)$-cycles on $E_{2n-2k} G \times_G X$ under the cycle map (\ref{equivariant cycle map}). 

Here the scheme $E_{2n-2k} G$ is taken as in section \ref{section: equivariant locally finite homology}, i.e. a $G$-invariant Zariski open set contained in the subset $\{ v \in V ~|~ \mathrm{Stab} (v) = 1 \}$ of a $G$-representation $V$ such that $\dim_{\mathbb{C}} (V \setminus E_{2n-2k} G) \ge n-k+1$. 
The subspace $H^{\mathrm{alg}, G}_{2k} (X, \mathbb{R})$ is independent of the choice of such $E_{2n-2k} G$ by \cite{EG1}. 
Beware that in general $H^{\mathrm{alg}, G}_{2k} (X, \mathbb{R})$ is larger than the subspace $\{ \sum_i a_i [E_{2n-2k} G \times_G Z_i ] ~|~ Z_i \subset X: G\text{-invariant divisor} \}$ spanned by $G$-equivariant fundamental classes of $G$-\textit{invariant} divisors on $X$. 

The following result is incorporated in the definition of $\mu$-Futaki invariant for test configurations and hence is the key ingredient of Theorem A. 

\begin{thm}
\label{convergence of absolute equivariant intersection}
Let $X$ be a pure $n$-dimensional proper scheme with an algebraic $T$-action. 
For any $L \in H^2_T (X, \mathbb{R})$ and $\alpha \in H^{\mathrm{alg}, T}_{2n-2} (X, \mathbb{R})$ (or $\alpha \in H^{\mathrm{lf}, T}_{2n} (X, \mathbb{R})$), the infinite series 
\[ {^\hbar (\alpha. e^L; \xi)} := \sum_{k=0}^\infty \frac{1}{k!} {^\hbar (\alpha. L^{\cdot k}; \xi)} \]
is compactly absolutely convergent. 
\end{thm}

The result is well-known for smooth $X$ as we can compute the equivariant intersection by equivariant integration of equivariant differential form. 
As for this absolute intersection, we can easily reduce the problem to the smooth case by equivariant resolution and equivariant projection formula. 
Compared to this, the following convergence result on relative equivariant intersection is less obvious. 

\begin{thm}
Let $\mathcal{X}$ be a pure dimensional $T \times G$-scheme, $B$ be a smooth $G$-variety and $\varpi: \mathcal{X} \to B$ be a $T \times G$-equivariant proper morphism of schemes. 
For $\mathcal{L} \in H^2_{T \times G} (\mathcal{X}, \mathbb{R})$ and $\alpha \in H^{\mathrm{alg}, T \times G}_{2 \dim \mathcal{X} -2} (\mathcal{X}, \mathbb{R})$, the following infinite series 
\[ \D^q_{\hbar. \xi} (\alpha. e^{\mathcal{L}})_B := \sum_{k=0}^\infty \frac{1}{k!} \D^q_{\hbar. \xi} (\alpha. \mathcal{L}^{\cdot k})_B \]
is compactly absolute-convergent in $H^{2q}_G (B, \mathbb{R})$. 
\end{thm}

Though we restrict our interest to schemes over $\mathbb{C}$ to show the convergence of this infinite series, the construction of the infinite series itself is rather universally designed in the sense that it does not rely on any specific construction of equivariant cohomology (see section \ref{Equivariant cohomology class of class epsilon ell}). 
Indeed, we can generalize the construction to schemes over any field, using for instance equivariant Chow cohomology with $\mathbb{R}$-coefficient instead of using equivariant singular cohomology (thanks to the projective space bundle formula $A (\mathbb{P}_k^N \times X) \cong A (\mathbb{P}_k^N) \otimes A (X)$).  

Now we explain our plan of the proof. 
First of all, for a resolution $\beta: X' \to X$, the morphism $\varpi' = \varpi \circ \beta: X' \to B$ is not a submersion, so the fibrewise integration of equivariant differential form is no longer equivariant differential form. 
Thus we make use of less familiar equivariant current and its pushforward. 

We can find equivariant differential forms $\varphi_k$ on $X'$ such that $c_k = {^\hbar \Phi^{-1}} [\int_{X'/B} \varphi_k]$. 
We naturally expect the convergence of the infinite series $\varphi = \sum_{k=0}^\infty \varphi_k$ implies the convergence of the infinite series $\sum_{k=0}^\infty c_k = \sum_{k=0}^\infty {^\hbar \Phi^{-1}} [\int_{X'/B} \varphi_k]$ of the pushforward homology classes in the equivariant current homology of $B$, but this is not so obvious. 
To show the convergence of $\sum_{k=0}^\infty [\int_{X'/B} \varphi_k]$, we must show the continuity of \textit{form-to-homology} pushforward map. 
Since the \textit{form-to-current} pushforward map is continuous, it suffices to show the topology on the equivariant current homology of the base $B$ induced from the construction of equivariant current homology coincides with the unique Hausdorff topology. 
In general, quotient spaces fail to be Hausdorff. 

Fortunately, we have the following affirmative answer. 
This is a key ingredient in Theorem B. 
The key idea is to reduce the problem to the non-equivariant case by computing the topology on the equivariant current homology via spectral sequence. 

\begin{thm}
\label{Hausdorffness of equivariant cohomology}
Let $K$ be a compact Lie group and $B$ be a smooth manifold with a smooth $K$-action. 
Then the quotient topology on the equivariant current homology ${^\hbar H_k^{\mathrm{cur}, K}} (B, \mathbb{R})$ induced from the weak topology on the space of equivariant currents $(\mathcal{D}')_k^K (B)$ is Hausdorff for every $k \in \mathbb{Z}$. 
\end{thm}

More basically, while the pushforward equivariant currents $\int_{X'/B} \varphi_k$ on $B$ become equivariantly closed, the integrands $\varphi_k$ itself are \textit{not} equivariantly closed. 
By this non-closedness, we cannot understand $\varphi_k$ in a way that makes sense independent of this specific construction, the Cartan model, of equivariant cohomology, contrary to the construction of $c_k$. 

The following diagram illustrates our approach for the proof of convergence. 
\[ \begin{tikzcd}
{\color{gray} \sum_{k=0}^\infty \frac{1}{k!} (A + \hbar \nu) \wedge (\Omega + \hbar \mu)^k} \ar[gray]{r} \ar[gray]{d} 
& {\color{gray} \sum_{k=0}^\infty \varphi_k = \varphi } \ar[gray]{dd}
\\
\sum_{k=0}^\infty \frac{1}{k!} (\alpha \frown \mathcal{L}^{\smile k}) \ar{d} 
& 
\\
(\alpha. e^{\mathcal{L}})_B = \sum_{k=0}^\infty \frac{1}{k!} (\alpha. \mathcal{L}^{\cdot k})_B \ar{r}{\D_{\hbar. \xi}}
& \sum_{k=0}^\infty \varphi_k {\color{gray} = {^\hbar \Phi^{-1}} [\int_{X'/B} \varphi] }
\end{tikzcd} \]

\subsubsection{Equivariant $\mathbb{Q}$-line bundle and relative positivity}

Our proofs on convergence results for relative equivariant intersections work for complex spaces and for general equivariant cohomology class $L \in H^2_G (X, \mathbb{R})$. 
On the other hand, in the proof of the naturality in Theorem B, we make use of equivariant Grothendieck--Riemann--Roch theorem, for which we need to restrict $L$ to some algebraic/holomorphic classes. 
We prepare some notations for this class. 

For a scheme $X$ over $\mathbb{C}$, we put 
\[ NS (X) := \{ c_1 (L) ~|~ L: \text{ algebraic line bundle over } X \} \subset H^2 (X, \mathbb{Z}) \]
and $NS (X, \Bbbk) := NS (X) \otimes \Bbbk$ for $\Bbbk = \mathbb{Q}, \mathbb{R}$. 
For a morphism $f: X \to S$ of schemes, we call an element $\alpha \in NS (X)$ \textit{relatively ample (resp. relatively semiample, relatively big)} if there exists some algebraic line bundle $L$ and $m \in \mathbb{N}$ for which $m\alpha = c_1 (L)$ and $L$ is $f$-ample (resp. $f^* f_* L \to L$ is surjective, $(L|_Z)^{\dim Z} > 0$ for every $s \in S$ and irreducible component $Z$ of $f^{-1} (s)$). 
For our sake, we may safely replace the relatively semiample condition with the weaker fibrewise semiample condition. 

Let $X$ be a $G$-scheme. 
For a $G$-equivariant algebraic line bundle $L$ on $X$, we can assign the equivariant first Chern class $c_{G, 1} (L) \in H^2_G (X, \mathbb{Z})$. 
We put 
\[ NS_G (X) := \{ c_{G, 1} (L) ~|~ L: G\text{-equivariant alg. line bundle over } X \} \subset H^2_G (X, \mathbb{Z}) \]
and $NS_G (X, \Bbbk) := NS_G (X) \otimes \Bbbk$. 
We denote the forgetful map $NS_G (X) \to NS (X)$ by $c_1$. 

\begin{rem}
We may use analytic objects: for a complex space $X$, one may use the subgroup $NS_G^{\mathrm{hol}} (X)$ consisting of the equivariant first Chern classes of holomorphic line bundles. 
\end{rem}

By a \textit{$G$-equivariant $\mathbb{Q}$-line bundle} on a $G$-complex space $X$, we mean a $G$-equivariant Neron--Severi class $L \in NS_G (X, \mathbb{Q}) \subset H^2_G (X, \mathbb{Q})$. 
When we would emphasize that $L$ denotes a $G$-equivariant class, we write it as $L_G$. 
For a morphism $f: X \to S$, we call a $G$-equivariant $\mathbb{Q}$-line bundle $L$ \textit{relatively ample (resp. relatively semiample, relatively big)} if the associated class $c_1 (L) \in NS (X, \mathbb{Q})$ is so. 

A pair $(X, L)$ of a pure dimensional projective $G$-scheme $X$ and an ample (resp. semiample and big) $G$-equivariant $\mathbb{Q}$-line bundle $L$ is called a \textit{$G$-polarized scheme (resp. $G$-semipolarized scheme)}. 
When $X$ is reduced and irreducible, we call it a $G$-polarized variety. 

For a $T$-equivariant $\mathbb{Q}$-line bundle $L$, a \textit{lift} of a one parameter subgroup $\Lambda: \mathbb{G}_m \to \mathrm{Aut}_T (X)$ to $L$ is an equivariant $\mathbb{Q}$-line bundle $L_{T \times \mathbb{G}_m} \in NS_{T \times \mathbb{G}_m} (X)$ which is mapped to $L \in NS_T (X)$ by the restriction to the subgroup $T \subset T \times \mathbb{G}_m$. 
Lifts to $L$ are unique modulo the image of $NS_{\mathbb{G}_m} (\mathrm{pt}) \to NS_{T \times \mathbb{G}_m} (X)$ if it exists. 

\subsubsection{Test configuration}
\label{section: Test configuration}

We describe basic notions on test configuration so that we can readily generalize notions to more general K\"ahler classes which do not lie in $NS_T (X, \mathbb{R})$. 
Having said that, as we already noted, we restrict our interest to the class $NS_T (X, \mathbb{R})$ in this article. 

\begin{defin}
Let $T$ be an algebraic torus and $(X, L)$ be a $T$-semipolarized scheme. 
A $T$-equivariant \textit{test configuration} $(\mathcal{X}, \mathcal{L})$ of $(X, L)$ consists of the following data: 
\begin{itemize}
\item A $T \times \mathbb{G}_m$-scheme $\mathcal{X}$ with a $T \times \mathbb{G}_m$-equivariant projective flat morphism $\varpi: \mathcal{X} \to \mathbb{A}^1$, where we define the $T \times \mathbb{G}_m$-action on the base $\mathbb{A}^1$ by $z. (t, \tau) = z \tau$ for $z \in \mathbb{A}^1$ and $(t, \tau) \in T \times \mathbb{G}_m$. 

\item A $T \times \mathbb{G}_m$-equivariant $\mathbb{Q}$-line bundle $\mathcal{L} \in NS_{T \times \mathbb{G}_m} (\mathcal{X}, \mathbb{Q})$ on $\mathcal{X}$ which is relatively semiample and relatively big. 

\item A $T \times \mathbb{G}_m$-equivariant isomorphism $\varphi: X \times (\mathbb{A}^1 \setminus \{ 0 \}) \xrightarrow{\sim} \mathcal{X} \setminus \mathcal{X}_0$ over the base with $\varphi^* \mathcal{L} = p_X^* L \in NS_{T \times \mathbb{G}_m} (X \times (\mathbb{A}^1 \setminus \{ 0 \}), \mathbb{Q}) \cong NS_T (X, \mathbb{Q})$, which we often abbreviate. 
\end{itemize}
We call a test configuration $(\mathcal{X}, \mathcal{L})$ \textit{ample} if $\mathcal{L}$ is relatively ample over $\mathbb{A}^1$. 
\end{defin}

For a test configuration $(\mathcal{X}, \mathcal{L})$, we construct its compacification $(\bar{\mathcal{X}}, \bar{\mathcal{L}})$ by the quotient of $\mathcal{X} \sqcup X \times \mathbb{A}^1_-$ under the identification map $\varphi \circ (\id_X \times r): (X \times \mathbb{A}^1_- \setminus \{ 0_- \}) \to \mathcal{X} \setminus \mathcal{X}_0$. 
The extended $\mathbb{Q}$-line bundle $\bar{\mathcal{L}} \in H^2_{T \times \mathbb{G}_m} (\bar{\mathcal{X}}, \mathbb{Q})$ is cohomologically characterized by the property $\bar{\mathcal{L}}|_{\mathcal{X}} = \mathcal{L}, \bar{\mathcal{L}}|_{X \times \mathbb{A}^1_-} = p_X^* L$, where we can use the following equivariant Mayer--Vietoris short exact sequence 
\[ 0 \to H^2_{T \times \mathbb{G}_m} (\bar{\mathcal{X}}) \to H^2_{T \times \mathbb{G}_m} (\mathcal{X}) \oplus H^2_{T \times \mathbb{G}_m} (X \times \mathbb{A}^1_-) \to H^2_{T \times \mathbb{G}_m} (X \times \mathbb{P}^1 \setminus \{ 0, 0_- \}) \to 0. \]
Here we can identify $H^*_{T \times \mathbb{G}_m} (X \times (\mathbb{P}^1 \setminus \{ 0, 0_- \}))$ with $H^*_T (X)$ and $H^*_{T \times \mathbb{G}_m} (X \times \mathbb{A}^1_-)$ with $H^*_T (X \times \mathrm{pt}_{\mathbb{G}_m}) = H^*_T (X) \otimes H^2 (\mathbb{C}P^\infty)$ and hence $H^*_{T \times \mathbb{G}_m} (X \times \mathbb{A}^1_-) \to H^*_{T \times \mathbb{G}_m} (X \times \mathbb{P}^1 \setminus \{ 0, 0_- \})$ is surjective. 

Though $\bar{\mathcal{L}}$ is mere relatively ample for ample $(\mathcal{X}, \mathcal{L})$, we can find a related test configuration $(\mathcal{X}, \mathcal{L}_c)$ with ample $\bar{\mathcal{L}}_c$ by twisting the weight of $\mathcal{L}$ as follows. 
The Poincare dual $[\mathrm{pt}]_{\mathbb{G}_m} \in H^2_{\mathbb{G}_m} (\mathbb{A}^1)$ of the equivariant fundamental class $[\mathrm{pt}]^{\mathbb{G}_m} \in H^{\mathrm{lf}, \mathbb{G}_m}_0 (\mathbb{A}^1)$ generates $H^2_{\mathbb{G}_m} (\mathbb{A}^1) \cong \mathbb{Z}$. 
The pullback of $[\mathrm{pt}]_{\mathbb{G}_m}$ along $\mathcal{X} \to \mathbb{A}^1$ is the equivariant fundamental class of the central fibre $[\mathcal{X}_0]_{\mathbb{G}_m} \in H^2_{T \times \mathbb{G}_m} (\mathcal{X})$. 
Since the restriction of $[\mathcal{X}_0]_{\mathbb{G}_m}$ to $\mathcal{X} \setminus \mathcal{X}_0$ is zero, for $c \in \mathbb{Q}$, $(\mathcal{X}, \mathcal{L}_c = \mathcal{L} + c. [\mathcal{X}_0]_{\mathbb{G}_m})$ gives another test configuration. 
The extended class $[\mathrm{pt}]_{\mathbb{G}_m} \in H^2_{\mathbb{G}_m} (\mathbb{P}^1)$ is ample, so by taking sufficiently large $c > 0$, we can make $\bar{\mathcal{L}}_c$ ample for ample $(\mathcal{X}, \mathcal{L})$. 

For a lift $L_{T \times \mathbb{G}_m}$ of a one parameter subgroup $\Lambda: \mathbb{G}_m \to \mathrm{Aut}_T (X)$, we can assign a $T$-equivariant test configuration $(X \times \mathbb{A}^1, p_X^* L_{T \times \mathbb{G}_m})$. 
We call it the ($T$-equivariant) \textit{product configuration} associated to $L_{T \times \mathbb{G}_m}$. 

Two test configurations $(\mathcal{X}, \mathcal{L}), (\mathcal{X}', \mathcal{L}')$ of $(X, L)$ are called \textit{equivalent} if the isomorphism $\varphi' \circ \varphi^{-1} : \mathcal{X} \setminus \mathcal{X}_0 \to \mathcal{X}' \setminus \mathcal{X}_0$ extends to an isomorphism $\mathcal{X} \setminus Z \cong \mathcal{X}' \setminus Z'$ for some $T$-invariant closed subschemes $Z, Z'$ of codimension greater than one. 
When $X$ is normal, equivalent test configurations give the same normalization, however, the reverse implication is not true. 
In terms of Boucksom--Jonsson's non-archimedean framework, the latter weaker equivalence is more essential. 
For a normal $X$, we say test configurations $(\mathcal{X}, \mathcal{L}), (\mathcal{X}', \mathcal{L}')$ of $(X, L)$ are \textit{pluripotentially equivalent} if the normalizations are naturally isomorphic to each other. 

\subsubsection{$\mu$-Futaki invariant and $\mu$K-stability}

Now we define the $\mu$-Futaki invariant for general test configurations using the equivariant intersection constructed in Theorem \ref{convergence of absolute equivariant intersection}. 

\begin{defin}[$\mu$-Futaki invariant]
Let $(X, L)$ be a $T$-polarized pure $n$-dimensional scheme. 
Fix parameters $\lambda \in \mathbb{R}$ and $\xi \in \mathfrak{t}$.
For a $T$-equivariant test configuration $(\mathcal{X}, \mathcal{L})$ of $(X, L)$, we define its \textit{$\check{\mu}^\lambda_{\hbar. \xi}$-Futaki invariant} by 
\begin{align*}
\cFut^\lambda_{\hbar. \xi} (\mathcal{X}, \mathcal{L}) 
&:= 2 \pi \frac{ {^\hbar (\kappa_{\bar{\mathcal{X}}/\mathbb{P}^1}^T. e^{\bar{\mathcal{L}}_T}; \xi)} \cdot {^\hbar (e^{L_T}; \xi)}  - {^\hbar (\kappa_X^T. e^{L_T}; \xi)} \cdot {^\hbar (e^{\bar{\mathcal{L}}_T}; \xi)} }{{^\hbar (e^{L_T}; \xi)}^2}
\\
& \qquad + \lambda \left[ \frac{ {^\hbar (\bar{\mathcal{L}}_T. e^{\bar{\mathcal{L}}_T}; \xi)} \cdot {^\hbar (e^{L_T}; \xi)} - {^\hbar (L_T. e^{L_T}; \xi)} \cdot {^\hbar (e^{\bar{\mathcal{L}}_T}; \xi)} }{{^\hbar (e^{L_T};\xi )}^2} - \frac{{^\hbar (e^{\bar{\mathcal{L}}_T}; \xi)}}{{^\hbar (e^{L_T}; \xi)}} \right]. 
\end{align*}
\end{defin}

Here $\kappa^T_X \in H^{\mathrm{lf}, T}_{2n-2} (X, \mathbb{Q})$ (resp. $\kappa^T_{\bar{\mathcal{X}}/\mathbb{P}^1} \in H^{\mathrm{lf}, T}_{2n} (\bar{\mathcal{X}}, \mathbb{Q})$) denotes the $T$-equivariant (resp. relative) canonical classes defined in Definition \ref{equivariant canonical class}, which is derived from the equivariant homology todd class $\tau_X^T (\mathcal{O}_X)$ studied in \cite{EG2}. 
These are well-defined for general (non-reduced, non-irreducible, non-normal) schemes and naturally appear in the equivariant Grothendieck--Riemann--Roch theorem. 
When $X$ is normal, the non-equivariant $\kappa_X$ coincides with the homology class of the canonical divisor $K_X$. 
In the equivariant case, even if $X$ is normal, $\kappa_X^T$ is not generally obtained from an \textit{invariant} $\mathbb{Q}$-divisor on $X$, but only from an \textit{equivariant} $\mathbb{Q}$-divisor, which is a divisor on the Borel construction $X \times_T ET$. 
Here one must recall the construction of equivariant locally finite homology. 
See section \ref{section: Equivariant singular cohomology and locally finite homology}. 

We will observe the relation to the differential geometric definition: we have 
\[ \hbar \cdot \cFut^\lambda_{\hbar .\xi} (\mathcal{X}, \mathcal{L}) = {^\hbar \cFut^\lambda_\xi} (\eta) \]
for the product test configuration associated to a lift $L_{T \times \mathbb{G}_m}$ of a $\mathbb{G}_m$-action on $X$ compatible with $T$. 
It is actually independent of the choice of the lift: it only depends on the $\mathbb{G}_m$-action on $X$. 
See Corollary \ref{algebraic mu-Futaki and differential mu-Futaki}. 

Since we formulated $\mu^\lambda_\xi$-cscK metric in \cite{Ino2} using $\bar{\partial}$-Hamiltonian potential $\theta_\xi = -2\mu_\xi$, the following convention fits into the convention of \cite{Ino2}: 
\begin{equation} 
\Fut^\lambda_\xi (\mathcal{X}, \mathcal{L}) := 2 \cFut^\lambda_{-2. \xi} (\mathcal{X}, \mathcal{L}). 
\end{equation}



\begin{rem}
We note the $\mu$-Futaki invariant is expressed by $T$-equivariant intersections, rather than $T \times \mathbb{G}_m$-equivariant intersections. 
The $\mathbb{G}_m$-action to the base direction is incorporated in the compactification $(\bar{\mathcal{X}}, \bar{\mathcal{L}})$. 

The $\mu$-Futaki invariant is invariant under the shift of the lift of the $\mathbb{G}_m$-action on $\mathcal{X}$ to $\mathcal{L}$: 
\[ \cFut^\lambda_{\hbar. \xi} (\mathcal{X}, \mathcal{L} + c. [\mathcal{X}_0]_{\mathbb{G}_m}) = \cFut^\lambda_{\hbar. \xi} (\mathcal{X}, \mathcal{L}). \]
Thus by shifting, we may restrict our interest to test configurations $(\mathcal{X}, \mathcal{L})$ with ample $\bar{\mathcal{L}}$. 
When computing equivariant intersection by the integration of equivariant differential form, we should be aware the associated moment map $\mu$ in $\Omega + \hbar \mu \in {^\hbar \Phi} (\mathcal{L})$ is replaced with $\mu - c. \eta^\vee$ under the shit $\mathcal{L} \mapsto \mathcal{L} + c. [\mathcal{X}_0]_{\mathbb{G}_m}$. 
\end{rem}

Now we define the $\mu$K-stability of a $T$-polarized scheme in the usual way. 

\begin{defin}[$\mu$K-stability]
We call a $T$-(semi)polarized scheme $(X, L)$ 
\begin{itemize}
\item \textit{${^\hbar \mu^\lambda_\xi}$K-semistable} if $\cFut_{\hbar. \xi}^\lambda (\mathcal{X}, \mathcal{L}) \ge 0$ for every test configuration $(\mathcal{X}, \mathcal{L})$ of $(X, L)$. 

\item \textit{$\mu^\lambda_\xi$K-polystable} if it is $\mu^\lambda_\xi$K-semistable and we have $\cFut_{\hbar. \xi}^\lambda (\mathcal{X}, \mathcal{L}) = 0$ for ample test configurations $(\mathcal{X}, \mathcal{L})$ only when $(\mathcal{X}, \mathcal{L})$ is equivalent to some product configuration. 

\item \textit{$\mu^\lambda_\xi$K-stable} if it is $\mu^\lambda_\xi$K-polystable and $\mathrm{Aut}_T^0 (X, L) = T$. 
\end{itemize}
\end{defin}

We will see when $X$ is normal, we may replace the equivalence condition in $\mu^\lambda_\xi$K-polystability with the pluripotential equivalence. 
Namely, if $(X, L)$ is a $\mu^\lambda_\xi$K-semistable normal variety and $(\mathcal{X}, \mathcal{L})$ is pluripotentially equivalent to a product configuration with $\cFut_{\hbar. \xi}^\lambda (\mathcal{X}, \mathcal{L}) = 0$, then $(\mathcal{X}, \mathcal{L})$ is equivalent to a product. 

\subsubsection{How to compute $\mu$-Futaki invariant}

Similarly as Donaldson--Futaki invariant, we can compute the $\mu$-Futaki invariant at least in three ways: (a) by integrating equivariant closed forms representing equivariant cohomology classes, (b) by following the definition of equivariant intersection and (c) by computing the equivariant index. 

For example, the equivariant intersection ${^\hbar (\kappa^T_{\bar{\mathcal{X}}/\mathbb{P}^1}. e^{\bar{\mathcal{L}}_T}; \xi)}$ can be computed as follows. 
We note $\kappa^T_{\bar{\mathcal{X}}/\mathbb{P}^1} \in H^{\mathrm{lf}, T}_{2n} (\bar{\mathcal{X}}, \mathbb{Q})$. 

(a) Suppose $\mathcal{X}$ is smooth, then we can compute ${^\hbar (\kappa^T_{\bar{\mathcal{X}}/\mathbb{P}^1}. e^{\bar{\mathcal{L}}_T}; \xi)}$ using equivariant forms $\Omega + \hbar \mu \in {^\hbar \Phi} (\bar{\mathcal{L}}_T)$ and $A + \hbar \nu \in {^\hbar \Phi} ( -c_{T, 1} (\bar{\mathcal{X}}) + c_1 (\mathbb{P}^1))$: 
\begin{align*} 
{^\hbar (\kappa^T_{\bar{\mathcal{X}}/\mathbb{P}^1}. e^{\bar{\mathcal{L}}_T}; \xi)} 
&= \int_{\bar{\mathcal{X}}} (A + \hbar \nu_{\xi}) e^{\Omega + \hbar \mu_{\xi}} 
\\
&= \int_{\bar{\mathcal{X}}} e^{\mu_{\hbar \xi}} \frac{A \wedge \Omega^n}{n!} + \int_{\bar{\mathcal{X}}} \nu_{\hbar \xi} e^{\mu_{\hbar \xi}} \frac{\Omega^{n+1}}{(n+1)!}. 
\end{align*}
Here $\Omega$ and $A$ are differential $2$-forms on $\bar{\mathcal{X}}$ and $\mu$ and $\nu$ are smooth maps from $\bar{\mathcal{X}}$ to the dual $\mathfrak{t}^\vee$: moment maps. 
For $\xi \in \mathfrak{t}$, we denote by $\mu_\xi, \nu_\xi: \bar{\mathcal{X}} \to \mathbb{R}$ the composition with $\langle \cdot , \xi \rangle: \mathfrak{t}^\vee \to \mathbb{R}$. 

(b) Take a basis $\{ \chi_i \}_{i=1}^r$ of the character lattice $M$ of $T$. 
Put $E_l^i \mathbb{G}_m := \mathbb{C}_{\chi_i}^{l+1} \setminus 0$ and $E_l T := \prod_{i=1}^r E_l^i \mathbb{G}_m$, where $\mathbb{C}_{\chi_i}$ denotes the representation associated to the character $\chi_i$: $\mathbb{G}_m$ acts on $\mathbb{C}_\chi$ from the right by $z. t = \chi (t) z$. 
For $l$ sufficiently large compared with $k$, the equivariant intersection $(\kappa^T_{\bar{\mathcal{X}}/\mathbb{P}^1}. \bar{\mathcal{L}}_T^{\cdot k})$ is identified with a degree $2(k-n)$ cohomology class on (a finite dimensional approximation of) the classifying space $B_l T := E_l T/T = \prod_{i=1}^r E_l^i \mathbb{G}_m/\mathbb{G}_m = \prod_{i=1}^r \mathbb{P}^l$ by our construction of equivariant locally finite homology (see section \ref{section: Equivariant singular cohomology and locally finite homology}). 
Under the isomorphism $\mathbb{Q} [x_1, \ldots, x_r] \cong H^{2*} (B_l T, \mathbb{Q}): x_i \mapsto p_i^* c_1 (\mathcal{O} (-1))$, we can identify $(\kappa^T_{\bar{\mathcal{X}}/\mathbb{P}^1}. \bar{\mathcal{L}}_T^{\cdot k})$ with a degree $2 (k-n)$ homogeneous polynomial $p (x_1, \ldots, x_r)$. 
Then the evaluation ${^\hbar (\kappa^T_{\bar{\mathcal{X}}/\mathbb{P}^1}. \bar{\mathcal{L}}_T^{\cdot k}; \xi)}$ for a vector $\xi = \sum \xi_i \eta_i \in \mathfrak{t}$ expressed with the dual basis $\{ \eta_i \}_{i=1}^r$ of $\{ \chi_i \}_{i=1}^r$ is computed by substituting $\hbar \xi_i \in \mathbb{R}$ into $x_i$. 

We recall $(\kappa^T_{\bar{\mathcal{X}}/\mathbb{P}^1}. \bar{\mathcal{L}}_T^{\cdot k}) \in H^{2 (k-n)} (B_l T)$ is constructed as follows. 
We denote by $\mathscr{L}_l$ a line bundle over $\bar{\mathcal{X}} \times_T E_l T$ obtained as the descent of the $T$-equivariant line bundle $p_{\bar{\mathcal{X}}}^* \bar{\mathcal{L}}$ over $\bar{\mathcal{X}} \times E_l T$. 
Suppose $\bar{\mathcal{X}}$ is normal, then we can similarly construct a divisor $\mathscr{K}_l$ on $\bar{\mathcal{X}} \times_T E_l T$ by considering the descent of the rank one reflexive sheaf $p_{\bar{\mathcal{X}}}^* \omega_{\bar{\mathcal{X}}/\mathbb{P}^1}$. 
Then the equivariant intersection $(\kappa^T_{\bar{\mathcal{X}}/\mathbb{P}^1} .\bar{\mathcal{L}}_T^{\cdot k})$ is the pushforward of the cycle $\mathscr{K}_l \frown c_1 (\mathscr{L}_l)^{\smile k}$ along the proper morphism $\bar{\mathcal{X}} \times_T E_l T \to B_l T$. 
The output is independent of the choice of $l \gg k$ (cf. \cite{EG1}). 

Finally, the equivariant intersection ${^\hbar (\kappa^T_{\bar{\mathcal{X}}/\mathbb{P}^1}. e^{\bar{\mathcal{L}}_T}; \xi)}$ is computed as the infinite sum 
\[ {^\hbar (\kappa^T_{\bar{\mathcal{X}}/\mathbb{P}^1}. e^{\bar{\mathcal{L}}_T}; \xi)} = \sum_{k=0}^\infty \frac{1}{k!} {^\hbar (\kappa^T_{\bar{\mathcal{X}}/\mathbb{P}^1}. \bar{\mathcal{L}}_T^{\cdot k}; \xi)}. \]

(c) Let us shift the lift of $\mathbb{G}_m$-action on $\mathcal{X}$ to $\mathcal{L}$ so that $\bar{\mathcal{L}}$ is ample. 
Then as shown in the proof of Proposition \ref{equivariant intersection and DH measure}, we have the following asymptotic expansion thanks to the equivariant Riemann--Roch theorem: 
\[ \int_{\mathfrak{t}^\vee} e^{- \langle x, \hbar \xi \rangle} \nu_k (x) = {^\hbar (e^{\bar{\mathcal{L}}_T}; \xi )} k^n - \frac{1}{2} {^\hbar (\kappa^T_{\bar{\mathcal{X}}}. e^{\bar{\mathcal{L}}_T}; \xi )} k^{n-1} + O (k^{n-2}) \]
with the measure 
\[ \nu_k = \sum_{m \in M} \dim H^0 (\bar{\mathcal{X}}, \bar{\mathcal{L}}^{\otimes k})_m . \delta_{k^{-1} m} \]
on $\mathbb{R}$, wehre $\delta_x$ denotes the Dirac measure at $x \in \mathbb{R}$. 
Therefore, by computing the decomposition $H^0 (\bar{\mathcal{X}}, \bar{\mathcal{L}}^{\otimes k}) = \bigoplus_{m \in M} H^0 (\bar{\mathcal{X}}, \bar{\mathcal{L}}^{\otimes k})_m$, we can compute equivariant intersections ${^\hbar (e^{\bar{\mathcal{L}}_T}; \xi )},  {^\hbar (\kappa^T_{\bar{\mathcal{X}}}. e^{\bar{\mathcal{L}}_T}; \xi )}$, in principle. 

\vspace{2mm}
In practice, it seems rather hard to compute $\mu$-Futaki invariant explicitly based on these methods. 
Not only that, to compute $\mu$-Futaki invariant, we must fix a vector $\xi \in \mathfrak{t}$ involved in the invariant, however, for almost all $\xi$ in $\mathfrak{t}$, the $\mu^\lambda_\xi$-Futaki invariant for product configurations does not vanish, and hence there exists some product configuration with negative $\mu^\lambda_\xi$-Futaki invariant. 
It is known by \cite[Theorem B]{Ino2} that for every $\lambda \in \mathbb{R}$ there exists some $\xi$ for which the $\mu$-Futaki invariant for product configuration vanishes (and such $\xi$ is unique for $\lambda \ll 0$), however, it is characterized in a transcendental way. 
In other words, we know \textit{there exists some invariant} by which we can detect the existence of $\mu^\lambda$-cscK metrics, however, the invariant is not explicitly given unless the candidate vector is explicitly written. 

Nevertheless, the notions of $\mu$-Futaki invariant and $\mu$K-stability are important in theoretic arguments. 
For instance, there is an application \cite{Ino5} of the results of this article to the study of moduli space of Fano manifolds with K\"ahler--Ricci solitons. 
See also the upcoming article \cite{Ino6}. 

Notably, modified K-stability ($\mu$K-stability for Fano manifolds with respect to equivariant special degenerations) can be effectively checked in some special case as in \cite{DaSz, Del}. 
In their arguments, the large symmetry of $(X, L)$ helps to reduce the problem to the $\mu$K-stability with respect to product configurations, for which the existence of the candidate $\xi$ is already known by Tian--Zhu's volume minimization argument \cite{TZ}. 
This volume minimization perspective will be pursued in \cite{Ino4}, where we reduce the $\mu$K-stability to the maximization problem of $\mu$-character for test configurations, which we call $\mu$-entropy. 
The $\mu$-entropy releases us from the problem of detecting the candidate vector $\xi$. 

\subsection{Main results: application of equivariant calculus}

Now we explain our main results: Theorem A and B. 

\subsubsection{$\mu$K-semistability of $\mu$-cscK manifolds}

First of all, we can compare our $\mu$-Futaki invariant with the following established Futaki invariants (Proposition \ref{comparison}): 
\begin{itemize}
\item Donaldson--Futaki invariant: $\cFut_0^\lambda (\mathcal{X}, \mathcal{L})$ is equivalent to the Donaldson--Futaki invariant $\mathrm{DF} (\mathcal{X}, \mathcal{L})$ for every test configuration $(\mathcal{X}, \mathcal{L})$ of a polarized scheme $(X, L)$. 

\item Modified Futaki invariant (cf. \cite{Xio, BW}): Suppose $(X, L)$ is a $\mathbb{Q}$-Fano variety with $L = -\lambda^{-1} K_X$ for $\lambda > 0$. 
Then $\cFut^{2\pi \lambda}_\xi (\mathcal{X}, \mathcal{L})$ is equivalent to the modified Futaki invariant $\mathrm{MFut}_\xi (\mathcal{X}, \mathcal{L})$ for every test configuration $(\mathcal{X}, \mathcal{L})$ of $(X, L)$ with $\mathcal{L} = -\lambda^{-1} K_{\mathcal{X}/\mathbb{A}^1}$. 
\end{itemize}

We also see that our definition of $\mu$-Futaki invariant is equivalent to Lahdili's definition of weighted Futaki invariant for smooth test configurations in our $\mu$-framework. 
Lahdili \cite{Lah} proved the weighted K-semistability of weighted cscK manifolds with respect to \textit{smooth test configurations} by establishing the boundedness of weighted Mabuchi functional and the slope formula along smooth rays. 
In terms of our equivariant intersection formula, we can understand the slope formula for $\mu$-Mabuchi functional along smooth rays is a simple consequence of equivariant Stokes theorem. 

We enhance his result to $\mu$K-semistability with respect to \textit{general test configurations}, which illustrates our definition of $\mu$-Futaki invariant for general test configuration is an appropriate one. 

\begin{thmA}
If a smooth polarized manifold $(X, L)$ admits a $\mu^\lambda_\xi$-cscK metric in the K\"ahler class $c_1 (L)$, then $(X, L)$ is $\mu^\lambda_\xi$K-semistable with respect to general test configurations. 
Namely, the $\mu^\lambda_\xi$-Futaki invariant is non-negative for every $T$-equivariant test configuration. 
\end{thmA}

We recall that for a Fano manifold $X$, a K\"ahler metric $\omega$ in the K\"ahler class $c_1 (X)$ is $\mu^{2\pi}_\xi$-cscK metric if and only if it is K\"ahler--Ricci soliton with respect to the vector field $\xi$: $\mathrm{Ric} (\omega) - L_{\xi^J} \omega = 2\pi \omega$. 
It is known by \cite{BW} and \cite{Xio} that a Fano manifold with a K\"ahler--Ricci soliton is modified K-semistable (even modified K-polystable) with respect to \textit{special degenerations}. 
We note that even in this case Theorem A (and the definition of $\mu$-Futaki invariant) is new as general singular test configurations are not special degenerations. 

To reduce Theorem A to Lahdili's result on weighted K-semistability for smooth test configurations, we establish basics on absolute equivariant intersection in section \ref{absolute equivariant intersection theory} and show the following basics on $\mu$-Futaki invariant. 
Since we have a projection formula for the equivariant intersection, we can show this in a similar way as in \cite{BHJ1} and \cite{DR} for the usual K-stability. 

\begin{thm}[Theorem \ref{fundamental lemma}: Intersection theoretic properties]
We have the following equivalences. 
\begin{enumerate}
\item A $T$-polarized normal variety $(X, L)$ is $\mu^\lambda_\xi$K-semistable (resp. $\mu^\lambda_\xi$K-polystable, $\mu^\lambda_\xi$K-stable) with respect to general test configurations iff it is $\mu^\lambda_\xi$K-semistable (resp. $\mu^\lambda_\xi$K-polystable, $\mu^\lambda_\xi$K-stable) with respect to normal test configurations. 

\item A $T$-polarized manifold $(X, L)$ is $\mu^\lambda_\xi$K-semistable with respect to general test configurations iff it is $\mu^\lambda_\xi$K-semistable with respect to smooth test configurations with reduced central fibres and ample $\mathcal{L}$. 
\end{enumerate}
\end{thm}

\subsubsection{Extremal limit $\lambda \to -\infty$}

By \cite[Theorem B]{Ino2}, for $\lambda \ll 0$ there exists a vector $\xi_\lambda$ for which $\cFut^\lambda_{\hbar. \xi_\lambda}$ vanishes for product configurations. 
Moreover, by \cite[Theorem D (1)]{Ino2}, we have $\lambda \xi_\lambda \to \xi_{\mathrm{ext}}$ for which the relative Futaki invariant $\cFut^\dagger_{\hbar. \xi_{\mathrm{ext}}}$ vanishes for product configurations. 

Using equivariant intersection, the relative Donaldson--Futaki invariant for test configuration (cf. \cite{Sze}) can be expressed as 
\begin{align*} 
\cFut^\dagger_{\hbar. \xi} (\mathcal{X}, \mathcal{L}) 
&= \frac{2\pi}{(L^{\cdot n})} \Big{(} \mathrm{DF} (\mathcal{X}, \mathcal{L}) - \frac{ {^\hbar (\bar{\mathcal{L}}^{\cdot n+2}; \xi)} }{(n+1)(n+2)}  + \frac{(\bar{\mathcal{L}}^{\cdot n+1})}{n+1} \frac{ {^\hbar (L^{\cdot n+1}; \xi)} }{(n+1) (L^{\cdot n})} \Big{)} 
\\
&= \frac{2\pi}{(L^{\cdot n})} \Big{(} \mathrm{DF} (\mathcal{X}, \mathcal{L}) + n! \int_{\mathfrak{t}^\vee \times \mathbb{R}} t (\langle x, \hbar \xi \rangle - \underline{\langle x, \hbar \xi \rangle}) \mathrm{DH}_{(\mathcal{X}, \mathcal{L})} \Big{)}, 
\end{align*}
where we put $\underline{\langle x, \hbar \xi \rangle} := \int_{\mathfrak{t}^\vee} \langle x, \hbar \xi \rangle \mathrm{DH}_{(X, L)} / \int_{\mathfrak{t}^\vee} \mathrm{DH}_{(X, L)}$. 
See Corollary \ref{equivariant intersection and DH measure, cor} and \ref{equivariant intersection on test configuration and DH measure, cor} for the last equality. 
It is shown by \cite{SS} that the relative Futaki invariant of test configurations is non-negative (relatively K-semistable) when $(X, L)$ admits an extremal metric. 
They moreover prove the relative K-polystability. 

\begin{thm}
For a general test configuration $(\mathcal{X}, \mathcal{L})$, we have 
\[ \lim_{\lambda \to -\infty} \cFut^\lambda_{\hbar. \xi_\lambda} (\mathcal{X}, \mathcal{L}) = \cFut^\dagger_{\hbar. \xi_{\mathrm{ext}}} (\mathcal{X}, \mathcal{L}). \]
\end{thm}

\begin{cor}
If $(X, L)$ is not relatively K-semistable, then the K\"ahler class $c_1 (L)$ does not admit any $\mu^\lambda$-cscK metrics for every $\lambda \ll 0$. 
\end{cor}

Equivalently, if there exists a sequence $\lambda_i \in \mathbb{R}$ diverging to $-\infty$ for which the K\"ahler class $c_1 (L)$ admits $\mu^{\lambda_i}$-cscK metrics, then $(X, L)$ is relatively K-semistable. 

The Yau--Tian--Donaldson conjecture for extremal metric predicts that there exists an extremal metric in $c_1 (L)$ when $(X, L)$ is relatively K-polystable. 
We proved in Theorem D (3) of \cite{Ino2} that if there exists an extremal metric then there exists a $\mu^\lambda$-cscK metric for $\lambda \ll 0$ in the same K\"ahler class. 
It is interesting to see if the weaker relative K-semistability assumption implies the existence of $\mu^\lambda$-cscK metrics for $\lambda \ll 0$ and also if the the limit of $\mu^\lambda$-cscK metrics gives a relative K-polystable degeneration for $(X, L)$. 

\subsubsection{$\mu$-character for family of polarized schemes}

We will observe an equivariant intersection expression of the minus log of the $\mu$-volume functional in section \ref{Equivariant cohomological interpretation}. 
This observation motivates us to introduce the following characteristic class. 

\begin{defin}[$\mu$-character]
\label{mu-character definition}
Let $\varpi: \mathcal{X} \to B$ be a proper flat family of schemes. 
Let $\mathcal{L} = \mathcal{L}_G \in NS_G (\mathcal{X}, \mathbb{R})$ be a relatively semiample and big class. 
Then we can define the following formal equivariant cohomology classes in $\hat{H}^{\mathrm{even}}_G (B, \mathbb{R})$: 
\begin{align}
\cmu_G (\mathcal{X}/B, \mathcal{L})
&:=2\pi (\kappa^G_{\mathcal{X}/B}. e^{\mathcal{L}_G})_B \cdot (e^{\mathcal{L}_G})_B^{-1} 
\\ 
\bm{\check{\sigma}}_G (\mathcal{X}/B, \mathcal{L}) 
&:= (\mathcal{L}_G. e^{\mathcal{L}_G})_B \cdot (e^{\mathcal{L}_G})_B^{-1} - \log (e^{\mathcal{L}_G})_B
\\
\cmu^\lambda_G (\mathcal{X}/B, \mathcal{L}) 
&:= \cmu_G (\mathcal{X}/B, \mathcal{L}) + \lambda \bm{\check{\sigma}}_G (\mathcal{X}/B, \mathcal{L}) . 
\end{align}
We call this the \textit{$\mu$-character} of the $G$-equivariant family $(\mathcal{X}/B, \mathcal{L})$. 
\end{defin}

The elements $(e^{\mathcal{L}_G})_B^{-1}, \log (e^{\mathcal{L}_G})_B$ are defined by (\ref{formal series of equivariant cohomology class}), for which we need the positivity of the degree $0$ part $(e^{\mathcal{L}})_B^{\langle 0 \rangle} = \frac{1}{n!} (\mathcal{L}_b^{\cdot n}) \in \mathbb{R}$. 
This is guaranteed by the assumption on $\mathcal{L}$. 

For $c_G \in H^2_G (B, \mathbb{R})$, we have 
\begin{align*}
(e^{\mathcal{L}_G + \varpi^* c_G})_B 
&= (e^{\mathcal{L}_G})_B \smile e^{c_G}, 
\\
((\mathcal{L}_G + \varpi^* c_G). e^{\mathcal{L}_G + \varpi^* c_G})_B 
&= ((e^{\mathcal{L}_G})_B + (\mathcal{L}_G. e^{\mathcal{L}_G})_B) \smile e^{c_G}, 
\\
(\kappa^G_{\mathcal{X}/B}. e^{\mathcal{L}_G + \varpi^* c_G})_B 
&= (\kappa^G_{\mathcal{X}/B}. e^{\mathcal{L}_G})_B \smile e^{c_G}, 
\end{align*}
so that the $\mu$-character is invariant under the replacement $\mathcal{L}_G \mapsto \mathcal{L}_G + \varpi^* c_G$. 

\subsubsection{Relative construction}

Now we explain the relative ($=$ family) construction of $\mu$-Futaki invariant. 
Throughout this paper, $B$ denotes a connected smooth variety with an algebraic action of an algebraic group $G$, for which we have the equivariant Poincare duality. 
We identify $H^0_G (B, \mathbb{R})$ with $\mathbb{R}$ in the natural way. 

\begin{thmB}
\label{relative Dmu}
Fix parameters $\lambda \in \mathbb{R}$ and $\xi \in \mathfrak{t}$. 
Then there exists an equivariant characteristic class 
\[ \D_{\hbar. \xi} \cmu^\lambda_{T \times G} (\mathcal{X}/B, \mathcal{L}) = \D_{\hbar. \xi} \cmu_{T \times G} (\mathcal{X}/B, \mathcal{L}) + \lambda \D_{\hbar. \xi} \bm{\check{\sigma}}_{T \times G} (\mathcal{X}/B, \mathcal{L}) \] 
valued in $NS_G (B, \mathbb{R})$ associated to a $T \times G$-equivariant family $(\mathcal{X}/B, \mathcal{L})$ of polarized schemes over a smooth $G$-variety $B$ with the trivial $T$-action which enjoys the following properties. 
\begin{enumerate}
\item Naturality: Suppose we have a morphism $\varphi: G' \to G$ of algebraic group and a $G'$-equivariant morphism $f: B' \to B$ from a smooth $G'$-variety $B'$. 
Let $(\mathcal{X}'/B', \mathcal{L}')$ be the $T \times G'$-equivaraint family given by the base change of $(\mathcal{X}/B, \mathcal{L})$ along $f$. 
Then we have 
\[ \D_{\hbar. \xi} \cmu^\lambda_{T \times G'} (\mathcal{X}'/B', \mathcal{L}') = f^* \varphi^\# \D_{\hbar. \xi} \cmu^\lambda_{T \times G} (\mathcal{X}/B, \mathcal{L}). \]

\item $\mu$-Futaki invariant: When the family $(\mathcal{X}/\mathbb{A}^1, \mathcal{L}) \circlearrowleft \mathbb{G}_m$ is a $T$-equivariant test configuration, then we have 
\[ \D_{\hbar. \xi} \cmu^\lambda_{T \times \mathbb{G}_m} (\mathcal{X}/\mathbb{A}^1, \mathcal{L}) = -\cFut^\lambda_{\hbar. \xi} (\mathcal{X}, \mathcal{L}). \bm{x} \in H^2_{\mathbb{G}_m} (\mathbb{A}^1, \mathbb{R}), \]
where $\bm{x}$ denotes the generator of $H^2_{\mathbb{G}_m} (\mathbb{A}^1, \mathbb{Z})$ corresponding to $c_1 (\mathcal{O} (-1)) \in H^2 (\mathbb{C}P^\infty, \mathbb{Z})$ under the identification $H^2_{\mathbb{G}_m} (\mathbb{A}^1, \mathbb{Z}) = H^2 (\mathbb{C}P^\infty, \mathbb{Z})$. 

\item CM line bundle: For $\xi =0$ ($\mathfrak{t} = 0$), we have 
\[ \D_{\hbar. 0} \cmu^\lambda_G (\mathcal{X}/B, \mathcal{L}) = \frac{2 \pi}{(L^{\cdot n})} c_{G, 1} (\mathrm{CM} (\mathcal{X}/B, \mathcal{L})) \]
for the CM line bundle $\mathrm{CM} (\mathcal{X}/B, \mathcal{L})$, independent of $\lambda \in \mathbb{R}$. (cf. \cite{PT})

\item Regularity: $\D_{\hbar. \xi} \cmu^\lambda_{T \times G} (\mathcal{X}/B, \mathcal{L})$ gives the following real analytic map: 
\[ \mathfrak{t} \to NS_G (B, \mathbb{R}): \xi \mapsto \D_{\hbar. \xi} \cmu^\lambda_{T \times G} (\mathcal{X}/B, \mathcal{L}). \]  
\end{enumerate}
\end{thmB}

We sketch the idea of the construction in section \ref{Equivariant cohomological interpretation}. 
By our construction of $\D_{\hbar. \xi} \cmu^\lambda_{T \times G} (\mathcal{X}/B, \mathcal{L})$, the base change property (1) is derived from the base change property of the equivariant intersection $(\kappa_{\mathcal{X}/B}. e^{\mathcal{L}})_B$. 
While the base change property for relative canonical class $K_{\mathcal{X}/B}$ is in general recognized as being involved with a mildness of singularities of families, the base change property for the equivariant intersection $(\kappa_{\mathcal{X}/B}. e^{\mathcal{L}})_B$ is just a consequence of the equivariant Grothendieck--Riemann--Roch theorem (cf. \cite{EG2}), which works for general schemes.

\subsubsection*{Acknowledgement}

This article was mostly written when I visited Kyoto University from Apr. 2019 to Mar. 2020. 
I am grateful to my second advisor Yuji Odaka for his hospitality, advices and comments on this article. 
I thank friends in Kyoto University for their hospitality. 
I am also grateful to my advisor Shigeharu Takayama for his constant support. 
This work is supported by JSPS KAKENHI Grant Number 18J22154, MEXT, Japan.

\section{Preliminaries for equivariant calculus}

\subsection{Notations and sign conventions}
\label{sign convention}

We build our equivariant calculus upon equivariant cohomology. 
It is quite sensitive to sign conventions on group action and equivariant cohomology. 
Here we clarify all these sign conventions in order to work within a consistent rule. 
See also Appendix for the sign conventions in the equivariant cohomologies $H_G (X, \mathbb{R}), {^\hbar H_{\mathrm{dR}, K}} (X)$ and the equivariant Chern class $c_G (E)$. 

\subsubsection*{$\bullet$ Fundamental vector field and right group action}

For a Lie group $G$, its Lie algebra $\mathfrak{g}$ consists of the left invariant vector fields on $G$. 
The exponential map $\exp: \mathfrak{g} \to G$ is given by the time one map of the one parameter subgroup associated to $\xi \in \mathfrak{g}$: $\frac{d}{dt}|_{t=0} (g \cdot \exp t \xi) = \xi_g$. 
The adjoint right action $\mathfrak{g} \circlearrowleft G$ is characterized by $\exp t (\xi. g) = g^{-1} \cdot \exp t \xi \cdot g$ and the coadjoint right action $\mathfrak{g}^\vee \circlearrowleft G$ is given by $\langle \mu. g, \xi \rangle = \langle \mu, \xi. g^{-1} \rangle$. 
We have $[\xi, \zeta] = L_\xi \zeta = \frac{d}{dt}|_{t=0} \xi. \exp t \zeta$. 

Unless specifically mentioned, we consider only right group actions. 
If we have a right action, then we can induce the left action by putting $g. x := x. g^{-1}$, which gives a functor from the category of right actions to left actions. 
When the group $T$ is abelian, we can also think of a right action as a left action by putting $t. x = x. t$, but we must beware of this as it may reverse various signs on equivariant cohomology. 
Indeed, this does not extend to a functor for non-abelian actions. 

If we have a right $G$-action on a vector bundle $E$ over a $G$-space $X$, the induced right $G$-action on the dual vector bundle $E^\vee$ is given by $\langle \mu. g, \xi \rangle = \langle \mu, \xi. g^{-1} \rangle$. 
The right action on the space of sections $\Gamma (X, E)$ of a $G$-equivariant vector bundle $E$ is given by $(s. g ) (x):= s (x. g^{-1}). g$. 

Let $G$ be a Lie group and $M$ be a manifold with a smooth $G$-action. 
The induced right action on the space of vector fields $\mathfrak{X} (M)$ is given by $V. g = (R_g)_* V$ for $R_g (x) = x.g$. 
Similarly, the right action on $\Omega^* (M)$ is given by $\varphi. g = R_{g^{-1}}^* \varphi$. 
The fundamental vector field $\xi^\#$ on $M$ is given by the differential of the flow $\mathrm{Exp}^t_\xi = R_{\exp t \xi}$. 
The map $\mathfrak{g} \to \mathfrak{X} (M): \xi \mapsto \xi^\#$ is $G$-equivariant as $\mathrm{Exp}^t_{\xi. g} (x) = \mathrm{Exp}^t_\xi (x. g^{-1}). g$ and hence is compatible with the Lie brackets $[\xi, \zeta]^\# = [\xi^\#, \zeta^\#]$ thanks to the formula $[\xi^\#, \zeta^\#] = \frac{d}{dt}|_{t=0} (\mathrm{Exp}^t_\zeta)_* \xi^\# = \frac{d}{dt}|_{t=0} \xi^\#. \exp t \zeta$. 
Note if we consider a left action on $M$, the fundamental vector field must be reversed as ${^\# \xi} = (d/dt)|_{t=0} (\exp (-t \xi). x)$: for $\xi_\# = - {^\# \xi} = (d/dt)|_{t=0} (\exp t \xi. x)$, we have $[\xi_\#, \zeta_\#] = - [\xi, \zeta]_\#$. 

\subsubsection*{$\bullet$ Hamiltonian vector field and momemt map}

Let $(M, \omega)$ be a symplectic manifold. 
For $f \in C^\infty (M)$, we define the associated Hamiltonian vector field $X_f$ by $i_{X_f} \omega = -df$, rather than $i_{X_f} \omega = df$. 
For $f, g \in C^\infty (M)$, the Poisson bracket $\{ f, g \} := \omega (X_f, X_g)$ is compatible with the Lie bracket: $[X_f, X_g] = X_{\{ f, g \}}$ under this sign convention. 

For a $G$-invariant $\omega$, we call a map $\mu: M \to \mathfrak{g}^\vee$ a moment map if it is $G$-equivariant and satisfies $i_{\xi^\#} \omega = - d\mu_\xi$ for every $\xi \in \mathfrak{g}$ where we put $\mu_\xi := \langle \mu, \xi \rangle$. 

\subsubsection*{$\bullet$ Chern--Weil isomorphism}

Let $G$ be an almost connected locally compact group and $K$ be a maximal compact group. 
There is a natural isomorphism of functors on the category of manifolds with smooth $G$-action
\begin{equation}
{^\hbar \Phi}: H^*_G (-, \mathbb{R}) \xrightarrow{\sim} {^\hbar H^*_{\mathrm{dR}, K}}. 
\end{equation}
The isomorphism is given by the composition of three isomorphisms, one of which is reversed from the original chain level construction. 
The isomorphisms are related by $(\hbar'/\hbar)_* \circ {^\hbar \Phi} = {^{\hbar'} \Phi}$. 

For instance, the equivariant cohomology class 
\[ {^\hbar \Phi^{-1}} [\omega + \hbar \mu] \in H^2_G (X, \mathbb{R}) \] 
is independent of $\hbar$: ${^{\hbar'} \Phi^{-1}} [\omega + \hbar' \mu] = {^{\hbar'} \Phi^{-1}} (\hbar'/\hbar)_* [\omega + \hbar \mu] = {^\hbar \Phi^{-1}} [\omega + \hbar \mu]$. 
This example illustrates the map $(\hbar'/\hbar)_*$ is not the simple scalar multiplication $[\omega + \hbar] \mu \mapsto \hbar'/\hbar \cdot [\omega + \hbar \mu]$. 

Substituting $M = \mathrm{pt}$, we obtain the Chern--Weil isomorphism 
\[ {^\hbar \Phi^{-1} }: (S^* \mathfrak{k}^\vee)^K \xrightarrow{\sim} H^*_G (\mathrm{pt}, \mathbb{R}). \]
We must clarify the sign of this isomorphism, which depends on $\hbar \in \mathbb{R}^\times$ involved in the construction of ${^\hbar H^*_{\mathrm{dR}, K}}$. 
We will compare equivariant Chern classes for this sake. 

As explained in Appendix, there are dual constructions of these functors: equivariant locally finite homology $H^{\mathrm{lf}, G}_* (-, \mathbb{R})$, equivariant current homology ${^\hbar H^{\mathrm{cur}, K}_*}$ and a natural isomorphism 
\[ {^\hbar \Phi}: H^{\mathrm{lf}, G}_* (-, \mathbb{R}) \xrightarrow{\sim} {^\hbar H^{\mathrm{cur}, K}_*}. \]
Similarly, we have a natural identification ${^\hbar H^{\mathrm{cur}, K}_p} (\mathrm{pt}) = (S^{-p} \mathfrak{k}^\vee)^K$ by the construction, where we put $S^{-p} \mathfrak{k}^\vee = 0$ for $p > 0$. 

We have $H_{\mathbb{G}_m}^* (\mathrm{pt}, \mathbb{R}) = H^* (\mathbb{C}P^\infty, \mathbb{R})$ by the construction. 
The latter can be identified with the polynomial ring $\mathbb{R} [\bm{x}]$ by assigning $\bm{x}$ to the first Chern class $c_1 (\mathcal{O} (-1))$ of the tautological line bundle $\mathcal{O} (-1)$ over $\mathbb{C}P^\infty$. 
(In our convention, the tautological line bundle $\mathcal{O} (-1)$ is associated to the weight $-1$ line bundle $\mathbb{C}_{-1}$ over the point. ) 
On the other hand, we can identify ${^\hbar H^*_{\mathrm{dR}, U (1)}} (\mathrm{pt}) = S^* \mathfrak{u} (1)^\vee$ with the polynomial ring $\mathbb{R} [\eta^\vee]$ under the identification $\mathfrak{u} (1)^\vee = \mathbb{R}. \eta^\vee$. 
Then the map ${^\hbar \Phi}: \mathbb{R} [\bm{x}] \to \mathbb{R} [\eta^\vee]$ is written explicitly as 
\begin{equation} 
\label{Chern--Weil sign}
{^\hbar \Phi} (\bm{x}) = \hbar \eta^\vee. 
\end{equation}
To see this, we compare two constructions of equivariant Chern classes. 

\subsubsection*{$\bullet$ Equivariant Chern class}

We recall that the equivariant Chern class $c_G (E)$ in the equivariant singular cohomology $H^*_G (X, \mathbb{Z})$ is defined to be the usual Chern class of the associated descent bundle $E_G = EG \times_G E$ over $X_G = EG \times_G X$. 
Here we note $G$ acts on $E$ from the right. 
This affects the sign convention on the equivariant Chern class, which we will see in example. 

On the other hand, the equivariant first Chern class ${^\hbar c_{K, 1}} (E) = {^\hbar \Phi} (c_{K, 1} (E))$ in the Cartan model ${^\hbar H^*_{\mathrm{dR}, K}} (M)$ is represented by an equivariant form constructed by using a $K$-invariant connection form $\vartheta \in \Omega^1 (P, \mathfrak{u} (1))$ on the principal $U (1)$-bundle $p: P \to M$ associated to $L$. 
A $K$-equivariant $2$-form $\omega + \hbar \mu$ satisfying $p^* \omega. \eta = -d \vartheta$ and $p^* \mu_\xi. \eta = -i_{\xi^\#} \vartheta$ for every $\xi \in \mathfrak{k}$ represents ${^\hbar c_{K, 1}} (E)$. 
We note the right $K$-action on $P = L^\times/\mathbb{R}_+ = (L \setminus M) / \langle v \sim \lambda v ~|~ \lambda > 0 \rangle$ is given by $[v]. g = [v. g]$ induced from the right $K$-action on $L$, which should not be confused with the $U (1)$-action $[v]. e^{\sqrt{-1} \theta} = [e^{\sqrt{-1} \theta} \cdot v]$ especially when the group $K$ is $U (1)$. 

This construction is justified by the following uniqueness on equivariant Chern class. 
We give a proof in section \ref{section: Cartan model of equivariant cohomology and locally finite homology} for the readers convenience. 

\begin{lem}
\label{equivariant Chern uniqueness}
The equivariant first Chern class for line bundles over smooth manifolds is characterized by the following properties. 
Namely, if a characteristic class $L_G \mapsto \tilde{c} (L_G) \in H^2_G (M, \mathbb{R})$ satisfies the following, then it coincides with $c_{G, 1} (L)$. 
\begin{enumerate}
\item For every $G$-equivariant smooth map $f: M \to N$, $\tilde{c} (f^* L_G) = f^* \tilde{c} (L_G)$. 

\item For every group homomorphism $\varphi: H \to G$, $\tilde{c} (L_H) = \varphi^\# \tilde{c} (L_G)$. 

\item For the trivial group $G = \{ 1 \}$, $\tilde{c} (L_G) = c_1 (L)$. 
\end{enumerate}
\end{lem}

We easily check that ${^\hbar \Phi^{-1}} [\omega + \hbar \mu] \in H^2_K (M, \mathbb{R})$ constructed as above satisfies the first two conditions. 
The last property, $[\omega] = c_1 (L)$ for $\omega$ satisfying $\varpi^* \omega. \eta = -d \vartheta$, is well-known as Chern--Weil construction. 
We can check the sign $-d\vartheta$ is appropriate by the following example. 

\begin{eg}[Chern form]
\label{Chern form}
Let $L = \mathcal{O} (-1)$ be the tautological line bundle over $\mathbb{P}^1$. 
The associated principal $U(1)$-bundle $P$ is identified with 
\[ S^3 = \{ (z, w) \in \mathbb{C}^2 ~|~ |z|^2 + |w|^2 = 1 \} \] 
endowed with the $U (1)$-action $(z, w). e^{\sqrt{-1} \theta} = (e^{\sqrt{-1} \theta} \cdot z, e^{\sqrt{-1} \theta} \cdot w)$. 
For $(z, w) = (e^{\rho_1 + \sqrt{-1} \theta_1}, e^{\rho_2 + \sqrt{-1} \theta_2})$, the fundamental vector field $\eta^\#$ on $P$ is given by $2\pi (\partial/\partial \theta_1 + \partial/\partial \theta_2)$. 

We can easily check 
\[ \vartheta = (2\pi)^{-1} (e^{2\rho_1} d\theta_1 + e^{2\rho_2} d \theta_2). \eta = (2\pi)^{-1} \mathrm{Re} (\sqrt{-1} z d \bar{z} + \sqrt{-1} w d \bar{w}). \eta \] 
gives a connection on $P$. 
We then get $d\vartheta = (2\pi)^{-1} (\sqrt{-1} d z \wedge d \bar{z} + \sqrt{-1} dw \wedge d \bar{w}). \eta$. 
It follows that  
\[ \omega = - \frac{\sqrt{-1}}{2\pi} (d z \wedge d \bar{z} + dw \wedge d \bar{w}) = - \frac{\sqrt{-1}}{2\pi} \frac{d\zeta \wedge d \bar{\zeta}}{(1 + |\zeta|^2)^2}  \]
for the coordinate $\zeta = z/w$ on $\mathbb{A}^1 \subset \mathbb{P}^1$. 
Therefore, we conclude $[\omega] = [-d \vartheta/\eta] = c_1 (\mathcal{O} (-1))$ as desired. 
\end{eg}

The following example illustrates a relation of the weight of $\mathbb{G}_m$-action on the fibre $L_x$ of a fixed point $x$ and the value of the moment map $\mu (x)$ normalized by $[\omega + \hbar \mu] = {^\hbar c_{\mathbb{G}_m, 1}} (L)$. 

\begin{eg}[Weight and moment image]
\label{weight and moment value}
Consider the $SU (2)$-action on $\mathbb{P}^1$ and $L = \mathcal{O} (-1) = \mathrm{Bl}_o \mathbb{C}^2$ derived from the $SU (2)$-representation $\mathbb{C}^2$ and a $U (1)$-action via the group homomorphism $\varphi: U (1) \to SU (2): e^{\sqrt{-1} \theta} \mapsto \mathrm{diag} (e^{p \sqrt{-1} \theta}, e^{q\sqrt{-1} \theta})$ for $p, q \in \mathbb{Z}$. 

The connection $\vartheta$ in Example \ref{Chern form} is $SU (2)$-invariant. 
Since the fundamental vector field $\eta_\varphi^\#$ on the principal $U(1)$-bundle $P$ is $2\pi (p \partial/\partial \theta_1 + q \partial/\partial \theta_2)$, the moment map $\mu: \mathbb{P}^1 \to \mathfrak{u} (1)_\varphi^\vee$ is given by $\varpi^* \mu_\eta = - i_{\eta_\varphi^\#} \vartheta/\eta = - (p|z|^2 +q |w|^2)$ on $P$. 

Now the point $(0:1) \in \mathbb{P}^1$ is $U (1)_\varphi$-invariant and $U (1)_\varphi$ acts on its fibre $L_{(0:1)} = \mathbb{C}$ by $z. t = t^q z$. 
On the other hand, we have $\mu (0:1) = -q \eta^\vee$ by the above observation. 

In particular, we get 
\[ {^\hbar c_{U (1), 1}} (\mathbb{C}_m) = - \hbar m \eta^\vee \] 
for the $U (1)$-equivariant line bundle $\mathbb{C}_m = \mathbb{C}$ over the point whose action is given by the weight $q$: $z. t = t^m z$. 
\end{eg}

Now we compute $c_{\mathbb{G}_m, 1} (\mathbb{C}_m) \in H^2_{\mathbb{G}_m} (\mathrm{pt}, \mathbb{R})$ and show the explicit formula (\ref{Chern--Weil sign}) of Chern--Weil isomorphism. 

\begin{eg}[Sign of Chern--Weil isomorphism]
Consider a torus $T = (\mathbb{G}_m)^{\times k}$. 
The principal $T$-bundle $(\mathbb{C}^\infty \setminus \{ 0 \})^{\times k} \to  (\mathbb{C}P^\infty)^{\times k}$ serve as the classifying space of $T$, so we may identify $\mathrm{pt}_T = ET \times_T \mathrm{pt}$ with $(\mathbb{C}P^\infty)^{\times k}$. 
Let $p_i: \mathrm{pt}_T \to \mathbb{C}P^\infty$ be the projection to the $i$-th component, then we can identify $H^*_T (\mathrm{pt}, \mathbb{R})$ with $\mathbb{R} [\bm{x}_1, \ldots, \bm{x}_k]$ by assigning $p_i^* c_1 (\mathcal{O} (-1))$ to $\bm{x}_i$. 
On the other hand, we have ${^\hbar H^*_{\mathrm{dR}, T}} (\mathrm{pt}) = \mathbb{R} [\eta^\vee_1, \ldots, \eta^\vee_k]$ from the construction. 
For a character $\chi = t_1^{m_1} \dotsb t_k^{m_k}: T \to \mathbb{G}_m$, we would compare the equivariant first Chern classes of the line bundle $\mathbb{C}_\chi = \mathbb{C}$ over the point $\mathrm{pt}$ endowed with the $T$-action $w. t = \chi (t) w$. 

We check the line bundle $L_{\chi, T}$ over $\mathrm{pt}_T$ associated to $\mathbb{C}_\chi$ is written explicitly as $\mathcal{O} (\chi) := p_1^* \mathcal{O} (m_1) \otimes \dots \otimes p_k^* \mathcal{O} (m_k)$. 
Indeed, by definition, we have 
\[ L_{\chi, T} := (\mathbb{C}^\infty \setminus \{ 0 \})^{\times k} \times_{\mathbb{G}_m} \mathbb{C} = (\mathbb{C}^\infty \setminus \{ 0 \})^{\times k} \times \mathbb{C}/\sim \]
with the equivalence relation 
\[ ((z^1_l)_{l=1}^\infty, \ldots, (z^k_l)_{l=1}^\infty, w) \sim ((t_1 z^1_l)_{l=1}^\infty, \ldots, (t_l z^k_l)_{l=1}^\infty, t_1^{m_1} \dotsb t_k^{m_k} w). \]
Then the following assignment gives a well-defined isomorphism $L_{t_1^{-1}, T}^{-m_1} \otimes \dotsb \otimes L_{t_k^{-1}, T}^{-m_k} \xrightarrow{\sim} L_{\chi, T}$ of line bundles: 
\[ ([((z^i_l)_{l=1}^\infty)_{i=1}^k, w_1]_1, \ldots,  [((z^i_l)_{l=1}^\infty)_{i=1}^k, w_k]_k) \mapsto [((z^i_l)_{l=1}^\infty)_{i=1}^k, w_1^{-m_1} \dotsb w_k^{-m_k}]. \]
On the other hand, it is easily seen that $L_{t_i^{-1}, T} = p_i^* L_{t^{-1}, \mathbb{G}_m}$, where the $T$-action on the latter bundle is given via $t_i^{-1}: T \to \mathbb{G}_m$. 
Thus it suffices to check $L_{t^{-1}, \mathbb{G}_m} = \mathcal{O} (-1)$ over $\mathbb{C}P^\infty$. 
Recall the tautological line bundle $\mathcal{O} (-1)$ is the subbundle $\{ ([z_l]_{l=1}^\infty, (\tau z_l)_{l=1}^\infty) \in \mathbb{C}P^\infty \times \mathbb{C}^\infty ~|~ \tau \in \mathbb{C} \}$ by definition. 
Since the following assignment gives a well-defined isomorphism $L_{t^{-1}, \mathbb{G}_m} \xrightarrow{\sim} \mathcal{O} (-1)$: 
\[ [(z_l)_{l=1}^\infty, w] \mapsto ([z_l]_{l=1}^\infty, (w z_l)_{l=1}^\infty) \]
we obtain the claim. 
Therefore, we have $c_{T, 1} (\mathbb{C}_\chi) = -(m_1 \bm{x}_1 + \dotsb + m_k \bm{x}_k)$. 

On the other hand, $\vartheta = (2\pi)^{-1} d\theta. \eta$ gives a $T$-invariant connection on $P = U (1)$. 
Since $\eta_i^\# = \chi_* \eta_i = m_i \eta$, we have $\mu_{\eta_i} = -i_{\eta_i^\#} \vartheta /\eta = - m_i$. 
Thus we get ${^\hbar c_{K, 1}} (\mathbb{C}_\chi) = - \hbar (m_1 \eta_1^\vee + \dotsb + m_k \eta_k^\vee) = - \hbar \chi$. 
\end{eg}

We note whether group acts from the left or right is important in the above computation. 
Recall for a $G$-equivariant vector bundle $E$ over $X$, the descent bundle 
\[ E_G := EG \times_G E = EG \times E/\sim \]
is constructed with the equivalence relation $\sim$: $(e, x) \sim (e. g, x. g)$. 
If we consider a left action on $E$ and $X$ while considering the right action on $EG$, then the equivalence relation $\stackrel{\mathrm{left}}{\sim}$ for $E_G^{\mathrm{left}} = EG \times E/\stackrel{\mathrm{left}}{\sim}$ must be defined as $(e, x) \stackrel{\mathrm{left}}{\sim} (e. g, g^{-1}. x)$, otherwise it does not give an equivalence relation for non-abelian $G$. 
Then if we consider the line bundle $\mathbb{C}^\chi$ endowed with the \textit{left $T$-action} $t. z = \chi (t) z$, the adjoint bundle $L_T^\chi = (\mathbb{C}^\infty \setminus \{ 0 \})^{\times k} \times \mathbb{C}/\stackrel{l}{\sim}$ is reversed as 
\[ ((z^1_l)_{l=1}^\infty, \ldots, (z^k_l)_{l=1}^\infty, w) \stackrel{l}{\sim} ((t_1 z^1_l)_{l=1}^\infty, \ldots, (t_l z^k_l)_{l=1}^\infty, t_1^{-m_1} \dotsb t_k^{-m_k} w). \]
Thus we have $L_T^\chi = L_{T, -\chi} = L_{T, \chi}^{-1}$. 

\begin{eg}[Equivariant Chern class of canonical bundle]
\label{Equivariant Chern class of canonical bundle}
Let $\omega$ be a $K$-invariant K\"ahler metric on a K\"ahler manifold $X$ and $\mu: X \to \mathfrak{k}^\vee$ be a moment map with resepct to $\omega$. 
Then we can check 
\[ {^\hbar c_{K, 1}} (X) = - {^\hbar c_{K, 1}} (K_X) = \frac{1}{2\pi} [\Ric (\omega) + \hbar \bar{\Box} \mu] \]
by computing the Chern connection $\vartheta$ on the principal bundle of $\det TX$ induced by the metric $\det \omega$. 
Indeed, since we have $i_{\xi^\#} \Ric (\omega) = d i_{\xi^\#} d^c \log \det \omega = - d \bar{\Box} \mu_\xi$, $\bar{\Box} \mu$ is a moment map for $\Ric (\omega)$. 
One can also check the normalization of the moment map by examples. 
\end{eg}

\subsubsection*{$\bullet$ Localization formula}

\begin{eg}[Localization formula]
\label{localization formula}
The $\mathbb{G}_m$-action on $\mathbb{P}^1$ has two fixed points: $i_0 (\mathrm{pt}) = 0 \in \mathbb{A} \subset \mathbb{P}^1$ and $i_{0_-} (\mathrm{pt}) = 0_- \in \mathbb{A}_- \subset \mathbb{P}^1$. 
For every $u \in H^2_{\mathbb{G}_m} (\mathbb{P}^1, \mathbb{R})$, we have the following localization formula: 
\begin{equation} 
\int_{\mathbb{P}^1} u. \bm{x} = i_{0_-}^* u - i_0^* u \in H^2_{\mathbb{G}_m} (\mathrm{pt}, \mathbb{R}) = \mathbb{R}. \bm{x}. 
\end{equation}

The localization formula is a paraphrase of Stokes theorem in this setup. 
We explain this to check the sign in the localization formula, which is crucial for our equivariant calculus. 

Pick an equivariant $2$-form $\omega + \hbar \mu$ in the cohomology class ${^\hbar u}$. 
Using the polar coordinate $e^{\rho + \sqrt{-1} \theta}$ of $\mathbb{A}^1 \setminus \{ 0 \} \subset \mathbb{P}^1$, we can write $\omega|_{\mathbb{A}^1 \setminus \{ 0 \}} = f (\rho) d\rho \wedge d\theta = d (\int_{-\infty}^\rho f (t) dt. d \theta)$ for some $f$. 
With this expression, we have $d\mu_\eta = - i_{\eta^\#} \omega = 2\pi f (\rho) d \rho$ as $\eta = 2 \pi \frac{\partial}{\partial \theta}$. 
Then we can express the function $\mu_\eta$ as 
\[ \mu_\eta (z) = 2 \pi \int_{-\infty}^{\log |z|} f (\rho) d\rho + \mu_\eta (0). \]
Applying Stokes theorem, we compute
\begin{align*} 
\int_{\mathbb{P}^1} \omega 
&= \lim_{\rho \to \infty} \left( \int_{|z| = e^\rho} \int_{-\infty}^\rho f (t) dt. d \theta - \int_{|z| = e^{-\rho}} \int_{-\infty}^{-\rho} f (t) dt. d \theta \right)
\\
&= 2 \pi \int_{-\infty}^\infty f (t) dt
\\
& = \mu_\eta (0_-) - \mu_\eta (0). 
\end{align*}
\end{eg}

We can check the sign by Example \ref{weight and moment value}: $\int_{\mathbb{P}^1} (\omega + \hbar \mu) = \int_{\mathbb{P}^1} \omega = -1$, $i_0^* (\omega + \hbar \mu) = \hbar \eta^\vee$, $i_{0_-}^* (\omega +\hbar \mu) = 0$. 

See \cite[Appendix C. 7]{GGK} for a more general localization formula. 

\subsubsection*{$\bullet$ Duistermaat--Heckman measure}

Let $X$ be a proper scheme with a $T$-action and $L_T$ be a semiample and big $T$-equivariant line bundle over $X$. 
We define the \textit{Duistermaat--Heckman measure} of $(X, L_T)$ on $\mathfrak{t}^\vee$ by 
\begin{equation}
\mathrm{DH}_{(X, L_T)} := \lim_{k \to \infty} k^{-n} \sum_{m \in M} \dim H^0 (X, L^{\otimes k})_{m} . \delta_{k^{-1} m}. 
\end{equation}
When $X$ is smooth, we have $(-\mu)_* (\omega^n/n!) = \mathrm{DH}_{(X, L_T)}$ for the moment map $\mu: X \to \mathfrak{t}^\vee$ with ${^\hbar \Phi^{-1}} [\omega + \hbar \mu] = L_T$. 
Note the sign of the moment map must be reversed as we witness in Proposition \ref{equivariant intersection and DH measure}. 

\subsection{Key observations for intersection theoretic expression}

\subsubsection{On the intersection formula of Donaldson--Futaki invariant}
\label{section: Odaka--Wang intersection formula}

Before studying $\mu$K-stability, we observe a simple equivariant cohomological interpretation of Odaka--Wang's intersection formula. 
The Futaki invariant for general test configurations is introduced by Donaldson \cite{Don} as a generalization of differential geometric Futaki invariant introduced in \cite{Fut}, using coefficients in Hilbert and weight polynomial given by equivariant Riemann--Roch theorem. 
Odaka \cite{Oda} and Wang \cite{Wang} find an intersection theoretic expression of Donaldson--Futaki invariant via Donaldson's definition. 

Here we derive Odaka--Wang's intersection formula for product configurations directly from the differential Futaki invariant, using the localization formula in Example \ref{localization formula}. 
This observation helps to understand the equivariant cohomological expression of the equivariant first Chern class of CM line bundle for smooth family and provides us an inspiration on the construction of the characteristic class $\D_{\hbar. \xi} \cmu^\lambda (\mathcal{X}/B, \mathcal{L})$ for general family. 

Let $X$ be a compact K\"ahler manifold with a holomorphic $\mathbb{G}_m$-action and $L_{\mathbb{G}_m} \in H^2_{\mathbb{G}_m} (X, \mathbb{R})$ be a lift of a K\"ahler class $L \in H^2 (X, \mathbb{R})$. 
The equivariant deRham class ${^\hbar L_{U (1)}} = {^\hbar \Phi} (L_{\mathbb{G}_m}) \in {^\hbar H^2_{\mathrm{dR}, U (1)}} (X)$ is represented by $\omega + \hbar \mu$ with some K\"ahler metric $\omega$ and a moment map $\mu$. 

As in section \ref{section: Test configuration}, we consider the product test configuration $(X \times \mathbb{A}^1, p_X^* L_{\mathbb{G}_m})$. 
Let $(\bar{\mathcal{X}}, \bar{\mathcal{L}})$ denote its compactification. 
Let $i_0, i_{0_-}: \mathrm{pt} \to \mathbb{P}^1$ be the maps $i_0 (\mathrm{pt}) = 0, i_{0_-} (\mathrm{pt}) = 0_-$ and $j_0, j_{0_-}: X \to \bar{\mathcal{X}}$ be the natural isomorphisms to the fibres over $0, 0_-$, respectively. 

Recall the Futaki invariant is defined as 
\[ {^\hbar \cFut} (\zeta) := \int_X \hat{s} (\omega) \hbar \mu_\zeta \omega^n \Big{/} \int_X \omega^n. \]
Many literatures choose the convention $\hbar = -2$. 
Under the identification ${^\hbar H^2_{\mathrm{dR}, U (1)}} (\mathrm{pt}) = \mathfrak{u} (1)^\vee = \mathbb{R}. \eta^\vee$, we regard ${^\hbar \cFut} = {^\hbar \cFut} (\eta). \eta^\vee$ as an element of ${^\hbar H^2_{\mathrm{dR}, U (1)}} (\mathrm{pt})$. 
Now we see that the following. 

\begin{lem}[Intersection formula]
In the setting as above, the equivariant cohomology class ${^\hbar \cFut} \in {^\hbar H^2_{\mathrm{dR}, U (1)}} (\mathrm{pt})$ coincides with the pullback of the following equivariant cohomology class along $i_0^*: {^\hbar H^2_{\mathrm{dR}, U (1)}} (\mathbb{P}^1) \to {^\hbar H^2_{\mathrm{dR}, U (1)}} (\mathrm{pt})$: 
\[ -\frac{2 \pi}{(L^{\cdot n})} \Big{(} {^\hbar (} K_{\bar{\mathcal{X}}/\mathbb{P}^1}^{ \mathbb{G}_m }. \bar{\mathcal{L}}_{ \mathbb{G}_m }^{\cdot n} )_{\mathbb{P}^1} + \frac{n}{n+1} \frac{(-K_X. L^{\cdot (n-1)})}{(L^{\cdot n})} {^\hbar (} \bar{\mathcal{L}}_{ \mathbb{G}_m }^{\cdot (n+1)})_{\mathbb{P}^1} \Big{)} \in {^\hbar H^2_{\mathrm{dR}, U (1) }} (\mathbb{P}^1). \] 

As a consequence, we obtain the following well-known intersection formula: 
\[ {^\hbar \cFut} (\eta) = \hbar \cdot \frac{2 \pi}{(L^{\cdot n})} \Big{(} (K_{\bar{\mathcal{X}}/\mathbb{P}^1} . \bar{\mathcal{L}}^{\cdot n}) + \frac{n}{n+1} \frac{(- K_X . L^{\cdot (n-1)})}{(L^{\cdot n})} (\bar{\mathcal{L}}^{\cdot (n+1)}) \Big{)}. \]
\end{lem}

\begin{proof}
By the construction of $\bar{\mathcal{L}}$ and by the Grothendieck--Riemann--Roch (or by explicit computations), we have 
\begin{align*} 
j_0^* {^\hbar c_{U (1), 1}} (\bar{\mathcal{L}}) 
&= {^\hbar L_{U (1)}} = [\omega + \hbar \mu], 
\\
j_0^* {^\hbar c_{U (1), 1}} (K_{\bar{\mathcal{X}}/\mathbb{P}^1}) &= {^\hbar c_{U (1), 1}} (K_X) = -\frac{1}{2\pi} [\Ric_\omega + \hbar \bar{\Box} \mu]. 
\end{align*}
Then we compute 
\begin{align*}
{^\hbar \cFut} 
&= \int_X \hat{s} (\omega) \hbar \mu_\eta ~\omega^n .\eta^\vee = n \int_X \hbar \mu \Ric (\omega) \wedge \omega^{n-1} - \bar{s} \int_X \hbar \mu \omega^n
\\
&= \int_X (\Ric (\omega) + \hbar \bar{\Box} \mu) \wedge (\omega + \hbar \mu)^n - \frac{\bar{s}}{n+1} \int_X (\omega + \hbar \mu)^{n+1}
\\
&= -2\pi \int_X \Big{(} j_0^* {^\hbar c_{U (1), 1}} (K_{\bar{\mathcal{X}}/\mathbb{P}^1}) \wedge (j_0^* {^\hbar c_{U (1), 1}} (\bar{\mathcal{L}}))^n 
\\
&\qquad \qquad \qquad + \frac{n}{n+1} \frac{- K_X. L^{ \cdot (n-1)}}{L^{\cdot n}} (j_0^* {^\hbar c_{U (1), 1}} (\bar{\mathcal{L}}))^{n+1} \Big{)}
\\
&=-2 \pi i_0^* \Big{(} {^\hbar (} K_{\bar{\mathcal{X}}/\mathbb{P}^1}^{ \mathbb{G}_m }. \bar{\mathcal{L}}_{ \mathbb{G}_m }^{\cdot n} )_{\mathbb{P}^1} + \frac{n}{n+1} \frac{(-K_X. L^{\cdot (n-1)})}{(L^{\cdot n})} {^\hbar (}\bar{\mathcal{L}}_{ \mathbb{G}_m }^{\cdot (n+1)})_{\mathbb{P}^1} \Big{)}, 
\end{align*}
which proves the first claim. 

For the second claim, we employ the localization formula in Example \ref{localization formula}: $\int_{\mathbb{P}^1} u. \hbar \eta^\vee =  i_{0_-}^* u - i_0^* u$ for $u \in {^\hbar H^2_{U (1)}} (\mathbb{P}^1)$ (see Example \ref{localization formula}). 
Since 
\begin{align*}
j_{0_-}^* {^\hbar c_{U (1), 1}} (\bar{\mathcal{L}}) 
&= c_1 (L) \oplus 0. \eta^\vee
\\
j_{0_-}^* {^\hbar c_{U (1), 1}} (K_{\bar{\mathcal{X}/\mathbb{P}^1}}) 
&= c_1 (K_X) \oplus 0. \eta^\vee
\end{align*}
in $H^2 (X) \oplus \mathfrak{u} (1)^\vee$, we have 
\begin{gather*}
i^*_{0_-} {^\hbar (} \bar{\mathcal{L}}_{ \mathbb{G}_m }^{\cdot (n+1)})_{\mathbb{P}^1} = ((L \oplus 0. \eta^\vee)^{\cdot n+1}) = 0. \eta^\vee, 
\\
i_{0_-}^* {^\hbar (} K_{\bar{\mathcal{X}}/\mathbb{P}^1}^{ \mathbb{G}_m }. \bar{\mathcal{L}}_{ \mathbb{G}_m }^{\dot n})_{\mathbb{P}^1} = ((K_X \oplus 0). (L \oplus 0. \bm{x})^{\cdot n}) = 0. \eta^\vee, 
\end{gather*} 
so that $i_{0_-}^* \Big{(} {^\hbar (} K_{\bar{\mathcal{X}}/\mathbb{P}^1}^{ \mathbb{G}_m } . \bar{\mathcal{L}}_{ \mathbb{G}_m }^{\cdot n})_{\mathbb{P}^1} + \frac{n}{n+1} \frac{(-K_X. L^{\cdot (n-1)})}{(L^{\cdot n})} {^\hbar (} \bar{\mathcal{L}}_{ \mathbb{G}_m }^{\cdot (n+1)})_{\mathbb{P}^1} \Big{)} = 0$. 
Thus we compute 
\begin{align*} 
{^\hbar \cFut} 
&=- 2 \pi \cdot i_0^* \Big{(} {^\hbar (} K_{\bar{\mathcal{X}}/\mathbb{P}^1}^{\mathbb{G}_m} . \bar{\mathcal{L}}_{\mathbb{G}_m}^{\dot n})_{\mathbb{P}^1} + \frac{n}{n+1} \frac{(- K_X . L^{\cdot (n-1)})}{(L^{\cdot n})} {^\hbar (} \bar{\mathcal{L}}_{\mathbb{G}_m}^{\cdot (n+1)})_{\mathbb{P}^1} \Big{)} 
\\
&= 2 \pi \cdot \int_{\mathbb{P}^1} \Big{(} {^\hbar (} K_{\bar{\mathcal{X}}/\mathbb{P}^1}^{\mathbb{G}_m} . \bar{\mathcal{L}}_{\mathbb{G}_m}^{\cdot n})_{\mathbb{P}^1} + \frac{n}{n+1} \frac{(- K_X . L^{\cdot (n-1)})}{(L^{\cdot n})} {^\hbar (} \bar{\mathcal{L}}_{\mathbb{G}_m}^{\smile (n+1)} )_{\mathbb{P}^1} \Big{)}. \hbar \eta^\vee 
\\
&= 2\pi \Big{(} ( K_{\bar{\mathcal{X}}/\mathbb{P}^1}. \bar{\mathcal{L}}^{\cdot n}) + \frac{n}{n+1} \frac{(- K_X . L^{\cdot (n-1)})}{(L^{\cdot n})} ( \bar{\mathcal{L}}^{\cdot (n+1)}) \Big{)}. \hbar \eta^\vee. 
\end{align*}
\end{proof}

This observation also yields the following: for a $G$-equivariant proper smooth morphism $\varpi: \mathcal{X} \to B$ between smooth varieties and a relatively ample $G$-equivariant line bundle $\mathcal{L}$ over $\mathcal{X}$, the equivariant cohomology class 
\[ \mathrm{CM}_G (\mathcal{X}/B, \mathcal{L}) = ( K_{\bar{\mathcal{X}}/B}^G. \bar{\mathcal{L}}_G^{\cdot n} )_B + \frac{n}{n+1} \frac{(-K_X. L^{\cdot (n-1)})}{(L^{\cdot n})} ( \bar{\mathcal{L}}_G^{\cdot (n+1)})_B \]
enjoys the following property: for every one parameter subgroup $\mathbb{G}_m \subset G$ and every $\mathbb{G}_m$-fixed point $i_b (\mathrm{pt}) = b \in B$, we have 
\[ {^\hbar \cFut}_{(\mathcal{X}_b, \mathcal{L}_b)} = - \frac{2\pi}{(L^{\cdot n})} {^\hbar \Phi} ( i_b^* \mathrm{CM}_G (\mathcal{X}/B, \mathcal{L})) \in {^\hbar H^2_{\mathrm{dR}, U (1)}} (\mathrm{pt}). \]
The smoothness assumption on the morphism is not essential as we have the equivariant Grothendieck--Riemann--Roch theorem for general proper flat morphism, while we currently need the smoothness of $B$ for the Poincare duality. 

\subsubsection{Equivariant cohomological interpretation of $\mu$-Futaki invariant}
\label{Equivariant cohomological interpretation}

Here we explain how to derive the $\mu$-Futaki invariant for test configurations from the differential geometric definition of $\mu$-Futaki invariant. 

As we remarked in section \ref{mu-Futaki and mu-volume for vectors}, we can express the $\mu$-volume functional by the equivariant intersection: 
\begin{align*} 
\log {^\hbar \mathrm{Vol}^\lambda} (\xi) 
&= \Big{(} \int_X (s (\omega) + \hbar \bar{\Box} \mu_\xi) e^{\hbar \mu_\xi} \omega^n - \lambda \int_X (n+ \hbar \mu_\xi) e^{\hbar \mu_\xi} \omega^n \Big{)} \Big{/} \int_X e^{\hbar \mu_\xi} \omega^n 
\\
&\qquad + \lambda \log \Big{(} \frac{1}{n!} \int_X e^{\hbar \mu_\xi} \omega^n \Big{)} + \lambda n + \lambda \log n!
\\
&= - 2\pi \frac{{^\hbar (\kappa_X^T. e^{L_T}; \xi)}}{{^\hbar (e^{L_T}; \xi)}} - \lambda \Big{(} \frac{{^\hbar (L_T. e^{L_T}; \xi)}}{{^\hbar (e^{L_T}; \xi)}} - \log {^\hbar (e^{L_T}; \xi)} \Big{)} + \log (n! e^n)^\lambda. 
\end{align*} 
We introduce the following variant 
\begin{align*} 
{^\hbar \cmu^\lambda_{(X, \omega)} } (\xi) 
&:= - \log \frac{{^\hbar \mathrm{Vol}^\lambda} (\xi)}{(n! e^n)^\lambda}  
\\
&= 2\pi \frac{{^\hbar (\kappa_X^T. e^{L_T}; \xi)}}{{^\hbar (e^{L_T}; \xi)}} + \lambda \Big{(} \frac{{^\hbar (L_T. e^{L_T}; \xi)}}{{^\hbar (e^{L_T}; \xi)}} - \log {^\hbar (e^{L_T}; \xi)} \Big{)}. 
\end{align*}

Since ${^\hbar \cmu^\lambda}$ differs from $-\log {^\hbar \mathrm{Vol}^\lambda}$ only by a constant, we have 
\begin{equation}
\label{mu to Fut}
(D_{\xi} {^\hbar \cmu^\lambda_{(X, \omega)} }) (\zeta) = \frac{d}{dt}\Big{|}_{t=0} {^\hbar \cmu^\lambda_{(X, \omega)}} (\xi + t \zeta) = - {^\hbar \cFut^\lambda_\xi} (\zeta) = - \hbar \cdot \cFut^\lambda_{\hbar. \xi} (\zeta). 
\end{equation}
One can think of this as giving an equivariant cohomological interpretation of $\mu$-Futaki invariant. 
We would like to realize the derivative $D_\xi$ in a more equivariant cohomological way so that we can apply this idea to the relative construction $\D_{\hbar. \xi} \cmu^\lambda_{T \times G} (\mathcal{X}/B, \mathcal{L})$. 
This is realized by the operator $\D_{\hbar. \xi}$ we explained in section \ref{The beauty of the Cartan model}. 
We will observe this in Proposition \ref{derivative}. 

More precisely, recall the equivariant intersection ${^\hbar (K_X^T. e^{L_T})}: \mathfrak{t} \to \mathbb{R}$ is constructed from the formal equivariant intersection
\[ (K_X^T. e^{L_T}) := \sum_{p=0}^\infty \frac{1}{p!} (K_X^T. L_T^{\cdot p}) = \sum_{k=0}^\infty \frac{1}{(n-1+k)!} (K_X^T. L_T^{\cdot n -1 + k}), \]
where $(K_X^T. L_T^{\cdot n -1 + k}) \in H^{2k}_T (\mathrm{pt})$. 
Since $(e^{L_T})^{\langle 0 \rangle} = (L^{\cdot n})/n!$ of $(e^{L_T})$ is non-zero, we can realize the following equivariant cohomology class in $\hat{H}^{\mathrm{even}}_T (\mathrm{pt}, \mathbb{Q})$: 
\begin{equation} 
\label{absolute mu character}
\cmu^\lambda_T (X, L) := 2\pi (K_X^T. e^{L_T}) \cdot (e^{L_T})^{-1} + \lambda \Big{(} (L_T. e^{L_T}) \cdot (e^{L_T})^{-1} - \log (e^{L_T}) \Big{)}. 
\end{equation}
The elements $(e^{L_T})^{-1}, \log (e^{L_T})$ are defined by (\ref{formal series of equivariant cohomology class}). 
Definition \ref{mu-character definition} is a direct generalization of this class to the relative case. 

Now we can realize ${^\hbar \Phi^{-1}} (D_\xi {^\hbar \cmu^\lambda_{(X, \omega)}}) = {^\hbar \Phi^{-1}} ({^\hbar \cFut^\lambda_\xi}) \in H^2_T (\mathrm{pt}, \mathbb{R})$ from the formal equivariant cohomology class $\cmu^\lambda_T (X, L)$, rather than the functional ${^\hbar \cmu^\lambda_{(X, \omega)}}$, as follows. 
We observe, using $f (\xi, \zeta) := {^\hbar \cmu^\lambda_{(X, \omega)}} (\xi + \zeta)$ on $\mathfrak{t} \times \mathfrak{t}$, the differential of ${^\hbar \cmu^\lambda_{(X, \omega)}}$ at $\xi$ is expressed as the differential of $f (\xi, \cdot)$ at the origin $0$. 
The pullback $\cmu^\lambda_{T \times T} (X, L)$ of $\cmu^\lambda_T (X, L) $ along $T \times T \to T: (s, t) \mapsto st$ corresponds to this $f$. 
We then apply the operator $\D_{\hbar. \xi}$ to $\cmu^\lambda_{T \times T} (X, L)$. 

We recall the operator $\D_{\hbar. \xi}$ is defined on $H^{\mathrm{even}}_{T \times G} (B, \mathbb{R})$ rather than on its completion $\hat{H}^{\mathrm{even}}_{T \times G} (B, \mathbb{R})$, in which the characteristic class $\cmu^\lambda_{T \times G} (\mathcal{X}/B, \mathcal{L})$ lives. 
For $\alpha \in H^{\mathrm{even}}_{T \times G} (B, \mathbb{R})$, $\D_{\hbar. \xi}^q \alpha$ turns into an infinite series of equivariant cohomology classes, so we must show its convergence. 
This leads us to the study in section \ref{derivation on equivariant cohomology class} and \ref{equivariant calculus}. 

\subsection{Pushforward in equivariant locally finite homology}

\subsubsection{Continuity of the form-to-homology pushforward map}
\label{section: spectral}

Now we show Theorem \ref{Hausdorffness of equivariant cohomology}. 
We will use this result to the equivariant calculus we study in section \ref{equivariant calculus}, which is the essential part of Theorem B. 

Let $B$ be a smooth real manifold. 
The space of equivariant currents 
\[ (\mathcal{D}')_k^K (B) = \Big{(} \bigoplus_{q-2p=k} S^p \mathfrak{k}^\vee \otimes \mathcal{D}'_q (B) \Big{)}^K \] 
is endowed with the weakest topology which makes the maps 
\[ (\mathcal{D}')_k^K (B) \to \mathbb{R} : \sigma \mapsto \langle \sigma_\xi ,\varphi \rangle \] 
continuous for every $\xi \in \mathfrak{k}$ and $\varphi \in \bigoplus_{p \ge 0} \mathcal{D}^{k+2p} (B)$. 
Here we identify elements of $(\mathcal{D}')_k^K (B)$ with $K$-equivariant polynomial maps $\Sigma: \mathfrak{k} \to \bigoplus_i \mathcal{D}'_{k+2p} (B)$ of degree bounded by $(\dim_\mathbb{R} B - k)/2$ in a natural way. 
The boundary map $\partial^K_k: (\mathcal{D}')_k^K (B) \to (\mathcal{D}')_{k-1}^K (B)$ and the chain-level pushforward map $f_*: (\mathcal{D}')_k^K (X) \to (\mathcal{D}')_k^K (B)$ for a proper smooth map $f: X \to B$ is continuous with respect to this topology. 

We would show the quotient topology on the current homology ${^\hbar H^{\mathrm{cur}, K}_k} (B)$ induced from $(\mathcal{D}')_k^K (B)$ is Hausdorff, even when $B$ is non-compact. 
This implies the following \textit{form-to-homology pushforward map} is continuous for every $K$-equivariant proper $C^\infty$-map $f: X \to B$ 
\begin{equation} 
\label{push}
f_*: \Omega^{n+k}_K (X) \cap (f_*)^{-1} \mathrm{Ker} ({^\hbar \partial^K_{b-k}}) \to {^\hbar H_{b-k}^{\mathrm{cur}, K}} (B)
\end{equation}
with respect to the Fr\'echet topology on $\Omega^{n+k}_K (X)$ and the unique Hausdorff topology on the finite dimensional space ${^\hbar H^{\mathrm{cur}, K}_{b-k}} (B)$. 
Here we put $b := \dim_{\mathbb{R}} B$ and $n := \dim_{\mathbb{R}} X - \dim_{\mathbb{R}} B$. 
Note the Hausdorff topology on ${^\hbar H^{\mathrm{cur}, K}_{b-k}} (B)$ is a priori unrelated to the quotient topology on ${^\hbar H_{b-k}^{\mathrm{cur}, K}} (B)$ induced from the construction. 
(For instance, Dolbeault cohomology is known to be a non-Hausdorff cohomology theory in general. )
It is evident the above map is continuous with respect to the latter topology. 

We compute the topology on ${^\hbar H^{\mathrm{cur}, K}_{b-k}} (B)$ using the spectral sequence of topological vector spaces derived from the double complex of the Cartan model. 
The following de Rham's theorem is the key in this computation. 

\begin{prop}
\label{deRham} \cite[Chapter IV, Theorem 17']{deR}
A $q$-current $\sigma \in \mathcal{D}'_k (B)$ is exact if and only if $\langle \sigma, \varphi \rangle = 0$ for every closed compactly supported $C^\infty$-form $\varphi \in \mathcal{D}^k (B)$. 
\end{prop}

\begin{cor}
The quotient topology on the current homology $H_k^{\mathrm{cur}} (B)$ is Hausdorff. 
\end{cor}

\begin{proof}
Since $\partial \mathcal{D}_{k+1}' (B) = \bigcap_{\varphi \in \mathcal{D}^k (B), d \varphi = 0} \mathrm{Ker} (\langle \cdot, \varphi \rangle)$, the space of exact $k$-currents $\partial \mathcal{D}_{k+1}' (B) \subset \mathcal{D}_k (B)$ is a closed subset of $\mathcal{D}_k (B)$. 
This shows the claim. 
\end{proof}

We prepare two lemmas. 

\begin{lem}
\label{separation lemma}
Let $V_1$ be a topological vector space and $V_2$ be a Hausdorff topological vector space. 
Suppose there is a continuous map $p: V_1 \to V_2$ such that the induced topology on the subspace $V_0 := p^{-1} (0)$ is Hausdorff, then $V_1$ is also Hausdorff. 
\end{lem}

\begin{proof}
The topological vector space $V_1$ is Hausdorff iff the origin $\{ 0 \} \subset V_1$ is closed. 
The closure $W := \overline{\{ 0 \}}$ in $V_1$ is a linear subspace of $V_1$ and $W \cap V_0 = \{ 0 \}$ as $V_0$ is Hausdorff. 
The closure of $\{ 0 \}$ in the quotient space $V_1/V_0$ coincides with $(W + V_0)/V_0$. 
On the other hand, as $V_2$ is Hausdorff and $p$ is continuous, $V_0 = p^{-1} (0)$ is a closed subspace of $V_1$. 
It follows that the quotient $V_1/V_0$ is Hausdorff, so that we have $\{ 0 \} = \overline{\{ 0 \}} = (W + V_0)/V_0$. 
This proves $W \subset W \cap V_0 = \{ 0 \}$. 
\end{proof}

\begin{lem}
\label{quotient topological vector space}
Let $V$ be a topological vector space, and $W$ and $W'$ be its subspaces. 
Then the natural linear bijection $(V/W)/(W'/(W \cap W')) \to V/(W+W')$ is a homeomorphism of topological vector spaces, where we identify $W'/(W \cap W')$ with a subspace of $V/W$ by the injection $W'/W \cap W' \to V/W$. 
\end{lem}

\begin{proof}
By definition of quotient topology, $V \to V/(W+W')$ is continuous. 
Then $V/W \to V/(W +W')$ is continuous by the universality of quotient topology, hence $(V/W)/(W'/(W \cap W')) \to V/(W+W')$ is continuous again by the universality of quotient topology. 

On the other hand, $V \to V/W \to (V/W)/(W'/(W \cap W'))$ is continuous by definition. 
Then the inverse map $V/(W+W') \to (V/W)/(W'/(W \cap W'))$ is continuous by the universality of quotient topology. 
\end{proof}

\begin{prop}
Let $(\{ C^{p, q} \}_{p \ge 0, q \ge 0}, \delta, d)$ be a first quadrant double complex of (Hausdorff) topological vector spaces with continuous derivatives $\delta, d$. 
Suppose the $E_1$-page of the associated spectral sequence 
\begin{align*} 
E_1^{m, n} = \frac{ C^{m,n} \cap \mathrm{Ker} d }{ d C^{m, n-1} }
\end{align*}
is finite dimensional and the induced topology is Hausdorff for every $m, n$. 
Then the cohomology $H^k (\prod_{p + q = \bullet} C^{p, q}, \delta + d)$ of the total complex is a finite dimensional Hausdorff topological vector space with respect to the quotient topology induced from $\prod_{p+q=k} C^{p, q}$. 
\end{prop}

\begin{proof}
Let us recall the usual argument of spectral sequence. 
We put $C^k_l := \prod_{p+q = k, p \ge l} C^{p,q}$ and 
\begin{align*} 
E^{l,k-l}_r 
&:= \frac{ C^k_l \cap (\delta + d)^{-1} C^{k+1}_{l+r} }
{ \Big{(} C^k_l \cap (\delta + d) C^{k-1}_{l+1-r} \Big{)} + \Big{(} C^k_{l+1} \cap (\delta + d)^{-1} C^{k+1}_{l+r} \Big{)} }.
\end{align*} 

We firstly note 
\begin{align*} 
E_1^{m, n} 
&= \frac{ C^{m+n}_m \cap (\delta + d)^{-1} C^{m+n+1}_{m+1} }{ (C^{m+n}_m \cap (\delta + d) C^{m+n-1}_m) + (C^{m+n}_{m+1} \cap (\delta + d)^{-1} C^{m+n+1}_{m+1}) }
\\
&= \frac{ \prod_{p+q=m+n, p \ge m+1} C^{p,q} \times (C^{m, n} \cap \mathrm{Ker} d) }{ \prod_{p+q=m+n-1, p \ge m} (\delta C^{p,q} + d C^{p,q}) + \prod_{p+q = m+n, p \ge m+1} C^{p,q} \times \{ 0 \}} 
\\
&= \frac{ \prod_{p+q=m+n, p \ge m+1} C^{p,q} \times (C^{m, n} \cap \mathrm{Ker} d) }{ \prod_{p+q = m+n, p \ge m+1} C^{p,q} \times d C^{m, n-1} } 
\\
&= \frac{ C^{m, n} \cap \mathrm{Ker} d }{ d C^{m, n-1} } 
\end{align*}
as topological vector space, where the last equality holds by Lemma \ref{quotient topological vector space}. 

There is a decreasing filtration 
\[ H^k (\prod_{p + q = \bullet} C^{p, q}, \delta + d) = H^k_0 \supset \dotsb \supset H^k_1 \supset \dotsb \supset H^k_k \supset 0 \] 
on the cohomology $H^k (\prod_{p + q = \bullet} C^{p, q}, \delta + d)$ of the total complex derived from the decreasing filtration $\{ C^k_l \cap \mathrm{Ker} (\delta + d) \}_{l=0}^k$
of $\mathrm{Ker} (\delta + d)$. 

Now we consider the quotient topology on each $H^k_l$ induced from the subspace $C^k_l \cap \mathrm{Ker} (\delta + d)$ of the product $\prod_{p+q = k} C^{p,q}$ (endowed with the product topology). 
By Lemma \ref{separation lemma}, it suffices to show the quotient topology on 
\[ E^{l,k-l}_\infty := H^k_l/H^k_{l+1} = \frac{ (C_l^k \cap \mathrm{Ker} (\delta +d) )/(C_l^k \cap \mathrm{Im} (\delta+ d)) }{ (C_{l+1}^k \cap \mathrm{Ker} (\delta +d) )/(C_{l+1}^k \cap \mathrm{Im} (\delta+ d)) } \] 
induced from $H^k_l$ is Hausdorff. 
By the usual fact on spectral sequence, we can (algebraically) compute the quotient vector space $E^{l,k-l}_\infty$ by computing the cohomologies of $E_r$-pages $E^{l,k-l}_r$ successively. 
We must see the successive computation of $E_r$-page also detects the Hausdorffness. 

By Lemma \ref{quotient topological vector space}, we have a linear bijective homeomorphism 
\begin{equation}
\label{E-infinity} 
E^{l,k-l}_\infty \to \frac{ C^k_l \cap \mathrm{Ker} (\delta + d) }{ \Big{(} C^k_l \cap \mathrm{Im} (\delta + d) \Big{)} + \Big{(} C^k_{l+1} \cap \mathrm{Ker} (\delta + d) \Big{)} }, 
\end{equation}
so that it suffices to show that the right hand side is Hausdorff. 

We consider the quotient topology on $E^{l, k-l}_r$ induced from $C^k_l \cap (\delta + d)^{-1} C^{k+1}_{l+r}$. 
Then since $E^{l, k-l}_r$ for $r > \max (l, k-l)$ coincides with the right hand side in (\ref{E-infinity}) as topological vector spaces, the Hausdorffness of $E^{l, k-l}_\infty$ follows from that of $E^{l, k-l}_r$ for $r > \max (l, k-l)$. 

Recall the usual argument of spectral sequence: we have a linear map $d^{m,n}_r: E^{m,n}_r \to E_r^{m+r, n-r+1}$ such that $d^{m+r, n-r+1}_r \circ d^{m,n}_r = 0$ and also have a linear bijection $E^{m,n}_{r+1} \to \mathrm{Ker} (d^{m,n}_r)/\mathrm{Im} (d^{m-r, n+r-1})$. 
The linear map $d^{l,k-l}_r: E^{l,k-l}_r \to E_r^{l+r, (k-l)-r+1}$ is induced from the continuous map 
\[ \delta +d: C^k_l \cap (\delta+d)^{-1} C^{k+1}_{l+r} \to C^{k+1}_{l+r} \cap (\delta+d)^{-1} C^{k+2}_{l+2r} \] 
and the linear bijection $E^{m,n}_{r+1} \to \mathrm{Ker} (d^{m,n}_r)/\mathrm{Im} (d^{m-r, n+r-1})$ is induced from the continuous inclusion 
\[ C^k_l \cap (\delta + d)^{-1} C^k_{l+r+1} \hookrightarrow C^k_l \cap (\delta + d)^{-1} C^k_{l+r}, \]
so that these maps are continuous linear bijection, while we do not state here the continuity of the inverse map as there is no open mapping theorem for general topological vector spaces. 
Thanks to the direction of the continuous bijection $E^{m,n}_{r+1} \to \mathrm{Ker} (d^{m,n}_r)/\mathrm{Im} (d^{m-r, n+r-1})$ and Lemma \ref{separation lemma}, $E^{m,n}_{r+1}$ is Hausdorff when the quotient topology on $\mathrm{Ker} (d^{m,n}_r)/\mathrm{Im} (d^{m-r, n-1+r})$ induced from $E^{m,n}_r$ is Hausdorff. 

Now our assumption that $E^{m,n}_1$ are finite dimensional Hausdorff spaces implies every subspace of $E^{m,n}_1$ is closed, so that $E^{m,n}_2$ are again finite dimensional Hausdorff spaces by the above general argument. 
Running the induction, we conclude that $E^{m,n}_r$ are finite dimesnional Hausdorff spaces for every $r \ge 1$, and so are the spaces $E^{m,n}_\infty$. 
\end{proof}

\begin{prop}
The quotient topology on the equivariant current homology ${^\hbar H_k^{\mathrm{cur}, K}} (B)$ induced from the weak topology on the space of equivariant currents $(\mathcal{D}')_k^K (B)$ is Hausdorff for every $k \in \mathbb{Z}$. 
\end{prop}

\begin{proof}
This follows by applying the above proposition to the double complex of Cartan model of equivariant current homology (with the reversed index) $E^{p, q}_0 = (S^p \mathfrak{k}^\vee \otimes \mathcal{D}'_{b -(p+q)} (B))^K$, whose assumption is confirmed by Proposition \ref{deRham} and the computation of $E_1$-term: 
\[ E^{p,q}_1 = (S^p \mathfrak{k}^\vee \otimes H^{\mathrm{cur}}_{b - (p+q)} (B))^K \]
as topological vector spaces.  
\end{proof}

\begin{cor}
\label{continuity of push}
The pushforward map (\ref{push}) is continuous with respect to the Fr\'echet topology on $\Omega^{n+k}_K (X)$ and the Hausdorff topology on ${^\hbar H^{\mathrm{cur}, K}_{b-k}} (B)$. 
\end{cor}

We apply this continuity result to the key construction in section \ref{equivariant calculus}, together with the following lemma. 

\begin{lem}
\label{Frechet to Banach}
Let $V$ be a Fr\'echet space and $\{ \| \cdot \|_l \}_{l \in \mathbb{Z}_{\ge 0}}$ be a collection of seminorms on $V$ defining its Fr\'echet structure. 
Let $W$ be a Banach space and $F: V \to W$ be a continuous linear map. 
Let $\{ v_i \}_{i=0}^\infty \in V$ be a sequence such that $\sum_{i=0}^\infty \| v_i \|_l < \infty$ for every $l \ge 0$. 
Then the infinite series $\sum_{i=0}^\infty F (v_i)$ is absolutely convergent with respect to the norm of $W$. 
\end{lem}

\begin{proof}
Remember that a linear map $F: V \to W$ from Fr\'echet space to Banach space is continuous if and only if there exists a constant $C > 0$ and $N \in \mathbb{Z}_{\ge 0}$ such that 
\[ \| F (v) \|_W \le C (\| v \|_0 + \dotsb \| v \|_N) \]
for every $v \in V$. 
So the claim follows by 
\[ \sum_{i=0}^\infty \| F (v_i) \|_W \le C \sum_{l=0}^N \sum_{i=0}^\infty \| v_i \|_l < \infty. \]
\end{proof}

\subsubsection{Equivariant homology todd class and equivariant Grothendieck--Riemann--Roch theorem}
\label{equivariant homology todd class}

Here we recall the equivariant Grothendieck--Riemann--Roch theorem for algebraic schemes established by Edidin--Graham \cite{EG2} as equivariant version of \cite{Ful}. 
The equivariant Chow group $A^G_p (X)$ is studied in \cite{EG1} which is defined in the same way as the equivariant locally finite homology. 
The statement is as follows. 

\begin{thm}
For each algebraic group $G$ and each algebraic $G$-schemes $X$, we can assign a homomorphism
\[ \tau^G_X: K (\mathrm{Coh}^G (X)) \to \hat{A}^G_\mathbb{Q} (X) \]
from the Grothendieck group $K (\mathrm{Coh}^G (X))$ of $G$-equivariant algebraic coherent sheaves on $X$ to the $G$-equivariant Chow group $\hat{A}^G_\mathbb{Q} (X) = \prod_{p \in \mathbb{Z}} A^G_p (X) \otimes \mathbb{Q}$ of $\mathbb{Q}$-coefficient which enjoys the following properties: 
\begin{enumerate}
\item (Grothendieck--Riemann--Roch) For any $G$-equivariant proper morphism $f: X \to Y$ of algebraic schemes, we have $f_* \tau^G_X (\alpha) = \tau^G_Y (f_! \alpha)$ for every $\alpha \in K (\mathrm{Coh}^G (X))$. 
Here $f_! \alpha$ for an element $\alpha = [\mathcal{F}]$ represented by a $G$-equivariant coherent sheaf $\mathcal{F}$ denotes the element $\sum_i (-1)^i [R^i f_* \mathcal{F}]$ in $K (\mathrm{Coh}^G (Y))$, where the higher direct image sheaves $R^i f_* \mathcal{F}$ are $G$-linearized in a natural way. 

\item For every $\alpha \in K (\mathrm{Coh}^G (X))$ and $\beta \in K (\mathrm{Vect}^G (X))$, we have $\tau^G_X (\alpha \otimes \beta) = \tau^G_X (\alpha) \frown \mathrm{ch}_G (\beta)$. 

\item For any closed subscheme $Z \subset X$ of pure dimension $p$, we have $\tau^G_X (\mathcal{O}_Z)_{\langle p \rangle} = [Z]_G \in A_n^G (X)$. 

\item When $X$ is smooth, we have $\mathrm{PD}^X_G ((\tau^G_X (\mathcal{O}_X))_{\langle p \rangle}) = \mathrm{Td}_G^{n-p} (X)$. 

\item For a homomorphism $\varphi: H \to G$ of algebraic groups, we have $\varphi^\# \circ \tau^G_X = \tau^H_X \circ \varphi^\#$ for naturally defined maps $\varphi^\#: K (\mathrm{Coh}^G (X)) \to K (\mathrm{Coh}^H (X))$ and $\varphi^\#: \hat{A}^G_{\mathbb{Q}} (X) \to \hat{A}^H_{\mathbb{Q}} (X)$. 
\end{enumerate}
For $\alpha \in \hat{A}^G_\mathbb{Q} (X)$, we denote by $\alpha_{\langle p \rangle} \in A^G_p (X) \otimes \mathbb{Q}$ the degree $p$-component. 
\end{thm}

There is also an analytic version of the result: \cite{Lev}. 

\begin{eg}
Let $f: \tilde{X} \to X$ be a $G$-equivariant proper morphism of pure dimensional $G$-schemes which is isomorphic away from a codimension $k+1$ subscheme of the target $X$. 
Namely, we assume there is a subscheme $Z \subset X$ of codimension $k+1$ such that the restriction $f^{-1} (X \setminus Z) \to X \setminus Z$ gives an isomorphism. 
Then we have $f_* (\tau^G_{\tilde{X}} (\mathcal{O}_{\tilde{X}}))_{\langle \dim \tilde{X} - i \rangle} = \tau^G_X (\mathcal{O}_X)_{\langle \dim X -i \rangle}$ for $i \le k$. 
Indeed, $f_! [\mathcal{O}_{\tilde{X}}] - [\mathcal{O}_X] = [f_* \mathcal{O}_{\tilde{X}}/\mathcal{O}_X] + \sum_{i \ge 1} (-1)^i [R^i f_* \mathcal{O}_{\tilde{X}}]$ is supported on $Z$ and hence $f_* (\tau^G_{\tilde{X}} (\mathcal{O}_{\tilde{X}}))_{\langle j \rangle} = (\tau^G_X (f_! [\mathcal{O}_{\tilde{X}}]))_{\langle j \rangle}$ for $j > \dim Z$. 

If $X$ has only rational singularities, we have $f_! \mathcal{O}_{\tilde{X}} = \mathcal{O}_X$ for any equivariant resolution $f: \tilde{X} \to X$, so that $\tau^G_X (\mathcal{O}_X) = f_* \tau^G_{\tilde{X}} (\mathcal{O}_{\tilde{X}}) = f_* ([\tilde{X}] \frown \mathrm{Td}_G (\tilde{X}))$. 
\end{eg}

Since $\mathrm{Td}_G (X)^{\langle 0 \rangle} = 1$, the equivariant Todd class $\mathrm{Td}_G (X) \in \hat{H}^{\mathrm{even}}_G (X, \mathbb{Q})$ for smooth $X$ is an invertible element with respect to the cup product. 

\begin{defin}
\label{equivariant canonical class}
For a pure $n$-dimensional algebraic $G$-scheme $X$, we define the \textit{equivariant canonical class} $\kappa^G_X \in H^{\mathrm{lf}, G}_{2n-2} (X, \mathbb{Q})$ by 
\begin{equation} 
\kappa^G_X := -2 cl^G (\tau^G_X (\mathcal{O}_X))_{\langle n-1 \rangle} 
\end{equation}
under the equivariant cycle map $cl^G: A^G_{n-1} (X) \to H^{\mathrm{lf}, G}_{2n-2} (X)$. 

For a relatively pure dimensional $G$-equivariant proper flat morphism $\varpi: \mathcal{X} \to B$ from an algebraic $G$-scheme $\mathcal{X}$ to a smooth $G$-variety $B$, we define the \textit{relative equivariant canonical class} $\kappa^G_{\mathcal{X}/B} \in H^{\mathrm{lf}, G}_{2 \dim \mathcal{X} - 2} (\mathcal{X}, \mathbb{Q})$ by 
\begin{align}
\kappa^G_{\mathcal{X}/B} 
&:= -2 cl^G (\tau^G_\mathcal{X} (\mathcal{O}_\mathcal{X}) \frown \varpi^* \mathrm{Td}_G (B)^{-1})_{\langle \dim \mathcal{X} -1 \rangle} 
\\ \notag
&= \kappa^G_\mathcal{X} - \varpi^* K^G_B,  
\end{align}
where we put $\varpi^* K^G_B := [\mathcal{X}] \frown \varpi^* c_{G,1} (B)$. 
\end{defin}

For a normal variety $X$, we denote by $K_X^G$ the locally finite homology class corresponding to the equivariant frist Chern class $c_{G, 1} (\omega_{X^\mathrm{reg}}) = - c_{G, 1} (X^\mathrm{reg}) \in H^2_G (X^\mathrm{reg}, \mathbb{Z})$ via the isomorphism $H^2_G (X^\mathrm{reg}) \cong H^{\mathrm{lf}, G}_{2n-2} (X^\mathrm{reg}) \cong H^{\mathrm{lf}, G}_{2n-2} (X)$. 
As we see in the above example, we have $\kappa^G_X = K_X^G$. 

\begin{cor}
\label{base change}
Let $\varpi: \mathcal{X} \to B$ be a relatively pure $n$-dimensional $G$-equivariant proper flat morphism from an algebraic $G$-scheme $\mathcal{X}$ to a smooth $G$-variety $B$. 
Let $\varphi: H \to G$ be a homomorphism of algebraic groups. 
For a $H$-equivariant morphism $f: B' \to B$ from a smooth $H$-variety $B'$, we put $\mathcal{X}' := \mathcal{X} \times_B B'$ and denote by $\hat{f}: \mathcal{X}' \to \mathcal{X}$ and $\varpi': \mathcal{X}' \to B'$ the projection to the first factor and the second factor, respectively. 

Then for any $G$-equivariant line bundle $\mathcal{L}$ on $\mathcal{X}$, we have 
\[ f^* \varphi^\# \big{(} \kappa^G_{\mathcal{X}/B}. \mathcal{L}_G^{\cdot (n + p-1)} \big{)}_B = \big{(} \kappa^H_{\mathcal{X}'/B'}. \hat{f}^* \varphi^\# \mathcal{L}_G^{\cdot (n + p-1)} \big{)}_{B'} \]
in $H_H^{2p} (B', \mathbb{Q})$. 

In other words, $(\kappa^G_{\mathcal{X}/B}. e^{\mathcal{L}_G})_B \in \hat{H}^{\mathrm{even}}_G (B, \mathbb{Q})$ is a characteristic class for equivariant polarized family. 
\end{cor}

\begin{proof}
Since $B$ is smooth, we have $K (\mathrm{Coh}^G (X)) = K (\mathrm{Vect}^G (X))$ (cf. \cite{EG2}, \cite{Tho}). 
By the equivariant Grothendieck--Riemann--Roch theorem, we have 
\[ \varpi_* \tau_\mathcal{X}^G (\mathcal{L}^{\otimes k}) = \tau_B^G (\varpi_! \mathcal{L}^{\otimes k}) = \tau^G_B (\mathcal{O}_B) \frown \mathrm{ch}_G (\varpi_! \mathcal{L}^{\otimes k})  \]
in $\hat{H}^{\mathrm{lf}, G} (B, \mathbb{Q}) = \prod_{q \in \mathbb{Z}} H^{\mathrm{lf}, G}_q (B, \mathbb{Q})$. 
(We note when $\mathcal{L}$ is relatively ample, we do not need to assume the smoothness of $B$ as $\pi_! \mathcal{L}^{\otimes k} = \pi_* \mathcal{L}^{\otimes k}$ is a vector bundle for $k \gg 0$. )
Since $\mathrm{Td}_G (B) = \mathrm{PD}^B_G \tau^G_B (\mathcal{O}_B)$ is invertible in $\hat{H}^{\mathrm{even}}_G (B, \mathbb{Q})$ with respect to the cup product, we obtain 
\[ \mathrm{ch}^G (\varpi_! \mathcal{L}^{\otimes k}) = \varpi_\diamond \tau^G_\mathcal{X} (\mathcal{L}^{\otimes k}) \smile \mathrm{Td}_G (B)^{-1}, \]
where we put $\varpi_\diamond := \mathrm{PD}^B_G \circ \varpi_*: H^{\mathrm{lf}, G}_l (\mathcal{X}, \mathbb{Q}) \to H^{2\dim B - l}_G (B, \mathbb{Q})$. 
As 
\[ \varpi_\diamond \tau^G_\mathcal{X} (\mathcal{L}^{\otimes k}) = \varpi_\diamond (\tau^G_\mathcal{X} (\mathcal{O}_\mathcal{X}) \frown \mathrm{ch}_G (\mathcal{L}^{\otimes k})) = \varpi_\diamond (\tau^G_\mathcal{X} (\mathcal{O}_\mathcal{X}) \frown e^{k c_{G, 1} (\mathcal{L})}), \]
we compute 
\begin{align*}
\big{(} \mathrm{ch}^G (\varpi_! \mathcal{L}^{\otimes k}) \big{)}^{\langle \dim B-q \rangle} 
&= \sum_{i=0}^\infty \frac{k^i}{i!} \Big{(} \varpi_\diamond (\tau^G_\mathcal{X} (\mathcal{O}_\mathcal{X}) \frown c_{G, 1} (\mathcal{L})^{\smile i})) \smile \mathrm{Td}_G (B)^{-1} \Big{)}^{\langle \dim B-q \rangle} 
\\
&= \sum_{i=0}^{\dim \mathcal{X} - p} \frac{k^i}{i!} \varpi_\diamond \big{(} (\tau^G_\mathcal{X} (\mathcal{O}_\mathcal{X}) \frown \varpi^* \mathrm{Td}_G (B)^{-1} )_{\langle i + q \rangle} \frown c_{G, 1} (\mathcal{L})^{\smile i} \big{)}
\end{align*} 
by the projection formula. 
It follows that 
\[ \big{(} \kappa^G_{\mathcal{X}/B}. \mathcal{L}_G^{\cdot (n + p-1)} \big{)}_B = \varpi_\diamond (\kappa^G_{\mathcal{X}/B} \frown c_{G, 1} (\mathcal{L})^{\smile (n+p-1)})  \]
is the coefficient of degree $n+p-1$ component of the polynomial map $k \mapsto (\mathrm{ch}^G (\varpi_! \mathcal{L}^{\otimes k}))^{\langle p \rangle} \in H^{2p}_G (B, \mathbb{Q})$. 
Now the claim follows from
\[ f^* (\mathrm{ch}^G (\varpi_! \mathcal{L}^{\otimes k})) = \mathrm{ch}^G (f^* \varpi_! \mathcal{L}^{\otimes k}) = \mathrm{ch}^G (\varpi'_! (\hat{f}^* \mathcal{L})^{\otimes k}). \]
\end{proof}

Since the map $H^2_G (\mathcal{X}, \mathbb{R}) \to H^{2p}_G (B, \mathbb{R}): x \mapsto (\kappa^G_{\mathcal{X}/B}. x^{\cdot (n+p-1)})_B$ is homogeneous and hence continuous, we conclude the same result for $\mathcal{L}_G \in NS_G (\mathcal{X}, \mathbb{R})$. 

\section{Equivariant calculus on $\mu$-character}

\subsection{Derivation on equivariant cohomology}
\label{derivation on equivariant cohomology class}

\subsubsection{Equivariant cohomology class of class $\varepsilon^\ell$}
\label{Equivariant cohomology class of class epsilon ell}

Now we study generalities on the operator $\D_{\hbar. \xi}^q: H^*_{T \times G} (X, \mathbb{R}) \to H^{2q}_G (X, \mathbb{R})$ defined by (\ref{derivation}). 
We use this notion to construct the cohomological $\mu$-Futaki invariant $\D_{\hbar. \xi} \cmu^\lambda_{T \times G} (\mathcal{X}/B, \mathcal{L})$. 

Let $G$ be a topological group, $B$ be a connected topological space with a $T \times G$-action under which $B$ is $T$-equivariantly homotopically equivalent to a space with the trivial $T$-action. 
In this article, we are interested in the case when $T$ acts trivially on $B$. 
Then by the assumption on the $T$-action, we have the K\"unneth decomposition (\ref{Kunneth}). 
Similarly as (\ref{derivation}), we can define the operator $\D_{\hbar. \xi}: H^*_{T \times G} (X) \to H^{2q}_G (X)$ for general topological space $X$, for which we cannot consider the Cartan model ${^\hbar H_{\mathrm{dR}, T_{\mathrm{cpt}} \times K}}$: 
\[ \D^q_{\hbar. \xi}: H^*_{T \times G} (X, \mathbb{R}) = H^*_T (\mathrm{pt}) \otimes H^*_G (X) \xrightarrow{{^\hbar \Phi_T^{\mathrm{pt}}} \otimes \mathrm{id}} S^* \mathfrak{t}^\vee \otimes H^*_G (X) \xrightarrow{q! \cdot t_q \circ \rho_\xi^T} H^{2q}_G (X, \mathbb{R}), \]
where the evaluation map $\rho_\xi^T: S^p \mathfrak{t}^\vee \otimes H^{2q}_G (X, \mathbb{R}) \to H^{2q}_G (X, \mathbb{R})$ is defined similarly as (\ref{rho evaluation}). 
For $\alpha \in H^*_{T \times G} (X)$, we denote by $\alpha^{\langle p, q \rangle}$ the $H^{2p}_T (\mathrm{pt}) \otimes H^{2q}_G (B, \mathbb{R})$-component. 
Putting $\rho^T_{\hbar. \xi} = \rho_\xi^T \circ ({^\hbar \Phi_T^{\mathrm{pt}} \otimes \mathrm{id}})$, we have 
\[ \D^q_{\hbar. \xi} \alpha^{\langle k \rangle} = q! \cdot \rho_{\hbar. \xi}^T (\alpha^{\langle k-q, q \rangle}). \]


For an equivariant cohomology class of even degree $\alpha \in \hat{H}^{\mathrm{even}}_{T \times G} (B, \mathbb{R})$, we define the \textit{formal $q$-th derivative} $\hat{\D}^q \alpha \in \hat{H}^{\mathrm{even}}_T (\mathrm{pt}) \otimes H^{2k}_G (B, \mathbb{R})$ by 
\begin{equation} 
\hat{\D}^q \alpha := q! \sum_{p=0}^\infty \alpha^{\langle p, q \rangle} \in \hat{H}^{\mathrm{even}}_T (\mathrm{pt}) \otimes H^{2q}_G (B, \mathbb{R}). 
\end{equation}


\begin{defin}[Class $\varepsilon^\ell$]
\label{class epsilon ell}
For $\ell =0, 1, \ldots, \infty$, we say that $\alpha \in \hat{H}^{\mathrm{even}}_{T \times G} (B, \mathbb{R})$ is \textit{of class $\varepsilon^\ell$ around $0 \in \mathfrak{t}$} if for each $q = 0, \ldots, \ell$ (or $q < \infty$ when $\ell =\infty$) the infinite series 
\[ \sum_{k=0}^\infty \D_{\hbar. \xi}^q (\alpha^{\langle k \rangle}) = q! \cdot \sum_{p=0}^\infty \rho^T_{\hbar. \xi} (\alpha^{\langle p, q \rangle}) \] 
is uniformly absolutely convergent on a small neighbourhood of $0 \in \mathfrak{t}$ with respect to some (hence any) norm on $H^{2q}_G (B, \mathbb{R})$. 


In this case, the infinite series $\D_{\hbar. \xi}^q \alpha := \sum_{k=0}^\infty \D_{\hbar. \xi}^q (\alpha^{\langle k \rangle})$ gives a real analytic map from a small neighbourhood of $0 \in \mathfrak{t}$ to $H^{2q}_G (B, \mathbb{R})$ for every $q \le \ell$. 
We say that $\alpha$ is \textit{of class $\varepsilon^\ell$} if these real analytic maps extend to real analytic maps on the whole $\mathfrak{t}$. 
We denote by $\D_{\hbar. \xi}^q \alpha$ the value of this real analytic map at $\xi \in \mathfrak{t}$, for which the above infinite series may not converge. 
\end{defin}

We denote by 
\begin{equation}
\varepsilon^\ell_T (B_G) \subset \hat{H}^{\mathrm{even}}_{T \times G} (B, \mathbb{R})
\end{equation} 
the subset consisting of elements of class $\varepsilon^\ell$, which is a subring as we see later. 
The operator $\D^q_{\hbar. \xi}$ for $q \le \ell$ is defined on this class: 
\begin{equation} 
\D^q_{\hbar. \xi}: \varepsilon^\ell_T (B_G) \to H^{2q}_G (B, \mathbb{R}). 
\end{equation}

The following is clear from the definition. 

\begin{lem}
For $\alpha \in \hat{H}_{T \times G}^{\mathrm{even}} (B, \mathbb{R})$, a group homomorphism $\varphi: G' \to G$ and a $G'$-equivariant continuous map $f: B' \to B$, we have 
\begin{itemize}
\item $\hat{\D}^q (f^* \varphi^\# \alpha) = f^* \varphi^\# \hat{\D}^q \alpha$. 

\item For $\alpha \in \varepsilon^\ell_T (B_G)$, we have $f^* \varphi^\# \alpha \in \varepsilon^\ell_T (B'_{G'})$ and $\D_{\hbar. \xi}^q f^* \varphi^\# \alpha = f^* \varphi^\# \D_{\hbar. \xi}^q \alpha$. 
\end{itemize}
\end{lem}

We note the Chern--Weil isomorphism ${^\hbar \Phi}: \hat{H}^{\mathrm{even}}_{T \times \mathbb{G}_m} (\mathrm{pt}, \mathbb{R}) \xrightarrow{\sim} \hat{S} (\mathfrak{t} \times \mathfrak{u} (1))^\vee$ preserves the decompositions 
\begin{align*}
\hat{H}^{\mathrm{even}}_{T \times \mathbb{G}_m} (\mathrm{pt}, \mathbb{R}) 
&= \prod_{k=0}^\infty \bigoplus_{p+q=k} H^{2p}_T (\mathrm{pt}, \mathbb{R}) \otimes H^{2q}_{\mathbb{G}_m} (\mathrm{pt}, \mathbb{R}), 
\\
\hat{S} (\mathfrak{t} \times \mathfrak{u} (1))^\vee
&= \prod_{k=0}^\infty \bigoplus_{p+q=k} S^p \mathfrak{t}^\vee \otimes S^q \mathfrak{u} (1)^\vee. 
\end{align*}
The latter decomposition is related to Taylor expansion, so we can compare $\D^q_{\hbar. } \alpha$ for $\alpha \in \varepsilon^\ell_T (\mathrm{pt}_{\mathbb{G}_m}) \subset \hat{H}^{\mathrm{even}}_{T \times \mathbb{G}_m} (\mathrm{pt})$ with the usual derivative as follows. 

Let $V, W$ be finite dimensional vector spaces. 
We put 
\begin{equation}
\bar{\varepsilon}^\ell_V (W) := C^\omega (V, \bigoplus_{q=0}^\ell S^q W^\vee). 
\end{equation}
For an element $u = (u_0, \ldots, u_\ell)$, we denote its component $u_q \in C^\omega (V, S^q W^\vee)$ by $D_W^q u$. 
We identify $\bar{\varepsilon}^\ell_V (W)$ with a subspace of $C^\omega (V \times W)$ by identifying $u \in \bar{\varepsilon}^\ell_V (W)$ with the following function 
\[ u (v, w) := \sum_{q=0}^\ell \frac{1}{q!} \langle (D_W^q u) (v), w \rangle, \]
where $\langle -, w \rangle: S^q W^\vee \to \mathbb{R}$ denotes the substitution of $w \in W$. 
Under this identification, we have 
\[ \langle (D_W^q u) (v),  w \rangle = \frac{d^q}{dt^q}\Big{|}_{t=0} u (v, t w). \]
By the Taylor expansion at $0$, we also identify $\bar{\varepsilon}^\ell_V (W)$ with a subspace of $\hat{S} V^\vee \otimes \bigoplus_{q=0}^\ell S^q W^\vee$. 
Under this identification, $u \in \bar{\varepsilon}^\ell_V (W)$ is identified with 
\begin{equation} 
\sum_{q=0}^\ell \frac{1}{q!} \sum_{p=0}^\infty \frac{1}{p!} D_0^p D_W^q u, 
\end{equation}
where $D_0^p f \in S^p V^\vee \otimes W'$ for $f: V \to W'$ is given by $\langle D_0^p f, v \rangle = \frac{d^p}{dt^p}|_{t=0} f (tv)$. 

For $\hat{u} = \sum_{k=0}^\infty \sum_{p+q=k} u^{\langle p, q \rangle} \in \hat{S} (V \times W)^\vee$, we define a truncation $\hat{u}^{\langle , \le \ell \rangle} \in \hat{S} V^\vee \otimes \bigoplus_{q=0}^\ell S^q W^\vee$ by 
\[ \hat{u}^{\langle , \le \ell \rangle} := \sum_{q=0}^\ell \sum_{p=0}^\infty u^{\langle p, q \rangle}. \]
Now we put 
\[ \varepsilon^\ell_V (W) := \{ \hat{u} \in \hat{S} (V \times W)^\vee ~|~ \hat{u}^{\langle , \le \ell \rangle} \in \bar{\varepsilon}^\ell_V (W) \}. \]

\begin{prop}
\label{derivative}
The Chern--Weil isomorphism ${^\hbar \Phi}: \hat{H}_{T \times \mathbb{G}_m} (\mathrm{pt}) \to \hat{S} (\mathfrak{t} \times \mathfrak{u} (1))^\vee$ induces an isomorphism 
\[ \varepsilon^\ell_T (\mathrm{pt}_{\mathbb{G}_m}) \xrightarrow{\sim} \varepsilon^\ell_{\mathfrak{t}} (\mathfrak{u} (1)). \]
For $\alpha \in \varepsilon^\ell_T (\mathrm{pt}_{\mathbb{G}_m})$ and $\xi \in \mathfrak{t}$, $\zeta \in \mathfrak{u} (1)$, we have 
\[ \langle {^\hbar \D^q_{\hbar. \xi} \alpha}, \zeta \rangle = \frac{d^q}{dt^q}\Big{|}_{t=0} ({^\hbar \alpha})^{\langle, \le \ell \rangle} (\xi, t \zeta) \]
for $q = 0, \ldots, \ell$. 
\end{prop}

\begin{proof}
For $\alpha \in \varepsilon^\ell_T (\mathrm{pt}_{\mathbb{G}_m})$ and $q=0, \ldots, \ell$, we have 
\[ \langle {^\hbar \D_{\hbar. \xi}^q \alpha}, \zeta \rangle = \langle q! \sum_{p=0}^\infty {^\hbar \rho^T_\xi} \alpha^{\langle p, q \rangle}, \zeta \rangle = q! \sum_{p=0}^\infty {^\hbar \alpha^{\langle p, q \rangle}} (\xi, \zeta) \]
around $\xi = 0$. 

On the other hand, since 
\[ ({^\hbar \alpha})^{\langle, \le \ell \rangle} (\xi, t\zeta) = \sum_{q=0}^\ell \sum_{p=0}^\infty {^\hbar \alpha^{\langle p, q \rangle}} (\xi, \zeta). t^q = \sum_{q=0}^\ell \frac{1}{q!} \langle {^\hbar \D_{\hbar. \xi}^q \alpha}. t^q, \zeta \rangle, \]
we have $({^\hbar \alpha})^{\langle, \le \ell \rangle} \in \bar{\varepsilon}^\ell_{\mathfrak{t}} (\mathfrak{u} (1))$ by $\alpha \in \varepsilon^\ell_T (\mathrm{pt}_{\mathbb{G}_m})$, hence ${^\hbar \alpha} \in \varepsilon^\ell_{\mathfrak{t}} (\mathfrak{u} (1))$. 
By this equality, we obtain 
\[ \frac{d^q}{dt^q} \Big{|}_{t=0} ({^\hbar \alpha})^{\langle, \le \ell \rangle} (\xi, t\zeta) =\langle {^\hbar \D_{\hbar. \xi}^q \alpha}, \zeta \rangle. \]
\end{proof}

\subsubsection{Leibniz rule and chain rule}

We prepare elementary formulas on derivatives. 

\begin{prop}[Leibniz rule]
\label{Leibniz rule}
We have the following formal Leibniz rule for $\alpha, \beta \in \hat{H}^{\mathrm{even}}_{T \times G} (B, \mathbb{R})$: 
\begin{equation} 
\hat{\D}^q (\alpha \smile \beta) = \sum_{l=0}^q \binom{q}{l} \hat{\D}^l \alpha \smile \hat{\D}^{q-l} \beta \in \hat{H}_T (\mathrm{pt}) \otimes H^{2q}_G (B, \mathbb{R}). 
\end{equation}
Suppose $\alpha$ and $\beta$ are of class $\varepsilon^\ell$, then so is $\alpha \smile \beta$. 
In particular, $\varepsilon^\ell_T (B_G) \subset \hat{H}^{\mathrm{even}}_{T \times G} (B, \mathbb{R})$ is a subring. 
In this case, we have 
\begin{equation} 
\D^q_{\hbar. \xi} (\alpha \smile \beta) = \sum_{l=0}^q \binom{q}{l} \D^l_{\hbar. \xi} \alpha \smile \D^{q-l}_{\hbar. \xi} \beta \in H^{2q}_G (B, \mathbb{R}). 
\end{equation}
\end{prop}

\begin{proof}
The first claim follows directly by 
\begin{align*} 
\hat{\D}^q (\alpha \smile \beta) 
&= q! \sum_{p=0}^\infty (\alpha \smile \beta)^{\langle p, q \rangle} = q! \sum_{p=0}^\infty \sum_{r+s=p} \sum_{l=0}^q \alpha^{\langle r, l \rangle} \smile \beta^{\langle s, q-l \rangle} 
\\
&= \sum_{l=0}^q q! \Big{(} \sum_{r=0}^\infty \alpha^{\langle r, l \rangle} \Big{)} \smile \Big{(} \sum_{s=0}^\infty \beta^{\langle s, q-l \rangle} \Big{)} 
\\
&= \sum_{l=0}^q \frac{q!}{l! (q-l)!} \Big{(} l! \sum_{r=0}^\infty \alpha^{\langle r, l \rangle} \Big{)} \smile \Big{(} (q-l)! \sum_{s=0}^\infty \beta^{\langle s, q-l \rangle} \Big{)}
\\
&= \sum_{l=0}^q \binom{q}{l} \hat{\D}^l \alpha \smile \hat{\D}^{q-l} \beta. 
\end{align*}
The second claim on the absolute convergence follows by Tonelli's theorem as follows. 
Firstly, we take a collection of norms $\{ \| \cdot \|_l \}_{l \le q}$ on $\{ H^{2l}_G (B, \mathbb{R}) \}_{l \le q}$ so that $\| u \smile v \|_q \le \| u \|_l \cdot \| v \|_{q-l}$ for every $u \in H^{2l}_G (B, \mathbb{R})$ and $v \in H^{2(q-l)}_G (B, \mathbb{R})$. 
Then since 
\[ \| {^\hbar \rho^T_\xi} (\alpha^{\langle r, l \rangle} \smile \beta^{\langle s, q-l \rangle}) \|_q \le \| {^\hbar \rho^T_\xi} \alpha^{\langle r, l \rangle} \|_l \cdot \| {^\hbar \rho^T_\xi} \beta^{\langle s, q-l \rangle} \|_{q-l}, \]
we have 
\begin{align*} 
\sum_{p=0}^\infty \| {^\hbar \rho^T_\xi} (\alpha \smile \beta)^{\langle p, q \rangle} \|_q 
&\le \sum_{p=0}^\infty \sum_{r+s=p} \sum_{l=0}^q \| {^\hbar \rho^T_\xi} (\alpha^{\langle r, l \rangle} \smile \beta^{\langle s, q-l \rangle}) \|_q 
\\
&\le \sum_{l=0}^q \sum_{p=0}^\infty \sum_{r+s=p} \| {^\hbar \rho^T_\xi} \alpha^{\langle r, l \rangle} \|_l \cdot \| {^\hbar \rho^T_\xi} \beta^{\langle s, q-l \rangle} \|_{q-l}
\\
&= \sum_{l=0}^q \Big{(} \sum_{r=0}^\infty \| {^\hbar \rho^T_\xi} \alpha^{\langle r, l \rangle} \|_l \Big{)} \Big{(} \sum_{s=0}^\infty \| {^\hbar \rho^T_\xi} \beta^{\langle s, q-l \rangle} \|_{q-l} \Big{)}. 
\end{align*}
\end{proof}

For $k, l \ge 1$, we put 
\[ \Delta^l (k) := \{ \bm{k} = (k_1, \ldots, k_l) \in \mathbb{N}^l ~|~  k_1 + \dotsb + k_l = k \}, \]
where $\mathbb{N} = \{ 1, 2, \ldots \}$. 
We have $\Delta^l (k) = \emptyset$ for $k < l$ and $\Delta^k (k) = \{ (1, 1, \ldots, 1) \}$. 
For $\alpha \in \hat{H}^{\mathrm{even}}_{T \times G} (B, \mathbb{R})$ and for $\bm{k} \in \Delta^l (k)$, we put 
\[ \alpha^{\langle \bm{k} \rangle} := \alpha^{\langle k_1 \rangle} \smile \dotsb \smile \alpha^{\langle k_l \rangle} \in H^{2k}_{T \times G} (B, \mathbb{R}). \]
Then for $\alpha \in \hat{H}^{\mathrm{even}}_{T \times G} (B, \mathbb{R})$ and $l \ge 1$, we have 
\[ (\alpha - \alpha^{\langle 0 \rangle})^l = \Big{(} \sum_{k=1}^\infty \alpha^{\langle k \rangle} \Big{)}^l = \sum_{k=l}^\infty \sum_{\bm{k} \in \Delta^l (k)} \alpha^{\langle \bm{k} \rangle}. \]
In particular, we have 
\[ ((\alpha - \alpha^{\langle 0 \rangle})^l )^{\langle k \rangle} = 
\begin{cases} 
\sum_{\bm{k} \in \Delta^l (k)} \alpha^{\langle \bm{k} \rangle} 
& k \ge l
\\
0 
& k < l
\end{cases} \]
for $k, l \ge 1$. 

Let $h$ be a real analytic function on a neighbourhood $U$ of $\alpha^{\langle 0 \rangle} \in \mathbb{R}$. 
Let $h (x) = \sum_{l=0}^\infty a_l (x- \alpha^{\langle 0 \rangle})^l$ be the Taylor series. 
We define $h (\alpha) \in \hat{H}^{\mathrm{even}}_{T \times G} (B, \mathbb{R})$ by 
\begin{equation} 
\label{formal series of equivariant cohomology class}
h (\alpha)^{\langle k \rangle} := 
\begin{cases}
a_0
& k = 0
\\
\sum_{l=1}^k a_l \sum_{\bm{k} \in \Delta^l (k)} \alpha^{\langle \bm{k} \rangle}
& k \ge 1
\end{cases}. 
\end{equation}
Putting 
\[ \sum_{\bm{k} \in \Delta^0 (k)} \alpha^{\langle \bm{k} \rangle} = 
\begin{cases} 
1 
& k= 0 
\\ 
0 
& k = 1, 2, \ldots 
\end{cases},  \]
we have 
\[ h (\alpha)^{\langle k \rangle} = \sum_{l=0}^k a_l \sum_{\bm{k} \in \Delta^l (k)} \alpha^{\langle \bm{k} \rangle}. \]

We are interested in $h (x) = x^{-1}$ on $U = \mathbb{R} \setminus \{ 0 \}$ and $h (x) = \log x$ on $U = (0, \infty)$. 
At $x= a$, the Taylor series are given by  
\begin{align*} 
x^{-1} 
&= \sum_{l=0}^\infty - \Big{(} \frac{-1}{a} \Big{)}^{l+1} (x - a)^l, 
\\
\log x 
&= \log a + \sum_{l=1}^\infty -\frac{1}{l} \Big{(} \frac{-1}{a} \Big{)}^l (x-a)^l 
\end{align*}

Since the K\"unneth decomposition (\ref{Kunneth}) is preserved by pullback, we have the following. 

\begin{lem}
For $\alpha \in \hat{H}_{T \times G}^{\mathrm{even}} (B, \mathbb{R})$, a group homomorphism $\varphi: G' \to G$ and a $G'$-equivariant continuous map $f: B' \to B$ equivariant with respect to $\varphi$, we have $f^* \varphi^\# h (\alpha) = h (f^* \varphi^\# \alpha)$. 
\end{lem}

\begin{prop}[Chain rule]
\label{chain rule}
We have 
\begin{align*} 
\hat{\D}^0 h (\alpha) 
&= h (\hat{\D}^0 \alpha) \in \hat{H}_T (\mathrm{pt}, \mathbb{R}), 
\\ 
\hat{\D} h (\alpha) 
&= h' (\hat{\D}^0 \alpha) \cdot \hat{\D} \alpha \in \hat{H}_T (\mathrm{pt}) \otimes H^2_G (B, \mathbb{R}). 
\end{align*}
Suppose $\alpha$ is of class $\varepsilon^\ell$ and $\D^0_{\hbar. \xi} \alpha \in U \subset \mathbb{R}$ for every $\xi$, then $h (\alpha)$ is of class $\varepsilon^\ell$ and we have 
\begin{align*} 
\D_{\hbar. \xi}^0 h (\alpha) 
&= h (\D^0_{\hbar. \xi} \alpha) \in \mathbb{R}, 
\\ 
\D_{\hbar. \xi} h (\alpha) 
&= h' (\D^0_{\hbar. \xi} \alpha) \cdot \D_{\hbar. \xi} \alpha \in H^2_G (B, \mathbb{R}). 
\end{align*}
\end{prop}

\begin{proof}
Since $(\alpha_{T \times G}^{\langle \bm{k} \rangle})^{\langle k, 0 \rangle} = \alpha_{T \times G}^{\langle \bm{k}, 0 \rangle} = (\hat{\D}^0 \alpha)_T^{\langle \bm{k} \rangle} \in H^{2k}_T (\mathrm{pt})$, we compute 
\begin{align*}
\hat{\D}^0 h (\alpha) 
&= \sum_{k=0}^\infty h (\alpha)^{\langle k, 0 \rangle} = \sum_{k=0}^\infty \Big{(} \sum_{l=0}^k a_l \sum_{\bm{k} \in \Delta^l (k)} \alpha^{\langle \bm{k} \rangle} \Big{)}^{\langle k, 0 \rangle}
\\
&= \sum_{k=0}^\infty \Big{(} \sum_{l=0}^k a_l \sum_{\bm{k} \in \Delta^l (k)} (\hat{\D}^0 \alpha)^{\langle \bm{k} \rangle} \Big{)} = h (\hat{\D}^0 \alpha)
\end{align*}
in the ring $\hat{H}^{\mathrm{even}}_T (\mathrm{pt}, \mathbb{R})$. 

Similarly, putting 
\[ \sum_{\bm{p} \in \Delta^0 (p)} \alpha^{\langle \bm{p}, 0 \rangle} = 
\begin{cases} 
1 
& p= 0 
\\ 
0 
& p = 1, 2, \ldots 
\end{cases},  \]
we have 
\[ \Big{(} \sum_{\bm{k} \in \Delta^l (k)} \alpha^{\langle \bm{k} \rangle} \Big{)}^{\langle k-1, 1 \rangle} = 
\begin{cases}
l \sum_{\substack{p+r=k-1 \\ p \ge l-1 }} \alpha^{\langle r, 1 \rangle} \Big{(} \sum_{\bm{p} \in \Delta^{l-1} (p)} \alpha^{\langle \bm{p}, 0 \rangle} \Big{)}
& k \ge l
\\
0
& k <  l
\end{cases} \]
in $H^{2k-2}_T (\mathrm{pt}) \otimes H^2_G (B)$ for $k, l \ge 1$, so that 
\[ h (\alpha)^{\langle k-1, 1 \rangle} = \sum_{l=1}^k l a_l \left( \sum_{\substack{p+r=k-1 \\ p \ge l-1}} \alpha^{\langle r, 1 \rangle} \Big{(} \sum_{\bm{p} \in \Delta^{l-1} (p)} \alpha^{\langle \bm{p}, 0 \rangle} \Big{)} \right). \]
Since $h' (x) = \sum_{l=1}^\infty l a_l (x - \alpha^{\langle 0 \rangle})^{l-1}$, we compute 
\begin{align*}
\hat{\D} h (\alpha) 
&= \sum_{k=1}^\infty h (\alpha)^{\langle k-1, 1 \rangle} = \sum_{k=1}^\infty \sum_{l=1}^k l a_l \left( \sum_{\substack{p+r=k-1 \\ p \ge l-1}} \alpha^{\langle r, 1 \rangle} \Big{(} \sum_{\bm{i} \in \Delta^{l-1} (p)} \alpha^{\langle \bm{p}, 0 \rangle} \Big{)} \right)
\\
&= \sum_{l=1}^\infty l a_l \Big{(} \sum_{r=0}^\infty \alpha^{\langle r, 1 \rangle} \Big{)} \Big{(} \sum_{p=l-1}^\infty \sum_{\bm{i} \in \Delta^{l-1} (p)} (\hat{\D}^0 \alpha)^{\langle \bm{p} \rangle} \Big{)}
\\
&= \left( \sum_{p=0}^\infty \Big{(} \sum_{l=1}^{p+1} l a_l \sum_{\bm{p} \in \Delta^{l-1} (p)} (\hat{\D}^0 \alpha)^{\langle \bm{p} \rangle} \Big{)} \right) \Big{(} \sum_{r=0}^\infty \alpha^{\langle r, 1 \rangle} \Big{)}
\\
&= h' (\hat{\D}^0 \alpha) \cdot \hat{\D} \alpha. 
\end{align*}
in $\hat{H}^{\mathrm{even}}_T (\mathrm{pt}) \otimes H^2_G (B)$. 

Now suppose $\alpha$ is of class $\varepsilon^0$ and $\D^0_{\hbar. \xi} \alpha \in U$ for every $\xi \in \mathfrak{t}$. 
To see that $h (\alpha)$ is of class $\varepsilon^0$, we show the following infinite series is uniformly absolute-convergent to $h (\D^0_{\hbar. \xi} \alpha)$ around the origin: 
\[ \sum_{k=0}^\infty {^\hbar \rho^T_\xi} \Big{(} \sum_{l=0}^k a_l \sum_{\bm{k} \in \Delta^l (k)} \alpha^{\langle \bm{k}, 0 \rangle} \Big{)}. \]
Once the convergence is proved around the origin, the limit real analytic function $\D^0_{\hbar. \xi} h (\alpha)$ around the origin extends to the real analytic function $h (\D^0_{\hbar. \xi} \alpha)$ defined on the whole $\mathfrak{t}$, so that we conclude $h (\alpha)$ is of $\varepsilon^0$. 

Since $h$ is real analytic and $\alpha^{\langle 0 \rangle} \in U$, the infinite series $\sum_{l=0}^\infty a_l x^l$ is uniformly absolute-convergent to $h (x - \alpha^{\langle 0 \rangle})$ around $x = 0$. 
Since $\alpha$ is of class $\varepsilon^0$, the infinite series $\sum_{i=0}^\infty {^\hbar \rho^T_\xi} \alpha^{\langle i, 0 \rangle}$ is uniformly absolute-convergent to $\D_{\hbar. \xi} \alpha$ around $\xi = 0$. 
Then for $\xi$ sufficiently close to $0$, we compute 
\begin{align*} 
\sum_{k=0}^\infty \Big{|} {^\hbar \rho^T_\xi} \sum_{l=0}^k a_l \sum_{\bm{k} \in \Delta^l (k)} \alpha^{\langle \bm{k}, 0 \rangle} \Big{|} 
&\le \sum_{k=0}^\infty \sum_{l=0}^k |a_l| \sum_{\bm{k} \in \Delta^l (k)} |{^\hbar \rho^T_\xi} \alpha^{\langle \bm{k}, 0 \rangle}| 
\\
&= \sum_{l=0}^\infty |a_l| \sum_{k=0}^\infty \sum_{\bm{k} \in \Delta^l (k)} |{^\hbar \rho^T_\xi} \alpha^{\langle \bm{k}, 0 \rangle}|
\\
&= \sum_{l=0}^\infty |a_l| \Big{|} \sum_{k=0}^\infty {^\hbar \rho^T_\xi} \alpha^{\langle k, 0 \rangle} \Big{|}^l 
\end{align*}
by Tonelli theorem, which shows the uniform absolute-convergence. 

Now for $k, l \in \mathbb{N}_{\ge 0} = \{ 0, 1, 2, \ldots \}$, we put 
\[ h_\xi (k, l) := 
\begin{cases} 
a_l \sum_{\bm{k} \in \Delta^l (k)} {^\hbar \rho^T_\xi} \alpha^{\langle \bm{k}, 0 \rangle}
& l \le k 
\\ 
0 
& l > k
\end{cases}. \] 
The above argument shows that for small $\xi$, $h_\xi (k, l)$ is integrable on $\mathbb{N}_{\ge 0}^2$. 
Thus we get 
\[ \sum_{k=0}^\infty {^\hbar \rho^T_\xi} \Big{(} \sum_{l=0}^k a_l \sum_{\bm{k} \in \Delta^l (k)} \alpha^{\langle \bm{k}, 0 \rangle} \Big{)} = \sum_{k=0}^\infty \sum_{l=0}^\infty h_\xi (k, l) = \sum_{l=0}^\infty \sum_{k=0}^\infty h_\xi (k, l) = h (\D^0_{\hbar. \xi} \alpha) \]
by Fubini's theorem. 

Next, suppose $\alpha$ is of class $\varepsilon^1$. 
Then we similarly compute 
\begin{align*} 
\sum_{k=1}^\infty 
&\Big{\|} {^\hbar \rho^T_\xi} \sum_{l=1}^k l a_l \Big{(} \sum_{\substack{ p+r=k-1 \\ p \ge l-1 }} \alpha^{\langle r, 1 \rangle} \Big{(} \sum_{\bm{p} \in \Delta^{l-1} (p)} \alpha^{\langle \bm{p}, 0 \rangle} \Big{)} \Big{)} \Big{\|} 
\\
&\le \sum_{k=1}^\infty \sum_{l=1}^k l |a_l| \Big{(} \sum_{\substack{p+r=k-1 \\ p \ge l-1 }} \| {^\hbar \rho^T_\xi} \alpha^{\langle r, 1 \rangle} \| \Big{(} \sum_{\bm{p} \in \Delta^{l-1} (p)} | {^\hbar \rho^T_\xi} \alpha^{\langle \bm{p}, 0 \rangle}| \Big{)} \Big{)}
\\
&= \sum_{l=1}^\infty l |a_l| \Big{(} \sum_{r=0}^\infty \| {^\hbar \rho^T_\xi} \alpha^{\langle r, 1 \rangle} \| \Big{)} \Big{(} \sum_{p=l-1}^\infty \sum_{\bm{p} \in \Delta^{l-1} (p)}  | {^\hbar \rho^T_\xi} \alpha^{\langle \bm{p}, 0 \rangle}| \Big{)}
\\
&= \left( \sum_{l=1}^\infty l |a_l| \Big{(} \sum_{p=1}^\infty | {^\hbar \rho^T_\xi} \alpha^{\langle p, 0 \rangle}| \Big{)}^{l-1} \right) \Big{(} \sum_{r=0}^\infty \| {^\hbar \rho^T_\xi} \alpha^{\langle r, 1 \rangle} \| \Big{)}
\end{align*}
by Tonelli theorem. 
Then similarly as above, by Fubini's theorem, we get 
\[ \sum_{k=1}^\infty {^\hbar \rho^T_\xi} \sum_{l=1}^k l a_l \Big{(} \sum_{\substack{ p+r=k-1 \\ p \ge l-1 }} \alpha^{\langle r, 1 \rangle} \Big{(} \sum_{\bm{p} \in \Delta^{l-1} (p)} \alpha^{\langle \bm{p}, 0 \rangle} \Big{)} \Big{)} = h' (\D_{\hbar. \xi}^0 \alpha) \cdot \D_{\hbar. \xi} \alpha \]
around $\xi = 0$, hence $h (\alpha)$ is of class $\varepsilon^1$ if $\D^0_{\hbar. \xi} \alpha \in U$ for every $\xi$. 
\end{proof}

\subsection{Equivariant calculus}
\label{equivariant calculus}

\subsubsection{On the absolute equivariant intersection $(\alpha. e^L)$}
\label{absolute equivariant intersection theory}

We firstly observe the case $B = \mathrm{pt}$ and $G = \{ 1 \}$. 
In this case, the K\"unneth decomposition (\ref{Kunneth}) is nothing but the identification $\hat{H}^{\mathrm{even}}_{T \times \{ 1 \}} (\mathrm{pt}, \mathbb{R}) = \prod_{p=0}^\infty H^{2p}_T (\mathrm{pt}, \mathbb{R})$, so that we have $\hat{u}^{\langle p, q \rangle} = 0$ for $\hat{u} \in \hat{H}^{\mathrm{even}}_{T \times \{ 1 \}} (\mathrm{pt}, \mathbb{R})$ and $q \ge 1$. 
Thus we have $\hat{\D}^q \hat{u} = 0$ for $q \ge 1$, $\hat{\D}^0 \hat{u} = \hat{u}$ and 
\[ \D^0_{\hbar. \xi} \hat{u} = \sum_{k=0}^\infty {^\hbar \hat{u}^{\langle k \rangle}} (\xi). \]

Let $X$ be a pure $n$-dimensional scheme with a $T$-action. 
For an equivariant locally finite homology class $\alpha \in H^{\mathrm{lf}, T}_{2 p} (X, \mathbb{R})$ and a second equivariant cohomology class $L \in H^2_T (X, \mathbb{R})$, we study the convergence of the above infinite series for $\hat{u} = (\alpha. e^L)$. 

We prepare the following lemma. 
The proof works also for complex spaces. 

\begin{lem}
\label{div-surj}
Let $X$ and $\tilde{X}$ be pure $n$-dimensional $G$-schemes and $f: \tilde{X} \to X$ be a proper surjective $G$-equivariant morphism. 
Then the pushforward maps $f_*: H^{\mathrm{alg}, G}_{2n-2} (\tilde{X}, \mathbb{R}) \to H^{\mathrm{alg}, G}_{2n-2} (X, \mathbb{R})$ and $f_*: H^{\mathrm{lf}, G}_{2n} (\tilde{X}, \mathbb{R}) \to H^{\mathrm{lf}, G}_{2n} (X, \mathbb{R})$ are surjective. 
\end{lem}

\begin{proof}
Since the induced morphism $\id \times_G f: E_2 G \times_G \tilde{X} \to E_2 G \times_G X$ is also proper surjective, the claim reduces to the trivial case $G= \{ 1 \}$. 
For each irreducible component $X_i$ of $X$, there is an irreducible component $\tilde{X}_i$ of $\tilde{X}$ with $f (\tilde{X}_i) = X_i$, so we may assume that $X$ and $\tilde{X}$ are irreducible. 
For any prime divisor $E \subset X$, there is a prime divisor component $F$ of $f^{-1} (E) \subsetneq \tilde{X}$ whose image is $E$. 
(If not, then since $f (f^{-1} (E)) = \cup_{F_i \in \mathrm{Irr} (f^{-1} (E))} f (F_i)$ are the union of $(n-2)$-dimensional subvarieties, we get $f (f^{-1} (E)) \neq E$, which contradicts to the surjectivity. )
Then since $f_* F = m. E$ for some integer $m > 0$, we proved the claim for $H^{\mathrm{alg}, G}_{2n-2}$ with $\mathbb{Q}$-coefficient. 
The claim for $H^{\mathrm{lf}, G}_{2n}$ follows by $H^{\mathrm{lf}, G}_{2n} (X, \mathbb{Q}) = \bigoplus_i \mathbb{Q} [X_i]$. 
\end{proof}

Let $f$ be a real analytic function on $\mathbb{R}$ whose Taylor series $\sum_{k=0}^\infty \frac{a_k}{k!} x^k$ is absolutely convergent for every $x \in \mathbb{R}$. 
We call such $f$ \textit{entire real analytic}. 

Now we prove the following. 

\begin{prop}
\label{absolute intersection}
Let $X$ be a pure $n$-dimensional proper scheme with an algebraic $T$-action and $L = L_T \in H^2_T (X, \mathbb{R})$ be a second $T$-equivariant cohomology class on $X$. 
If $\alpha \in H^{\mathrm{lf}, T}_{2n} (X, \mathbb{R})$ or $\alpha \in H^{\mathrm{alg}, T}_{2n-2} (X, \mathbb{R})$, the infinite series 
\[ {^\hbar (\alpha. f (L_T); \xi)} = \sum_{k=0}^\infty \frac{a_k}{k!} {^\hbar (\alpha. L_T^{\cdot k}; \xi)} \in \mathbb{R} \quad  \Big{(} = \D^0_{\hbar. \xi} (\alpha. f (L_T)) \Big{)} \]
is compactly absolutely convergent on $\mathfrak{t}$. 

Suppose $L$ is semiample and big and the $n$-th derivative $f^{(n)}$ is positive. 
Then ${^\hbar (f (L); \xi)} = {^\hbar ([X]. f (L); \xi)}$ is positive. 
\end{prop}

\begin{proof}
Take the irreducible decomposition of the reduction $X^{\mathrm{irr}} \to X^{\mathrm{red}} \to X$ and a $T$-equivariant resolution of singularities $\tilde{X} \to X^{\mathrm{irr}}$. 
Let $\beta: \tilde{X} \to X$ be the composition of these morphisms. 
By Lemma \ref{div-surj}, we can pick an element $\tilde{\alpha} \in H^{\mathrm{alg}, T}_{2n-2} (\tilde{X}, \mathbb{R})$ (resp. $\in H^{\mathrm{lf}, T}_{2n} (\tilde{X}, \mathbb{R})$) with $\beta_* \tilde{\alpha} = \alpha$ for $\alpha \in H^{\mathrm{alg}, T}_{2n-2} (X, \mathbb{R})$ (resp. $\in H^{\mathrm{lf}, T}_{2n} (X, \mathbb{R})$). 
By the projection formula, we have $(\alpha. e^L) = (\tilde{\alpha}. e^{\beta^* L})$ in $\hat{H}_T (\mathrm{pt}, \mathbb{R})$, so that we may assume $X$ is smooth. 

Suppose $\alpha \in H^{\mathrm{lf}, T}_{2n-2} (X, \mathbb{R})$. 
Pick equivariant $2$-forms $A + \hbar \nu$ in ${^\hbar \mathrm{PD}^X_T (\alpha)} \in {^\hbar H^2_{\mathrm{dR}, T_{\mathrm{cpt}}}} (X)$ and $\Omega + \hbar \mu$ in ${^\hbar L} \in {^\hbar H^2_{\mathrm{dR}, T_{\mathrm{cpt}}}} (X)$. 
We compute 
\begin{align*} 
{^\hbar  (\alpha. L^{\cdot k}; \xi)} 
&= \int_X (A + \hbar \nu_\xi) (\Omega + \hbar \mu_\xi)^k
\\
&= \int_X \Big{(} \binom{k}{n-1} \mu_{\hbar \xi}^{k-(n-1)} A \wedge \Omega^{n-1} + \binom{k}{n} \nu_{\hbar \xi} \mu_{\hbar \xi}^{k-n} \Omega^n \Big{)}, 
\end{align*}
where we put $\binom{k}{l} = 0$ when $k < l$. 
Now recall 
\[ f^{(l)} (x) = \frac{d^l}{dx^l} f (x) = \sum_{k=l}^\infty \frac{a_k}{(k-l)!} x^{k-l} = \sum_{j=0}^\infty \frac{a_{l+j}}{j!} x^j. \]
It follows that the infinite series of $2n$-forms on $X$
\[ \sum_{k=0}^\infty \frac{a_k}{k!} \Big{(} \binom{k}{n-1} \mu_{\hbar \xi}^{k-(n-1)} A \wedge \Omega^{n-1} + \binom{k}{n} \nu_{\hbar \xi} \mu_{\hbar \xi}^{k-n} \Omega^n \Big{)} \]
is compactly absolutely convergent on $\mathfrak{t}$ with respect to the $C^l$-norm for every $l \in \mathbb{Z}_{\ge 0}$, to the limit $2n$-form 
\[ \frac{1}{(n-1)!} f^{(n-1)} (\mu_{\hbar \xi}) A \wedge \Omega^{n-1} + \frac{1}{n!} \nu_{\hbar \xi} f^{(n)} (\mu_{\hbar \xi}) \Omega^n. \]
It follows that by Corollary \ref{continuity of push} and Lemma \ref{Frechet to Banach}, the infinite series $\sum_{k=0}^\infty \frac{a_k}{k!} {^\hbar (\alpha. L^{\cdot k}; \xi)}$ is compactly absolutely convergent to the integral 
\[ \int_X \Big{(} \frac{1}{(n-1)!} f^{(n-1)} (\mu_{\hbar \xi}) A \wedge \Omega^{n-1} + \frac{1}{n!} \nu_{\hbar \xi} f^{(n)} (\mu_{\hbar \xi}) \Omega^n \Big{)}, \]
which shows the claim on the convergence. 
If we put 
\[ f (\Omega + \hbar \mu_\xi) = \sum_{k=0}^\infty \frac{a_k}{k!} (\Omega + \hbar \mu_\xi)^k = \sum_{k=0}^\infty \frac{a_k}{k!} \sum_{l=0}^k \binom{k}{l} \mu_{\hbar \xi}^{k-l} \Omega^l = \sum_{l=0}^n \frac{f^{(l)} (\mu_{\hbar \xi})}{l!} \Omega^l, \]
we can express the above integral as 
\[ \int_X (A + \hbar \nu_\xi) f (\Omega + \hbar \mu_\xi). \]
We obtain the claim for $\alpha \in H^{\mathrm{lf}, T}_{2n} (X, \mathbb{R})$ in a similar way. 

To see the positivity, we may assume the reduction $X^{\mathrm{red}}$ is a disjoint union of smooth varieties by a similar argument. 
If $L$ is semiample and big, then we can pick a $2$-form $\Omega$ so that it is semipositive and is strictly positive on some open set of $X$. 
Since $f^{(n)} (\mu_{\hbar \xi})$ is positive, 
\begin{align*} 
{^\hbar (\alpha. e^L; \xi)} 
&= \sum_i \mathrm{PD}^{X^{\mathrm{red}}_i}_T (\alpha|_{X^{\mathrm{red}}_i}) \int_{X^{\mathrm{red}}_i} f (\Omega|_{X^{\mathrm{red}}_i} + \hbar \mu_\xi|_{X^{\mathrm{red}}_i}) 
\\
&= \sum_i \mathrm{PD}^{X^{\mathrm{red}}_i}_T (\alpha|_{X_i^{\mathrm{red}}}) \int_{X^{\mathrm{red}}_i} f^{(n)} (\mu_{\hbar \xi}|_{X^{\mathrm{red}}}) \frac{ \Omega|_{X^{\mathrm{red}}_i}^n }{n!}
\end{align*}
is positive when the Poincare duals $\mathrm{PD}^{X^{\mathrm{red}}_i}_T (\alpha|_{X^{\mathrm{red}}_i}) \in \mathbb{R} \cong H^0_T (X^{\mathrm{red}}_i, \mathbb{R})$ are positive for each irreducible component. 
Here $X^{\mathrm{red}}_i$ are the connected components of the smooth $X^{\mathrm{red}}$. 
This proves the claim on the positivity for $\alpha = [X]$ as the Poincare dual $([X^{\mathrm{red}}]^T \frown)^{-1} ([X]^T)$ of the fundamental class of an irreducible scheme $X$ is the length of $\mathcal{O}_X/\mathcal{O}_{X^{\mathrm{red}}}$ at the generic point and hence is positive. 
\end{proof}

In particular, ${^\hbar (e^L; \xi)}$ is positive. 
It follows that the equivariant cohomology classes $(\alpha. e^{L_T}), (e^{L_T})^{-1}$ and $\log (e^{L_T})$ are of class $\varepsilon^\infty$ by Proposition \ref{chain rule}. 

\begin{rem}
When $X$ is a complex space, we have the same result for $\alpha$ associated to an analytic divisor (as the above lemma holds even in the analytic case), or more generally, for a homology class which can be written as the pushforward of some homology class of a resolution $X' \to X$. 
\end{rem}

We can compare the equivariant self intersection of $L$ with an integration with respect to the Duistermaat--Heckman measure. 
This is well-known for smooth $X$ and is a consequence of the equivariant Riemann--Roch theorem for general $X$. 

\begin{prop}
\label{equivariant intersection and DH measure}
For a $T$-equivariant semiample and big $\mathbb{Q}$-line bundle $L$ on a proper scheme $X$, we have 
\begin{equation}
{^\hbar (e^L; \xi)} = \int_{\mathfrak{t}^\vee} e^{- \langle x, \hbar \xi \rangle} \mathrm{DH}_{(X, L_T)} (x). 
\end{equation}
\end{prop}

\begin{proof}
Since 
\begin{align*} 
{^\hbar (\mathrm{PD}^{\mathrm{pt}}_T \tau_{\mathrm{pt}}^T (p_! L^{\otimes k})) }
&= {^\hbar \mathrm{ch}_T (p_! L^{\otimes k}) }
\\
&= \sum_{\chi \in M} \dim H^0 (X, L^{\otimes k})_\chi e^{{^\hbar c_{T, 1}} (\mathbb{C}_\chi)} 
\\
&= \sum_{\chi \in M} \dim H^0 (X, L^{\otimes k})_\chi e^{- \hbar \chi} \in \hat{S} \mathfrak{t}^\vee, 
\end{align*}
we have 
\[ {^\hbar (\mathrm{PD}^{\mathrm{pt}}_T \tau_{\mathrm{pt}}^T (p_! L^{\otimes k}))} (k^{-1} \xi) = \sum_{m \in M} \dim H^0 (X, L^{\otimes k})_m e^{- \langle k^{-1} m, \hbar \xi \rangle} = k^n \int_{\mathfrak{t}^\vee} e^{-\langle x, \hbar \xi \rangle} \nu_k (x) \]
for the measure 
\[ \nu_k := \frac{1}{k!} \sum_{m \in M} \dim H^0 (X, L^{\otimes k})_m. \delta_{k^{-1} m} \]
on $\mathfrak{t}^\vee$. 
Thus we get 
\[ \int_{\mathfrak{t}^\vee} e^{- \langle x, \hbar \xi \rangle} \nu_k (x) = \lim_{k \to \infty} k^{-n} \cdot {^\hbar ( \mathrm{PD}^{\mathrm{pt}}_T \tau_{\mathrm{pt}}^T (p_! L^{\otimes k}))} (k^{-1} \xi). \]

On the other hand, we have 
\begin{align*}
{^\hbar (\mathrm{PD}^{\mathrm{pt}}_T p_* \tau^T_X (L^{\otimes k}))} (k^{-1} \xi)
&= {^\hbar ( \mathrm{PD}^{\mathrm{pt}}_T p_* (\tau^T_X (\mathcal{O}_X) \frown e^{k c_{T, 1} (L)}))} (k^{-1} \xi)
\\
&= {^\hbar (\tau^T_X (\mathcal{O}_X). e^{k L_T}; k^{-1} \xi )}
\\
&= {^\hbar (e^{L_T}; \xi )} k^n - \frac{1}{2} {^\hbar (\kappa^T_X. e^{L_T}; \xi )} k^{n-1} + O (k^{n-2}), 
\end{align*}
which proves the claim by the equivariant Riemann--Roch theorem $\tau_{\mathrm{pt}}^T (p_! L^{\otimes k}) = p_* \tau^T_X (L^{\otimes k})$. 
Here the last equality is a consequence of the following lemma. 
\end{proof}

\begin{lem}
\label{scaling of exponential equivariant intersection}
For $\alpha \in H^T_{2 k} (X, \mathbb{R})$, $L \in H^2_T (X, \mathbb{R})$ and $\tau \in \mathbb{R}^\times$, we have 
\[ (\alpha. e^{\tau L}; \tau^{-1} \xi) = \tau^k (\alpha. e^L; \xi). \]
\end{lem}

\begin{proof}
We compute 
\[ (\alpha. e^{\tau L}; \tau^{-1} \xi) = \sum_{l=0}^\infty \frac{1}{l!} \tau^l (\alpha. L^l; \tau^{-1} \xi) = \sum_{l=0}^\infty \frac{1}{l!} \tau^l \tau^{k-l} (\alpha. L^l; \xi) = \tau^k (\alpha. e^L; \xi). \]
\end{proof}

\begin{cor}[cf. Appendix B in \cite{BHJ1}]
\label{equivariant intersection and DH measure, cor}
For a $T$-equivariant semiample and big $\mathbb{Q}$-line bundle $L$ on a proper scheme $X$, we have 
\begin{equation}
\frac{1}{(n+k)!} {^\hbar (} L^{\cdot n+k}; \xi) = \frac{1}{k!} \int_{\mathfrak{t}^\vee} (-\langle x, \hbar \xi \rangle)^k \mathrm{DH}_{(X, L_T)} (x). 
\end{equation}
\end{cor}

\begin{proof}
The claim follows by 
\[ \frac{k!}{(n+k)!} {^\hbar (} L^{\cdot n+k}; \xi) = \frac{d^k}{dt^k}\Big{|}_{t=0} {^\hbar (} e^L; t\xi) = \int_{\mathfrak{t}^\vee} (- \langle x, \hbar \xi \rangle)^k \mathrm{DH}_{(X, L_T)} (x). \]
\end{proof}

\begin{cor}
Let $f$ be an entire real analytic function on $\mathbb{R}$. 
Then we have 
\[ {^\hbar (f (L); \xi)} = \int_{\mathfrak{t}^\vee} f^{(n)} (- \langle x, \hbar \xi \rangle) \mathrm{DH}_{(X, L_T)} (x). \]
\end{cor}

The following basic property is applied to reduce Theorem A to Lahdili's result. 

\begin{prop}
\label{projection formula for equivariant intersection}
Let $f$ be a real analytic function on $\mathbb{R}$ whose Taylor series $\sum_{k=0} \frac{a_k}{k!} x^k$ is absolutely convergent for every $x \in \mathbb{R}$. 
Let $\beta: X' \to X$ be a $T$-equivariant morphism of pure $n$-dimensional projective schemes and $L_T$ be a $T$-equivariant $\mathbb{Q}$-line bundle on $X$. 
\begin{enumerate}
\item If $\beta$ is an isomorphism away from a codimension one subscheme of the target $X$, then we have 
\[ {^\hbar (f (\beta^* L_T); \xi)} = {^\hbar (\beta^* f (L_T); \xi)} = {^\hbar (f (L_T); \xi)}. \]

\item If moreover the $(n-1)$-the derivative $f^{(n-1)}$ is positive, $\beta$ is finite away from a codimension two subscheme of the target $X$, and $L$ is semiample \& big, then we have 
\[ {^\hbar (\kappa^T_{X'}. f (\beta^* L_T); \xi)} \le {^\hbar (\kappa^T_X. f (L_T); \xi)}. \]

\item If $\beta$ is an isomorphism away from a codimension two subscheme of the target $X$, and $L$ is semiample \& big, then we have 
\[ {^\hbar (\kappa^T_{X'}. f (\beta^* L_T); \xi)} = {^\hbar (\kappa^T_X. f (L_T); \xi)}. \]
\end{enumerate}
\end{prop}

\begin{proof}
We have $\beta_* [X']^T = [X]^T$ under the first assumption, so that by the equivariant projection formula, we have 
\[ ([X']^T. (\beta^* L_T)^{\cdot (n+k)}) = (\beta_* [X']^T. L_T^{\cdot (n+k)}) = ([X]^T. L_T^{\cdot (n+k)}) \in H_T^{2k} (\mathrm{pt}), \]
which shows the first claim. 

By the equivariant Grothendieck--Riemann--Roch theorem, we have $\beta_* \tau^T_{X'} (\mathcal{O}_{X'}) = \tau^T_X (\beta_! \mathcal{O}_{X'})$. 
Since $\beta$ is finite away from a codimension two subscheme of $X$, the supports of the higher direct images of $\mathcal{O}_{X'}$ is contained in the codimension two subscheme. 
So we have 
\[ (\tau^T_X (\beta_! \mathcal{O}_{X'}))_{\langle n-1 \rangle} = (\tau^T_X (\beta_* \mathcal{O}_{X'}))_{\langle n-1 \rangle} = (\tau^T_X (\beta_* \mathcal{O}_{X'}/\mathcal{O}_X) + \tau^T_X (\mathcal{O}_X))_{\langle n-1 \rangle}. \]
Since $\beta$ is isomorphism away from a codimension one subscheme of $X$, we have 
\[ (\tau^T_X (\beta_* \mathcal{O}_{X'}/\mathcal{O}_X))_{\langle n-1 \rangle} = \sum_i m_i [D_i]^T, \]
where $m_i \ge 0$ are the multiplicitis of $\beta_* \mathcal{O}_{X'}/\mathcal{O}_X$ along $T$-invarianrt prime divisors $D_i$ contained in the codimension one subscheme. 
It follows that we have 
\[ \beta_* \kappa^T_{X'} = \kappa^T_X - 2 \sum_i m_i [D_i]^T. \]
In particular, under the third assumption, we have $\beta_* \kappa^T_{X'} = \kappa^T_X$, which shows the third claim. 

By the equivariant projection formula, we have 
\begin{align*} 
(\kappa^T_{X'}. (\beta^* L_T)^{\cdot (n+k-1)})
&= (\beta_* \kappa^T_{X'}. L^{\cdot (n+k-1)}) 
\\
&= (\kappa^T_X . L^{\cdot (n+k-1)}) - 2 \sum_i m_i ([D_i]^T . L^{\cdot (n+k-1)}) \in H^{2k}_T (\mathrm{pt}), 
\end{align*}
so that 
\[ {^\hbar (\kappa^T_{X'}. f (\beta^* L_T); \xi)} = {^\hbar (\kappa^T_X. f (L_T); \xi)} - 2 \sum_i m_i {^\hbar (f (L_T|_{D_i}); \xi)}. \]
Then by Proposition \ref{absolute intersection}, $2 \sum_i m_i {^\hbar (f (L_T|_{D_i}); \xi)}$ is positive when $f^{(n-1)} = f^{(\dim D_i)}$ is positive, which shows the second claim.  
\end{proof}

\subsubsection{Equivariant calculus on the relative equivariant intersection $(\alpha. e^{\mathcal{L}})_B$}
\label{relative equivariant intersection theory} 

Now we study equivariant calculus of relative equivariant intersection. 
Let $G$ be an algebraic group, $T$ be an algebraic torus. 
Let $\mathcal{X}$ be a pure dimensional $T \times G$-scheme, $B$ be a smooth $G$-variety with the trivial $T$-action, $\varpi: \mathcal{X} \to B$ be a $T \times G$-equivariant proper morphism. 
Let $\mathcal{L}$ be a $T \times G$-equivariant second cohomology class on $\mathcal{X}$. 
Let $f$ be a real analytic function on $\mathbb{R}$ whose Taylor series $\sum_{k=0} \frac{a_k}{k!} x^k$ is absolutely convergent for every $x \in \mathbb{R}$. 

\begin{thm}
\label{derivative of relative equivariant intersection}
If $\alpha \in H^{\mathrm{lf}, T \times G}_{2\dim \mathcal{X}} (\mathcal{X}, \mathbb{R})$ or $\alpha \in H^{\mathrm{alg}, T \times G}_{2\dim \mathcal{X} -2} (\mathcal{X}, \mathbb{R})$, the relative equivariant intersection $(\alpha. f (\mathcal{L}))_B \in \hat{H}^{\mathrm{even}}_{T \times G} (B, \mathbb{R})$ is of class $\varepsilon^\infty$. (See Definition \ref{class epsilon ell} for \textit{class $\varepsilon^\infty$}. )
\end{thm}

\begin{proof}
We put $n := \dim \mathcal{X} - \dim B$. 
Let $K$ be a maximal compact subgroup of $G$. 
As in the proof of Proposition \ref{absolute intersection}, we may assume that $\mathcal{X}$ is smooth and $\varpi: \mathcal{X} \to B$ is a smooth $K$-equivariant map (not necessarily a submersion). 
Pick an equivariant $2$-from $\Omega + \hbar \mu = \Omega + \hbar (\mu^T + \mu^K)$ in $\mathcal{L}$. 
We firstly compute $(\mathcal{L}^{\smile k})_B^{\langle p, q \rangle}$. 
Compairing the degree, we have $(\mathcal{L}^{\smile k})_B^{\langle p, q \rangle} = 0$ when $k-n \neq p+q$. 
For $k= n+p+q$, we compute 
\begin{align*} 
{^\hbar \Phi} (\mathcal{L}^{\smile n+p+q})_B^{\langle p, q \rangle} 
&= (\mathrm{PD}^B_{T \times K} [\varpi_* (\Omega + \hbar \mu)^{n+p+q}])^{\langle p, q \rangle} 
\\
&= \sum_{l=n}^{n+p+q} \binom{n+p+q}{l} (\mathrm{PD}^B_{T \times K} [ \varpi_* ((\hbar \mu)^{n+p+q-l} \Omega^l)])^{\langle p, q \rangle} 
\\
&= \sum_{l=n}^{n+q} \binom{n+p+q}{n+p+q-l} \mathrm{PD}^B_{T \times K} \left[ \binom{n+p+q-l}{p} \varpi_* ((\hbar \mu^T)^p (\hbar \mu^K)^{n+q-l} \Omega^l) \right] 
\\
&= \mathrm{PD}^B_{T \times K} \left[ \binom{n+p+q}{n+q} \varpi_* ((\hbar \mu^T)^p \sum_{l=n}^{n+q} \binom{n+k}{l} (\hbar \mu^K)^{n+q-l} \Omega^l) \right]
\\
&= \binom{n+p+q}{n+q} \mathrm{PD}^B_{T \times K} [\varpi_* ((\hbar \mu^T)^p (\Omega + \hbar \mu^K)^{n+q})], 
\end{align*}
so that we obtain 
\begin{align*} 
{^\hbar \rho^T_\xi} ((f (\mathcal{L}))_B^{\langle p, q \rangle}) 
&= \frac{a_{n+p+q}}{(n+p+q)!} {^\hbar \rho^T_\xi} ((\mathcal{L}^{\smile n+p+q})_B^{\langle p, q \rangle}) 
\\
&= \frac{1}{(n+q)!} \mathrm{PD}^B_K [\varpi_* (\frac{a_{n+q+p}}{p!} (\mu_{\hbar \xi}^T)^p (\Omega + \hbar \mu^K)^{n+q})]. 
\end{align*}

We pick a collection of seminorms $\{ \| \cdot \|_\ell \}_{\ell \in \mathbb{Z}_{\ge 0}}$ on $ \Omega^{2(n+q)}_K (\mathcal{X})$ so that it defines the Fr\'echet structure of $ \Omega^{2(n+q)}_K (\mathcal{X})$. 
For instance, we may put $\| \varphi \|_\ell := \| \varphi|_{D_\ell} \|_{C^\ell}$ using an exhaustion $\{ D_\ell \subset \mathcal{X} \}_{\ell \in \mathbb{Z}_{\ge 0}}$ by compact sets. 
Now we easily see the infinite series of $K$-equivariant forms 
\[ \sum_{p=0}^\infty \frac{a_{n+q+p}}{p!} (\mu_{\hbar \xi}^T)^p (\Omega + \hbar \mu^K)^{n+q} \]
is locally unifromly absolutely-convergent on $\mathfrak{t}$ with respect to the seminorm $\| \cdot \|_\ell$ for each $\ell \in \mathbb{Z}_{\ge 0}$, to the limit $K$-equivariant form 
\[ f^{(n+q)} (\mu^T_{\hbar \xi}) (\Omega + \hbar \mu^K)^{n+q}. \]
Since $\varpi_* (\frac{a_{n+q+p}}{p!} (\mu_{\hbar \xi}^T)^p (\Omega + \hbar \mu^K)^{n+k})$ are $K$-equivariantly closed forms, the infinite series of $2k$-th $K$-equivariant cohomology classes on $B$
\[ \sum_{p=0}^\infty {^\hbar \rho^T_\xi} ((f (\mathcal{L}))_B^{\langle p, q \rangle}) = \frac{1}{(n+q)!} \sum_{p=0}^\infty \mathrm{PD}^B_K [ \varpi_* (\frac{a_{n+q+p}}{p!} (\mu^T_{\hbar \xi})^p (\Omega + \hbar \mu^K)^{n+q}) ] \]
is compactly absolutely convergent on $\mathfrak{t}$ by Corollary \ref{continuity of push} and Lemma \ref{Frechet to Banach}. 
The limit equivariant cohomology class is given by 
\[ \frac{1}{(n+q)!} \mathrm{PD}^B_K [\varpi_* (f^{(n+q)} (\mu^T_{\hbar \xi}) (\Omega + \hbar \mu^K)^{n+q})]. \]

Similarly, for $\alpha \in H^{\mathrm{lf}, T \times G}_{2 \dim \mathcal{X} - 2} (\mathcal{X}, \mathbb{R})$, we pick an equivariant $2$-form $A + \hbar \nu = A + \hbar (\nu^T + \nu^K)$ in the second equivariant cohomology class $\mathrm{PD}^B_{T \times K} ({^\hbar \alpha})$. 
We compute 
\begin{align*} 
{^\hbar \Phi} (\alpha. f (\mathcal{L}))_B^{\langle p, q \rangle} 
&= \frac{a_{n+p+q-1}}{(n+p+q-1)!} (\mathrm{PD}^B_{T \times K} \left[ \varpi_* \big{(} (A + \hbar \nu) (\Omega + \hbar \mu)^{n+p+q-1} \big{)} \right])^{\langle p, q \rangle}
\\
&= 
\begin{cases}
\frac{1}{(n+q)!} \mathrm{PD}^B_K [\varpi_* \big{(} (n+q) a_{n+q-1} (A + \hbar \nu^K) (\Omega + \hbar \mu^K)^{n+q-1} \big{)}]
& p=0
\\ 
\frac{1}{(n+q)!} \mathrm{PD}^B_K [\varpi_* \big{(} \hbar \nu^T \frac{a_{(n+q)+(p-1)}}{(p-1)!} (\hbar \mu^T)^{p-1} (\Omega + \hbar \mu^K)^{n+q} 
&
\\
\qquad \qquad + (n+q) \frac{a_{(n+q-1)+p}}{p!} (\hbar \mu^T)^p (A + \hbar \nu^K) (\Omega + \hbar \mu^K)^{n+q-1} \big{)} ]
& p \ge 1. 
\end{cases}
\end{align*}
By the same argument as above, we obtain that the infinite series 
\begin{align*}
\sum_{p=0}^\infty {^\hbar \rho^T_\xi} (\alpha. f (\mathcal{L}))^{\langle p, q \rangle}_B
&= \frac{\mathrm{PD}^B_K}{(n+q)!} \sum_{p=0}^\infty \Big{[} \varpi_* \Big{(} \nu^T_{\hbar \xi} \frac{a_{n+q+p}}{p!} (\mu^T_{\hbar \xi})^p (\Omega + \hbar \mu^K)^{n+q} 
\\
&\qquad \qquad \qquad + (n+q) \frac{a_{n+q-1 +p}}{p!} (\mu_{\hbar \xi}^T)^p (A + \hbar \nu^K)(\Omega + \hbar \mu^K)^{n+q-1} \Big{)} \Big{]}
\end{align*}
is compactly absolutely convergent on $\mathfrak{t}$ to the following equivariant cohomology class in ${^\hbar H^{2q}_{\mathrm{dR}, K}} (B)$: 
\[ \frac{\mathrm{PD}^B_K}{(n+q)!} \Big{[} \varpi_* \Big{(} \nu^T_{\hbar \xi} f^{(n+q)} (\mu^T_{\hbar \xi}) (\Omega + \hbar \mu^K)^{n+q} + (n+q) f^{(n+q-1)} (\mu_{\hbar \xi}^T) (A + \hbar \nu^K)(\Omega + \hbar \mu^K)^{n+q-1} \Big{)} \Big{]}, \]
which proves the claim. 
\end{proof}

From the proof, we obtain the following expressions: 
\begin{align} 
\label{differential of relative equivariant intersection eq 1}
\D^q_{\hbar. \xi} (f (\mathcal{L})) 
&= \frac{\mathrm{PD}^B_G {^\hbar \Phi^{-1}}}{(n+q)!} [\varpi_* (f^{(n+q)} (\mu^T_{\hbar\xi}) (\Omega + \hbar \mu^K)^{n+q})] 
\end{align}
and
\begin{align}
\label{differential of relative equivariant intersection eq 2}
\D^q_{\hbar. \xi} 
&(\alpha. f (\mathcal{L})) 
\\ \notag
&= \frac{\mathrm{PD}^B_G {^\hbar \Phi^{-1}}}{(n+q)!}  \Big{[} \varpi_* \Big{(} \nu^T_{\hbar\xi} f^{(n+q)} (\mu^T_{\hbar\xi}) (\Omega + \hbar \mu^K)^{n+q}
\\ \notag
&\qquad + (n+q) f^{(n+q-1)} (\mu^T_{\hbar\xi}) (A + \hbar \nu^K)(\Omega + \hbar \mu^K)^{n+q-1} \Big{)} \Big{]}. 
\end{align}

\subsubsection{$\mu$-character and its derivative}

\begin{lem}
Suppose $\varpi: \mathcal{X} \to B$ is moreover flat and $\mathcal{L}$ is in the equivariant Neron--Severi group $NS_{T \times G} (\mathcal{X}, \mathbb{R})$. 
Let $\mathcal{L}_b \in H^2_T (\mathcal{X}_b, \mathbb{R})$ be the restriction to the fibre $\mathcal{X}_b$ of a point $b \in B$. 
Then we have the following equalities
\begin{gather}
\D^0_{\hbar. \xi} (f (\mathcal{L}))_B = {^\hbar (f (\mathcal{L}_b); \xi)} \in \mathbb{R}, 
\\
\D^0_{\hbar. \xi} (\kappa_{\mathcal{X}/B}. f (\mathcal{L}))_B = {^\hbar (\kappa_{\mathcal{X}_b}. f (\mathcal{L}_b); \xi)} \in \mathbb{R}. 
\end{gather}
\end{lem}

\begin{proof}
By Corollary \ref{base change} and Proposition \ref{absolute intersection} and Theorem \ref{derivative of relative equivariant intersection}, we have 
\[ i_b^* \D^0_{\hbar. \xi} (\kappa_{\mathcal{X}/B}. f (\mathcal{L}))_B = \D^0_{\hbar. \xi} (\kappa_{\mathcal{X}_b}. f (\mathcal{L}_b)) = {^\hbar (\kappa_{\mathcal{X}_b}. f (\mathcal{L}_b); \xi)}. \]
Similarly, by the equivariant Grothendieck--Riemann--Roch theorem, we have 
\[ i_b^* \D^0_{\hbar. \xi} (f (\mathcal{L}))_B = \D^0_{\hbar. \xi} (f (\mathcal{L}_b)) = {^\hbar (f (\mathcal{L}_b); \xi)}. \]
\end{proof}

Now we obtain the characteristic class $\D_{\hbar. \xi} \cmu^\lambda_{T \times G} (\mathcal{X}/B, \mathcal{L})$ in Theorem B. 

\begin{prop} 
The following equivariant cohomology classes in $\hat{H}^{\mathrm{even}}_{T \times G} (B, \mathbb{R})$ defined in Definition \ref{mu-character definition} 
\[ \cmu_{T \times G} (\mathcal{X}/B, \mathcal{L}), \quad \bm{\check{\sigma}}_{T \times G} (\mathcal{X}/B, \mathcal{L}), \quad \cmu^\lambda_{T \times G} (\mathcal{X}/B, \mathcal{L}) \] 
are of class $\varepsilon^\infty$. 
The differentials 
\[ \D_{\hbar. \xi} \cmu_{T \times G} (\mathcal{X}/B, \mathcal{L}), \quad \D_{\hbar. \xi} \bm{\check{\sigma}}_{T \times G} (\mathcal{X}/B, \mathcal{L}) \in H^2_G (B, \mathbb{R}) \]
are given by 
\begin{gather*}
2 \pi \frac{ \D_{\hbar. \xi} (\kappa^{T \times G}_{\mathcal{X}/B}. e^{\mathcal{L}_{T \times G}})_B \cdot {^\hbar (e^{\mathcal{L}_b}; \xi)} - {^\hbar  (\kappa^T_{\mathcal{X}_b}. e^{\mathcal{L}_b}; \xi)} \cdot \D_{\hbar. \xi} (e^{\mathcal{L}_{T \times G}})_B }{{^\hbar (e^{\mathcal{L}_b}; \xi)}^2}, 
\\ 
\frac{ \D_{\hbar. \xi} (\mathcal{L}_{T \times G}. e^{\mathcal{L}_{T \times G}})_B \cdot {^\hbar (e^{\mathcal{L}_b}; \xi)} - {^\hbar (\mathcal{L}_b. e^{\mathcal{L}_b}; \xi)} \cdot \D_{\hbar. \xi} (e^{\mathcal{L}_{T \times G}})_B }{{^\hbar (e^{\mathcal{L}_b}; \xi)}^2} - \frac{\D_{\hbar. \xi} (e^{\mathcal{L}_{T \times G}})_B}{{^\hbar (e^{\mathcal{L}_b}; \xi)}}, 
\end{gather*}
respectively, and we have 
\[ \D_{\hbar. \xi} \cmu^\lambda_{T \times G} (\mathcal{X}/B, \mathcal{L}) = \D_{\hbar. \xi} \cmu_{T \times G} (\mathcal{X}/B, \mathcal{L}) + \lambda \D_{\hbar. \xi} \bm{\check{\sigma}}_{T \times G} (\mathcal{X}/B, \mathcal{L}). \]
\end{prop}

\begin{proof}
The claim follows by Theorem \ref{derivative of relative equivariant intersection}, Proposition \ref{Leibniz rule}, Proposition \ref{chain rule} and the above lemma. 
\end{proof}

In particular, we obtain Theorem B (4). 
By Corollary \ref{base change} and the base change properties of $\D_{\hbar. \xi} \alpha$ and $f (\alpha)$, we obtain Theorem B (1). 

\begin{rem}
In the construction of the characteristic class $\D_{\hbar. \xi} \cmu^\lambda_{T \times G} (\mathcal{X}/B, \mathcal{L})$, we assume the smoothness of the base $B$ in order to ensure the Poincar\'e duality between the equivariant cohomology $H^2_G (B, \mathbb{R})$ and the equivariant locally finite homology $H^{\mathrm{lf}, G}_{2\dim B -2} (B, \mathbb{R})$. 
Thus there is a room for extending our result to some singular bases satisfying the Poincar\'e duality (or perhaps to general singular bases by using other cohomology theory for which we can establish equivariant calculus as we have developed for equivariant singular / deRham cohomology). 

As another viewpoint, our construction works even for families of almost complex manifolds. 
In this case, $\mathcal{L}$ is just a $T \times G$-equivariant cohomology class. 
As a consequence, we can define $\D_{\hbar. \xi} \cmu^\lambda_{T \times G} (\mathcal{X}/B, \mathcal{L})$ for a Kuranishi family of $T$-polarized manifold with a singular base $B$ by pulling back the equivariant cohomology class $\D_{\hbar. \xi} \cmu^\lambda_{T \times G} (\tilde{\mathcal{X}}/\tilde{B}, \tilde{\mathcal{L}}) \in H^2 (\tilde{B}, \mathbb{R})$ associated to the Kuranishi slice $B \subset \tilde{B} \to \mathcal{J}$ which appeared in the construction of the Kuranishi family. 
Since the Kuranishi family is a local complete family, we can construct the characteristic class for any local family by pulling back the cohomology class for the Kuranishi family. 
The author speculates this idea allows us to construct the characteristic class $\D_{\hbar. \xi} \cmu^\lambda_{T \times G} (\mathcal{X}/B, \mathcal{L})$ for a general singular base $B$ by gluing these characteristic classes in some canonical way. 

It is preferable for this aim that we realize the characteristic class $\D_{\hbar. \xi} \cmu^\lambda_{T \times G} (\mathcal{X}/B, \mathcal{L})$ as some geometric object whose category forms a stack (namely, has a natural criterion for the descent of objects), like the CM line bundle. 
Such a geometric realization is also important when descending it to the moduli space. 
In fact, since in general the moduli space does not admit a universal family, our Thoerem B constructs nothing on the moduli space, at present. 

We note the cohomology class $\D_{\hbar. \xi} \cmu^\lambda_{T \times G} (\mathcal{X}/B, \mathcal{L})$ deforms continuously (and non-linearly) in $H^2_G (B, \mathbb{R})$ as $\xi$ varies, so it does not make sense to realize it as a complex line bundle on $B$. 
Real line bundle on $B$ endowed with pluri-harmonic transition functions may serve as such geometric object, but the actual construction is out of the author's consideration at the moment. 
\end{rem}

\subsubsection{$\mu$-Futaki inavriant is the weight of $\D_{\hbar. \xi} \cmu^\lambda_{T \times G}$}

Now we compare $\D_{\hbar. \xi} \cmu^\lambda_{T \times \mathbb{G}_m} (\mathcal{X}/\mathbb{A}^1, \mathcal{L})$ for $T$-equivariant test configuration and the $\mu$-Futaki invariant $\cFut^\lambda_{\hbar. \xi} (\mathcal{L}, \mathcal{L})$. 
This follows by the following localization formula. 

\begin{prop}
Let $f$ be an entire real analytic function on $\mathbb{R}$. 
For any test configuration $(\mathcal{X}/\mathbb{A}^1, \mathcal{L})$, we have 
\begin{align*} 
\D_{\hbar. \xi} (f (\mathcal{L}_{T \times \mathbb{G}_m}) )_{\mathbb{A}^1} 
&= - {^\hbar (f (\bar{\mathcal{L}}_T); \xi)}. \bm{x}, 
\\
\D_{\hbar. \xi} (\kappa_{\mathcal{X}/\mathbb{A}^1}^{T \times \mathbb{G}_m}. f (\mathcal{L}_{T \times \mathbb{G}_m}))_{\mathbb{A}^1} 
&= - {^\hbar (\kappa_{\bar{\mathcal{X}}/\mathbb{P}^1}^T. f (\bar{\mathcal{L}}_T); \xi)}. \bm{x}, 
\end{align*}
where $\bm{x} \in H^2_{\mathbb{G}_m} (\mathbb{A}^1, \mathbb{Z})$ is the generator corresponding to $c_1 (\mathcal{O} (-1)) \in H^2 (\mathbb{C}P^\infty, \mathbb{Z})$. 
\end{prop}

\begin{proof}
Since $j^*_{0_-} \bar{\mathcal{L}}_{T \times \mathbb{G}_m} = L_T \in NS_{T \times \mathbb{G}_m} (X, \mathbb{R})$, we have 
\[ i_{0_-}^* (\bar{\mathcal{L}}_{T \times \mathbb{G}_m}^{\cdot (n+i+1)})_{\mathbb{P}^1} = ((j_{0_-}^* \bar{\mathcal{L}}_{T \times \mathbb{G}_m})^{\cdot (n+i+1)}) = (L_T^{\cdot (n+i+1)}) \in H^{2i+2}_{T \times \mathbb{G}_m} (\mathrm{pt}, \mathbb{R}) \]
and 
\[ i_{0_-}^* (\kappa_{\bar{\mathcal{X}} / \mathbb{P}^1}^{T \times \mathbb{G}_m} . \bar{\mathcal{L}}_{T \times \mathbb{G}_m}^{\cdot (n+i)})_{\mathbb{P}^1} = (\kappa^T_X . L_T^{\cdot (n+i)}) \in H^{2i+2}_{T \times \mathbb{G}_m} (\mathrm{pt}, \mathbb{R}) \]
by Corollary \ref{base change}. 
In particular, $(i_{0_-}^* (\bar{\mathcal{L}}_{T \times \mathbb{G}_m}^{\cdot (n+i+1)})_{\mathbb{P}^1})^{\langle i, 1 \rangle}$ and $(i_{0_-}^* (\kappa_{\bar{\mathcal{X}} / \mathbb{P}^1 }^{T \times \mathbb{G}_m} . \bar{\mathcal{L}}_{T \times \mathbb{G}_m}^{\cdot (n+i)})_{\mathbb{P}^1})^{\langle i, 1 \rangle}$ are zero. 
Thus by the localization formula (see Example \ref{localization formula}), we compute 
\begin{align*}
((\mathcal{L}_{T \times \mathbb{G}_m}^{\cdot (n+i+1)})_{\mathbb{A}^1})^{\langle i, 1 \rangle} 
&= (i_0^* (\mathcal{L}_{T \times \mathbb{G}_m}^{\cdot (n+i+1)})_{\mathbb{A}^1})^{\langle i, 1 \rangle} 
= - (i_{0_-}^* (\bar{\mathcal{L}}_{T \times \mathbb{G}_m}^{\cdot (n+i+1)})_{\mathbb{P}^1} - i_0^* (\bar{\mathcal{L}}_{T \times \mathbb{G}_m}^{\cdot (n+i+1)})_{\mathbb{P}^1})^{\langle i, 1 \rangle}
\\
&= - ((\bar{\mathcal{L}}_{T \times \mathbb{G}_m}^{\cdot (n+i+1)}) \otimes \bm{x})^{\langle i, 1 \rangle} 
= - (\bar{\mathcal{L}}_{T}^{\cdot (n+i+1)})^{\langle i \rangle} \otimes \bm{x}
\end{align*}
and 
\begin{align*}
((\kappa_{\mathcal{X}/\mathbb{A}^1}^{T \times \mathbb{G}_m} . \mathcal{L}_{T \times \mathbb{G}_m}^{\cdot (n+i)})_{\mathbb{A}^1})^{\langle i, 1 \rangle} 
&= (i_0^* (\kappa_{\mathcal{X}/\mathbb{A}^1}^{T \times \mathbb{G}_m} . \mathcal{L}_{T \times \mathbb{G}_m}^{\cdot (n+i)})_{\mathbb{A}^1})^{\langle i, 1 \rangle} 
\\
&= - (i_{0_-}^* (\kappa_{\mathcal{X}/\mathbb{A}^1}^{T \times \mathbb{G}_m} . \bar{\mathcal{L}}_{T \times \mathbb{G}_m}^{\cdot (n+i)})_{\mathbb{P}^1} - i_0^* (\kappa_{\mathcal{X}/\mathbb{A}^1}^{T \times \mathbb{G}_m} . \bar{\mathcal{L}}_{T \times \mathbb{G}_m}^{\cdot (n+i)})_{\mathbb{P}^1})^{\langle i, 1 \rangle}
\\
&= - ((\kappa_{\bar{\mathcal{X}}/\mathbb{P}^1}^{T \times \mathbb{G}_m}. \bar{\mathcal{L}}_{T \times \mathbb{G}_m}^{\cdot (n+i)}) \otimes \bm{x})^{\langle i, 1 \rangle} 
= - (\kappa_{\bar{\mathcal{X}}/\mathbb{P}^1}^{T}. \bar{\mathcal{L}}_{T}^{\cdot (n+i)})^{\langle i \rangle} \otimes \bm{x}. 
\end{align*} 
Thus we get 
\[ \D_{\hbar. \xi} (\mathcal{L}_{T \times \mathbb{G}_m}^{\cdot (n+i+1)})_{\mathbb{A}^1} = - {^\hbar (\bar{\mathcal{L}}_{T}^{\cdot (n+i+1)};\xi)}. \bm{x} \]
and 
\[ \D_{\hbar. \xi} (\kappa_{\mathcal{X}/\mathbb{A}^1}^{T \times \mathbb{G}_m}. \mathcal{L}_{T \times \mathbb{G}_m}^{\cdot (n+i)})_{\mathbb{A}^1} = - {^\hbar (\kappa_{\mathcal{X}/\mathbb{A}^1}^{T}. \bar{\mathcal{L}}_{T}^{\cdot (n+i)}; \xi)}. \bm{x} \]
as all the equivariant intersections we treat here are real analytic on $\xi$. 
\end{proof}

As a corollary, we obtain the following. 

\begin{cor}
We have 
\[ \D_{\hbar. \xi} \cmu^\lambda (\mathcal{X}/\mathbb{A}^1, \mathcal{L}) = - \cFut^\lambda_{\hbar. \xi} (\mathcal{X}, \mathcal{L}). \bm{x}. \]
We can compare this with (\ref{mu to Fut}) by the following paraphrase 
\[ \langle {^\hbar \D_{\hbar. \xi}} \cmu^\lambda (\mathcal{X}/\mathbb{A}^1, \mathcal{L}) ,\eta \rangle = - \hbar \cdot \cFut^\lambda_{\hbar. \xi} (\mathcal{X}, \mathcal{L}). \]
\end{cor}

In particular, we obtain the following by the property (\ref{mu to Fut}) of the functional ${^\hbar \cmu^\lambda}$ and by Proposition \ref{derivative}. 

\begin{cor}
\label{algebraic mu-Futaki and differential mu-Futaki}
For the product configuration $(\mathcal{X}, \mathcal{L})$ associated to a lift $L_{T \times \mathbb{G}_m}$ of a one parameter subgroup $\Lambda: \mathbb{G}_m \to \mathrm{Aut}_T (X)$, we have 
\[ \hbar \cdot \cFut^\lambda_{\hbar. \xi} (\mathcal{X}, \mathcal{L}) = {^\hbar \cFut^\lambda_\xi} (\eta) \]
for ${^\hbar \cFut^\lambda_\xi} (\eta)$ defined by (\ref{mu-Futaki invariant}) in the introduction. 
\end{cor}

\subsubsection{Relation to CM line bundle}

Now we show the rest property (3) of the characteristic class $\mathfrak{D}_{\hbar. \xi} \cmu^\lambda (\mathcal{X}/B, \mathcal{L})$ in Theorem B. 

\begin{proof}[Proof of Theorem B (3)]
Recall the definition of the CM $\mathbb{Q}$-line bundle $\mathrm{CM} (\mathcal{X}/B, \mathcal{L})$: we put 
\[ \mathrm{CM} (\mathcal{X}/B, \mathcal{L}) := \lambda_{n+1}^{\otimes \frac{n}{n+1} \frac{(-K_X. L^{\cdot (n-1)})}{(L^{\cdot n})}} \otimes (\lambda_{n+1}^{\otimes n} \otimes \lambda_n^{\otimes (-2)}) \]
using line bundles $\lambda_i$ over $B$ appeared in the Knudsen--Mumford expansion 
\[ \det (\varpi_! (\mathcal{L}^{\otimes k})) \cong \lambda_{n+1}^{\binom{k}{n+1}} \otimes \lambda_n^{\binom{k}{n}} \otimes \dotsb \otimes \lambda_0 \quad (k \gg 0). \]
So we have 
\begin{align*} 
c_{G, 1} (\det (\varpi_! (\mathcal{L}^{\otimes k}))) 
&= \binom{k}{n+1} c_{G, 1} (\lambda_{n+1}) + \binom{k}{n} c_{G, 1} (\lambda_n) + \dotsb + c_{G, 1} (\lambda_0) 
\\
&= \frac{k^{n+1}}{(n+1)!} c_{G, 1} (\lambda_{n+1}) - \frac{1}{2} \frac{k^n}{n!} (n c_{G, 1} (\lambda_{n+1}) - 2 c_{G, 1} (\lambda_n)) + O (k^{n-1}). 
\end{align*}

On the other hand, by the equivariant Grothendieck--Riemann--Roch theorem, we have $\tau^G_B (\varpi_! (\mathcal{L}^{\otimes k})) = \varpi_* \tau^G_{\mathcal{X}} (\mathcal{L}^{\otimes k})$. 
Since $\mathcal{L}$ is relatively ample, $\varpi_! \mathcal{L}^{\otimes k}$ is a vector bundle for $k \gg 0$, so that we have 
\begin{align*}
(\mathrm{PD}^B_G \tau^G_B (\varpi_! \mathcal{L}^{\otimes k}))^{\langle 1 \rangle} = (\mathrm{PD}^B_G \tau^G_B (\mathcal{O}_B) \frown \mathrm{ch}_G (\varpi_! \mathcal{L}^{\otimes k}))^{\langle 1 \rangle} = c_{G, 1} (\varpi_! \mathcal{L}^{\otimes k}) - \frac{1}{2} \mathrm{rk} (\varpi_! \mathcal{L}^{\otimes k}) K_B^G
\end{align*}
and 
\begin{align*}
&(\mathrm{PD}^B_G \varpi_* \tau^G_{\mathcal{X}} (\mathcal{L}^{\otimes k}))^{\langle 1 \rangle} 
\\
&\quad = \mathrm{PD}^B_G (\varpi_* (\tau^G_{\mathcal{X}} (\mathcal{O}_{\mathcal{X}}) \frown e^{k c_{G, 1} (\mathcal{L})}))_{\langle \dim_{\mathbb{C}} B -1 \rangle}
\\
&\quad = k^{n+1}. (e^{\mathcal{L}_G})_B^{\langle 1 \rangle} - \frac{1}{2} k^n. (\kappa_{\mathcal{X}}^G. e^{\mathcal{L}_G})_B^{\langle 1 \rangle} + O (k^{n-1})
\\
&\quad = k^{n+1} \D_{\hbar. 0} (e^{\mathcal{L}})_B - \frac{1}{2} k^n \D_{\hbar. 0} (\kappa_{\mathcal{X}/B}. e^{\mathcal{L}})_B + O (k^{n-1}). 
\end{align*}
Therefore, we get 
\begin{align*} 
c_{G, 1} (\det (\varpi_! (\mathcal{L}^{\otimes k}))) = k^{n+1} \D_{\hbar. 0} (e^{\mathcal{L}})_B - \frac{1}{2} k^n \D_{\hbar. 0} (\kappa_{\mathcal{X}/B}. e^{\mathcal{L}})_B + O (k^{n-1}). 
\end{align*}

Comparing these expressions, we obtain 
\begin{align*} 
c_{G, 1} (\mathrm{CM} (\mathcal{X}/B, \mathcal{L})) 
&= \frac{n}{n+1} \frac{(-K_X. L^{\cdot (n-1)})}{(L^{\cdot n})} (n+1)! \cdot \D_{\hbar. 0} (e^{\mathcal{L}})_B + n! \cdot \D_{\hbar. 0} (\kappa_{\mathcal{X}/B}. e^{\mathcal{L}})_B 
\\
&= (L^{\cdot n}) \frac{\D_{\hbar. 0} (\kappa_{\mathcal{X}/B}. e^{\mathcal{L}})_B \cdot {^\hbar (e^{\mathcal{L}_b}; 0)} - {^\hbar (\kappa_{\mathcal{X}_b}^T. e^{\mathcal{L}_b}; 0)} \cdot \D_{\hbar. 0} (e^{\mathcal{L}})_B }{ {^\hbar (e^{\mathcal{L}_b}; 0)}^2 }
\\
&= \frac{(L^{\cdot n})}{2\pi} \D_{\hbar. 0} \cmu_G (\mathcal{X}/B, \mathcal{L}) 
\end{align*}
by ${^\hbar (e^{\mathcal{L}_b}; 0)} = \frac{1}{n!} (L^{\cdot n})$ and ${^\hbar (\kappa^T_{\mathcal{X}_b}. e^{\mathcal{L}_b}; 0)} = \frac{1}{(n-1)!} (K_X. L^{\cdot (n-1)})$. 

The independece of $\lambda$ follows from $\D_{\hbar. 0} (\mathcal{L}_G. e^{\mathcal{L}_G})_B = \frac{1}{n!} \D_{\hbar. 0} (\mathcal{L}_G^{\cdot (n+1)})_B$, $\D_{\hbar. 0} (e^{\mathcal{L}_G})_B = \frac{1}{(n+1)!} \D_{\hbar. 0} (\mathcal{L}_G^{\cdot (n+1)})_B$ and ${^\hbar (\mathcal{L}_b. e^{\mathcal{L}_b}; 0)} = \frac{1}{(n-1)!} (L^{\cdot n})$: 
\begin{align*} 
\D_{\hbar. 0} \bm{\check{\sigma}}_G
&= \frac{ \D_{\hbar. 0} (\mathcal{L}_{T \times G}. e^{\mathcal{L}_{T \times G}})_B \cdot {^\hbar (e^{\mathcal{L}_b}; 0)}  - {^\hbar (\mathcal{L}_b. e^{\mathcal{L}_b}; 0)} \cdot \D_{\hbar. 0} (e^{\mathcal{L}_{T \times G}})_B }{{^\hbar (e^{\mathcal{L}_b}; 0)}^2} - \frac{\D_{\hbar. 0} (e^{\mathcal{L}_{T \times G}})_B}{{^\hbar (e^{\mathcal{L}_b}; 0)}} 
\\
&= \frac{1}{(L^{\cdot n})} \Big{(} \D_{\hbar. 0} (\mathcal{L}^{\cdot (n+1)}_G)_B - \frac{n}{n+1} \D_{\hbar. 0} (\mathcal{L}_G^{\cdot (n+1)}) - \frac{1}{n+1} \D_{\hbar. 0} (\mathcal{L}_G^{\cdot (n+1)}) \Big{)}
=0.
\end{align*}
\end{proof}

\subsection{$\mu$K-stability of polarized schemes}

\subsubsection{Relation to established Futaki invariants}

Here we check that our definition of $\mu$-Futaki invariant is compatible with the following established notions. 

Recall the Duistermaat--Heckman measure $\mathrm{DH}_{(\mathcal{X}, \mathcal{L}_{T \times \mathbb{G}_m})}$ on $\mathfrak{t}^\vee \times \mathbb{R}$ is given by 
\[ \mathrm{DH}_{(\mathcal{X}, \mathcal{L}_{T \times \mathbb{G}_m})} := \lim_{k \to \infty} k^{-n} \sum_{(m, l) \in M \times \mathbb{Z}} \dim H^0 (\mathcal{X}_0, \mathcal{L}|_{\mathcal{X}_0})_{(m, l)} \delta_{k^{-1} (m, l)}. \]

\begin{defin}
Let $(X, L)$ be a (semi)polarized scheme and $(\mathcal{X}, \mathcal{L})$ be a test configuration of $(X, L)$. 
The following Futaki invariants are studied in the literatures. 
\begin{enumerate}

\item \textit{Modified Futaki invariant} for the modified K-stability of $\mathbb{Q}$-Fano variety (cf. \cite{BW, Xio} and \cite{WZZ}): Let $(X, L_T) = (X, -\lambda K_X^T)$ be a $\mathbb{Q}$-Fano variety ($\lambda > 0$) with a torus $T$ action. 
For a $T$-equivariant normal test configuration $(\mathcal{X}, \mathcal{L})$ with $\mathbb{Q}$-Gorenstein $\mathcal{X}$ and $\mathcal{L}_{T \times \mathbb{G}_m} = - K_{\mathcal{X}/\mathbb{A}^1}^{T \times \mathbb{G}_m}$, we put 
\begin{equation}
\mathrm{MFut}_{\hbar. \xi} (\mathcal{X}, \mathcal{L}) := - \int_{\mathfrak{t}^\vee \times \mathbb{R}} t e^{- \langle x, \hbar \xi \rangle} \mathrm{DH}_{(\mathcal{X}, \mathcal{L})} (x, t) 
\end{equation}
for $\xi \in \mathfrak{t}$. 

When the central fibre $\mathcal{X}_0$ is a $\mathbb{Q}$-Fano variety, we have the following expression: 
\[ 2  \mathrm{MFut}_{-2. \xi} (\mathcal{X}, \mathcal{L}) = - \int_{\mathcal{X}_0} \theta_\eta e^{\theta_\xi} \omega^n = \int_{\mathcal{X}_0} \eta^J (h - \theta_\xi) e^{\theta_\xi} \omega^n \]
by a suitably regular K\"ahler metric $\omega$ on $\mathcal{X}_0$, where $h$ denotes a Ricci potential: $\sqrt{-1} \partial \bar{\partial} h = \mathrm{Ric} (\omega) - \omega$. 
This is the expression in \cite{BW, Xio}. 
Here the $\bar{\partial}$-Hamiltonian potential $\theta$ is normalized so that $[\beta^* \omega + \beta^* \theta] = - \beta^* ({^{-2} K_{\mathcal{X}_0}^{T \times U (1)}}) \in {^{-2} H^2_{T \times U (1)}} (\tilde{\mathcal{X}}_0, \mathbb{R})$ for some/any resolution $\beta: \tilde{\mathcal{X}}_0 \to \mathcal{X}_0$. 
This normalization is equivalent to the equation $\bar{\Box} \theta_\eta - \eta^J h = \theta_\eta$. 

\item \textit{Weighted Futaki invariant} of smooth test configuration (cf. \cite{Lah}) for weighted K-stability: Let $v$ and $w$ be smooth positive functions on the moment polytope $P := \mu^\omega (X) \subset \mathfrak{t}^\vee$. 
For a $T$-equivariant smooth test configuration $(\mathcal{X}, \mathcal{L}_T)$ with ample $\bar{\mathcal{L}}$, we pick a K\"ahler form $\Omega$ in $\bar{\mathcal{L}}$ and the moment map $\mu^\Omega: \bar{\mathcal{X}} \to \mathfrak{t}^\vee$ so that $[\Omega + \hbar \mu^\Omega] = {^\hbar \bar{\mathcal{L}}_T}$. 
Then we put 
\begin{align} 
\mathcal{F}_{v, w} (\mathcal{X}, \mathcal{L}) 
&= - \frac{1}{n+1} \int_{\bar{\mathcal{X}}} \Big{(} s_v (\Omega) - \frac{\int_X s_v (\omega) \omega^n}{\int_X (w \circ \mu^\omega) \omega^n} (w \circ \mu^\Omega) \Big{)} \Omega^{n+1} 
\\ \notag
&\qquad \qquad + 4\pi \int_X (v \circ \mu^\omega) \omega^n. 
\end{align}
\end{enumerate}
\end{defin}

We consider the following variant of weighted Futaki invariant for $T$-equivariant smooth test configuration with ample $\bar{\mathcal{L}}$: 
\begin{align} 
\mathcal{F}_{\hbar. \xi}^\lambda (\mathcal{X}, \mathcal{L})
&:= \mathcal{F}_{e^{\langle x, \hbar \xi \rangle}, e^{\langle x, \hbar \xi \rangle}} (\mathcal{X}, \mathcal{L}) + \frac{\lambda}{n+1} \int_{\bar{\mathcal{X}}} \big{(} \hbar \mu^\Omega_{\xi} - \frac{\int_X \hbar \mu^\omega_{\xi} e^{\hbar \mu^\omega_{\xi}} \omega^n}{\int_X e^{\hbar \mu^\omega_{\xi}} \omega^n} \big{)} e^{\hbar \mu^\Omega_{\xi}} \Omega^{n+1} 
\\ \notag
&=- \frac{1}{n+1} \int_{\bar{\mathcal{X}}} \big{(} {^\hbar \check{s}^\lambda_\xi} (\Omega) - {^\hbar \bar{s}^\lambda_\xi} (\omega) \big{)} e^{\hbar \mu^\Omega_\xi} \Omega^{n+1} +4\pi \int_X e^{\hbar \mu^\omega_\xi} \omega^n.
\end{align}
Though it seems not explicitly claimed in \cite{Lah} for this variant case, we easily see by basic computations in \cite{Lah} that the above invariant coincides the slope of $\mu^\lambda_\xi$-Mabuchi functional (cf. \cite{Ino2}) along a smooth subgeodesic ray subordinate to a test configuration. 
Indeed, this follows by the equivariant Stokes theorem. 
Since a $\mu^\lambda_\xi$-cscK metric is an extremal weighted cscK metric (in our formulation, extremal vector $\zeta$ is imposed to be proportional to the weight vector $\xi$: $\zeta = \lambda \xi$), we have the boundedness of $\mu^\lambda_\xi$-Mabuchi functional if there exists a $\mu^\lambda_\xi$-cscK metric as proved in \cite{Lah}. 
As a consequence, a polarized manifold is $\mu^\lambda_\xi$K-semistable with respect to \textit{smooth test configurations with ample $\bar{\mathcal{L}}$} if there exists a $\mu^\lambda_\xi$-cscK metric. 

\begin{prop}
\label{equivariant intersection on test configuration and DH measure}
For a test configuration $(\mathcal{X}, \mathcal{L})$, we have 
\[ {^\hbar (e^{\bar{\mathcal{L}}_T}; \xi)} = \int_{\mathfrak{t}^\vee \times \mathbb{R}} t e^{-\langle x, \hbar \xi \rangle} \mathrm{DH}_{(\mathcal{X}, \mathcal{L}_{T \times \mathbb{G}_m})} (x, t). \]
\end{prop}

\begin{proof}
Since ${^\hbar (e^{\mathcal{L}_{T \times \mathbb{G}_m}|_{\mathcal{X}_0}}; (\xi, \rho \eta))} = \int_{\mathfrak{t}^\vee \times \mathbb{R}} e^{-(\langle x, \hbar \xi \rangle + t \hbar \rho)} \mathrm{DH}_{(\mathcal{X}, \mathcal{L})} (x, t)$ and $\D_{\hbar. \xi} (e^{\mathcal{L}_{T \times \mathbb{G}_m}|_{\mathcal{X}_0}}) = - {^\hbar (e^{\bar{\mathcal{L}}_T}; \xi)}. \bm{x}$, we compute 
\begin{align*} 
{^\hbar (e^{\bar{\mathcal{L}}_T}; \xi)}. \hbar 
&= - \langle {^\hbar \D_{\hbar. \xi}} (e^{\mathcal{L}_{T \times \mathbb{G}_m}|_{\mathcal{X}_0}}), \eta \rangle
\\ 
&= - \frac{d}{d\rho}\Big{|}_{\rho = 0} \int_{\mathfrak{t}^\vee \times \mathbb{R}} e^{-(\langle x, \hbar \xi \rangle + t \hbar \rho)} \mathrm{DH}_{(\mathcal{X}, \mathcal{L})} (x, t) 
\\
&= \hbar \int_{\mathfrak{t}^\vee \times \mathbb{R}} t e^{-\langle x, \hbar \xi \rangle} \mathrm{DH}_{(\mathcal{X}, \mathcal{L})} (x, t) 
\end{align*}
by Proposition \ref{derivative}. 
\end{proof}

\begin{cor}
\label{equivariant intersection on test configuration and DH measure, cor}
For a test configuration $(\mathcal{X}, \mathcal{L})$, we have 
\begin{equation}
\frac{1}{(n+1+k)!} {^\hbar (} \bar{\mathcal{L}}_T^{\cdot n+1+k}; \xi) = \frac{1}{k!} \int_{\mathfrak{t}^\vee \times \mathbb{R}} t (- \langle x, \hbar \xi \rangle)^k \mathrm{DH}_{(\mathcal{X}, \mathcal{L}_{T \times \mathbb{G}_m})} (x, t). 
\end{equation}
\end{cor}

\begin{proof}
The claim follows by 
\[ \frac{k!}{(n+1+k)!} {^\hbar (} \bar{\mathcal{L}}_T^{\cdot n+1+k}; \xi) = \frac{d^k}{ds^k}\Big{|}_{s=0} {^\hbar (} e^{\bar{\mathcal{L}}_T}; s \xi) = \int_{\mathfrak{t}^\vee \times \mathbb{R}} t (- \langle x, \hbar \xi \rangle)^k \mathrm{DH}_{(\mathcal{X}, \mathcal{L}_{T \times \mathbb{G}_m})} (x, t). \]
\end{proof}

\begin{cor}
Let $f$ be an entire real analytic function on $\mathbb{R}$. 
Then we have 
\[ {^\hbar (f (\bar{\mathcal{L}}_T); \xi)} = \int_{\mathfrak{t}^\vee \times \mathbb{R}} t f^{(n+1)} (- \langle x, \hbar \xi \rangle) \mathrm{DH}_{(\mathcal{X}, \mathcal{L}_{T \times \mathbb{G}_m})} (x, t) \]
\end{cor}

Now we compare them with our $\mu$-Futaki inavriant. 

\begin{prop}
\label{comparison}
We can compare the $\mu$-Futaki invariant with these established Futaki invariants as follows. 
\begin{enumerate}
\item When $(X, L)$ is a normal polarized variety, we have 
\begin{equation*}
\cFut_{\hbar. 0}^\lambda (\mathcal{X}, \mathcal{L}) = \frac{2\pi}{(L^{\cdot n})} \mathrm{DF} (\mathcal{X}, \mathcal{L})
\end{equation*} 
for every normal test configuration $(\mathcal{X}, \mathcal{L})$ of $(X, L)$. 

\item When $X$ is a $\mathbb{Q}$-Fano variety and $L = -\lambda^{-1} K_X$ for $\lambda > 0$, we have 
\begin{equation*}
\cFut^{2\pi \lambda}_{\hbar. \xi} (\mathcal{X}, \mathcal{L}) = \frac{2\pi \lambda}{\int_{\mathfrak{t}^\vee} e^{- \langle x, \hbar \xi \rangle} \mathrm{DH} (x)} \mathrm{MFut}_{\hbar. \xi} (\mathcal{X}, \mathcal{L}) 
\end{equation*} 
for every $T$-equivariant test configuration $(\mathcal{X}, \mathcal{L})$ of $(X, L)$ with $\mathbb{Q}$-Gorenstein $\mathcal{X}$ and $\mathcal{L} = -\lambda^{-1} K_{\mathcal{X}/\mathbb{A}^1}$. 

\item Let $f, g$ be entire real analytic functions on $\mathbb{R}$. 
We put $\tilde{v} = f^{(n)},  \tilde{w} = g^{(n+1)}$ and $v = \tilde{v} (\langle \cdot, \hbar \xi \rangle), w = \tilde{w} (\langle \cdot, \hbar \xi \rangle)$. 
Then for a smooth $(X, L)$ and a smooth $T$-equivariant test configuration $(\mathcal{X}, \mathcal{L})$ with ample $\bar{\mathcal{L}}$, we have 
\begin{equation} 
\label{equivariant intersection formula of weighted Futaki}
\frac{1}{2\pi n!} \mathcal{F}_{v, w} (\mathcal{X}, \mathcal{L}) = {^\hbar (K^T_{\bar{\mathcal{X}}/\mathbb{P}^1}. f (\bar{\mathcal{L}}); \xi)} - \frac{ {^\hbar (K^T_X. f' (L); \xi)} }{ {^\hbar (g' (L); \xi)} } {^\hbar (g (\bar{\mathcal{L}}); \xi)}, 
\end{equation}
In particular, we have 
\begin{equation*}
\cFut^\lambda_{\hbar. \xi} (\mathcal{X}, \mathcal{L}) = \frac{1}{\int_X e^{\hbar \mu^\omega_{\xi}} \omega^n} \mathcal{F}_{\hbar. \xi}^\lambda (\mathcal{X}, \mathcal{L}). 
\end{equation*} 
\end{enumerate}
\end{prop}

\begin{proof}~

(1) We compute $\cFut^\lambda_{\hbar. 0} (\mathcal{X}, \mathcal{L})$ as 
\begin{align*}
&2 \pi \frac{ {^\hbar (\kappa_{\bar{\mathcal{X}}/\mathbb{P}^1}. e^{\bar{\mathcal{L}}}; 0)} \cdot {^\hbar (e^L; 0)}  - {^\hbar (\kappa_X. e^L; 0)} \cdot {^\hbar (e^{\bar{\mathcal{L}}}; 0)} }{{^\hbar (e^L; 0)}^2}
\\
&\qquad \qquad + \lambda \left[ \frac{ {^\hbar (\bar{\mathcal{L}}. e^{\bar{\mathcal{L}}}; 0)} \cdot {^\hbar (e^L; 0)} - {^\hbar (L. e^L; 0)} \cdot {^\hbar (e^{\bar{\mathcal{L}}}; 0)} }{{^\hbar (e^L; 0)}^2} - \frac{{^\hbar (e^{\bar{\mathcal{L}}}; 0)}}{{^\hbar (e^L; 0)}} \right] 
\\
&\quad = 2 \pi \left( \frac{n!}{(L^{\cdot n})} \right)^2 \left( \frac{1}{n!} (K_{\bar{\mathcal{X}}/\mathbb{P}^1}. \bar{\mathcal{L}}^{\cdot n}) \cdot \frac{1}{n!} (L^{\cdot n}) - \frac{1}{(n-1)!} (K_X. L^{\cdot (n-1)}) \cdot \frac{1}{(n+1)!} (\bar{\mathcal{L}}^{\cdot (n+1)}) \right) 
\\
&\qquad \qquad + \lambda \left( \frac{n!}{(L^{\cdot n})} \right)^2 \left( \frac{(\bar{\mathcal{L}}^{\cdot (n+1)})}{n!} \cdot \frac{(L^{\cdot n})}{n!} - \frac{(L^{\cdot n})}{(n-1)!} \cdot \frac{(\bar{\mathcal{L}}^{\cdot (n+1)})}{(n+1)!}  - \frac{(L^{\cdot n})}{n!} \frac{(\bar{\mathcal{L}}^{\cdot (n+1)})}{(n+1)!} \right)
\\
&\quad = \frac{2 \pi}{(L^{\cdot n})} \left( (K_{\bar{\mathcal{X}}/\mathbb{P}^1}. \bar{\mathcal{L}}^{\cdot n}) - \frac{n}{n+1} \frac{(K_X. L^{\cdot (n-1)})}{(L^{\cdot n})} (\bar{\mathcal{L}}^{\cdot (n+1)}) \right) = \frac{2\pi}{(L^{\cdot n})} \mathrm{DF} (\mathcal{X}, \mathcal{L}). 
\end{align*}

(2) We compute $\cFut^{2 \pi \lambda}_{\hbar. \xi} (\mathcal{X}, \mathcal{L})$ as 
\begin{align*}
&2 \pi \frac{ {^\hbar (\kappa_{\bar{\mathcal{X}}/\mathbb{P}^1}. e^{\bar{\mathcal{L}}}; \xi)} \cdot {^\hbar (e^L; \xi)}  - {^\hbar (\kappa_X. e^L; \xi)} \cdot {^\hbar (e^{\bar{\mathcal{L}}}; \xi)} }{{^\hbar (e^L; \xi)}^2}
\\
&\qquad + 2 \pi \lambda \left[ \frac{ {^\hbar (\bar{\mathcal{L}}. e^{\bar{\mathcal{L}}}; \xi)} \cdot {^\hbar (e^L; \xi)} - {^\hbar (L. e^L; \xi)} \cdot {^\hbar (e^{\bar{\mathcal{L}}}; \xi)} }{{^\hbar (e^L; \xi)}^2} - \frac{{^\hbar (e^{\bar{\mathcal{L}}}; \xi)}}{{^\hbar (e^L; \xi)}} \right] 
\\
&\quad = - 2\pi \lambda \frac{ {^\hbar (e^{\bar{\mathcal{L}}}; \xi)} }{ {^\hbar (e^L; \xi)} }
= - 2\pi \lambda \int_{\mathfrak{t}^\vee \times \mathbb{R}} t e^{- \langle x, \hbar \xi \rangle} \mathrm{DH}_{(\mathcal{X}, \mathcal{L})} (x, t) \Big{/} \int_{\mathfrak{t}^\vee} e^{-\langle x, \hbar \xi \rangle} \mathrm{DH}_{(X, L)} (x). 
\end{align*}

(3) For a $T$-polarized smooth variety $(X, L)$ and an entire real anlaytic function $f$ on $\mathbb{R}$, we compute 
\begin{align*}
{^\hbar (K^T_X. f (L); \xi)} 
&= - \frac{1}{2\pi} \int_X (\mathrm{Ric} (\omega) + \bar{\Box} \mu_{\hbar \xi}) f (\omega+ \mu_{\hbar \xi})
\\
&= - \frac{1}{2\pi} \int_X (s (\omega) f^{(n-1)} (\mu_{\hbar \xi}) + \bar{\Box} \mu_{\hbar \xi} f^{(n)} (\mu_{\hbar \xi})) \omega^n/n! 
\\
&= - \frac{1}{2\pi} \int_X s_{f^{(n -1)} (\langle \cdot, \hbar \xi \rangle)} \omega^n/n!
\end{align*}
and 
\begin{align*}
{^\hbar (f (L); \xi)}
&= \int_X f (\omega + \mu_{\hbar \xi}) = \int_X f^{(n)} (\mu_{\hbar \xi}) \omega^n/n!. 
\end{align*}

The equality (\ref{equivariant intersection formula of weighted Futaki}) follows by 
\begin{align*} 
{^\hbar (\varpi^* K_{\mathbb{P}^1}. f (\bar{\mathcal{L}}); \xi)} 
&= \mathfrak{D}_{\hbar. \xi} \int_{\mathbb{P}^1} K_{\mathbb{P}^1} \smile \varpi_* (f (\bar{\mathcal{L}}))
\\
&= \int_{\mathbb{P}^1} c_1 (\mathcal{O} (-2)) \cdot i^*_{0_-} \mathfrak{D}_{\hbar. \xi} \varpi_* (f (\bar{\mathcal{L}}))
\\
&= -2 {^\hbar (f (L); \xi)}. 
\end{align*}

To see the equality on $\mathcal{F}_{\hbar .\xi}^\lambda$, it suffices to check 
\[ \int_{\bar{\mathcal{X}}} \hbar \mu^\Omega_{\xi} e^{\hbar \mu^\Omega_{\xi}} \frac{\Omega^{n+1}}{(n+1)!} = {^\hbar (\bar{\mathcal{L}}. e^{\bar{\mathcal{L}}}; \xi)} - {^\hbar (e^{\bar{\mathcal{L}}}; \xi)} \]
and 
\[ \frac{\int_X \hbar \mu^\omega_{\xi} e^{\hbar \mu^\omega_{\xi}} \omega^n}{\int_X e^{\hbar \mu^\omega_{\xi}} \omega^n} \cdot \int_{\bar{\mathcal{X}}}  e^{\hbar \mu^\Omega_{\xi}} \frac{\Omega^{n+1}}{(n+1)!} = \frac{ {^\hbar (L. e^L; \xi)} }{{^\hbar (e^L; \xi)} } {^\hbar (e^{\bar{\mathcal{L}}}; \xi)}, \]
which we already know by the above computation. 

\end{proof}

Now we show the following connection with the relative Donaldson--Futaki invariant. 

\begin{thm}
For general test configuration $(\mathcal{X}, \mathcal{L})$, we have 
\[ \lim_{\lambda \to -\infty} \cFut^\lambda_{\xi_\lambda} (\mathcal{X}, \mathcal{L}) = \cFut^\dagger_{\xi_{\mathrm{ext}}} (\mathcal{X}, \mathcal{L}). \]
\end{thm}

\begin{proof}
Recall for each $\lambda \ll 0$, we have a unique vector $\xi_\lambda$ for which $\cFut^\lambda_{\xi_\lambda}$ vanishes for product configurations and we have $\lambda \xi_\lambda \to \xi_{\mathrm{ext}}$ for which $\cFut^\dagger_{\xi_{\mathrm{ext}}}$ vanishes for product configurations. 
In particular, we have $\xi_\lambda \to 0$ as $\lambda \to -\infty$, so that we compute 
\[ \frac{ {^\hbar (\kappa_{\bar{\mathcal{X}}/\mathbb{P}^1}^T. e^{\bar{\mathcal{L}}_T}; \xi_\lambda)} \cdot {^\hbar (e^{L_T}; \xi_\lambda)}  - {^\hbar (\kappa_X^T. e^{L_T}; \xi_\lambda)} \cdot {^\hbar (e^{\bar{\mathcal{L}}_T}; \xi_\lambda)} }{{^\hbar (e^{L_T}; \xi_\lambda)}^2}  \to \frac{1}{(L^{\cdot n})} \mathrm{DF} (\mathcal{X}, \mathcal{L}). \]
On the other hand, for $\xi^\tau = \xi_{\tau^{-1}}$, we have $(d\xi^\tau/d\tau)|_{\tau=-0} = \lim_{\tau \to -0} \tau^{-1} \xi^\tau = \xi_{\mathrm{ext}}$ and  
\[ {^\hbar (\alpha. e^L; \xi^\tau)} = \tau^{- \mathrm{rk} \alpha} \cdot {^\hbar (\alpha. e^{\tau L}; \tau^{-1} \xi^\tau)} \]
by Lemma \ref{scaling of exponential equivariant intersection}. 

Since 
\begin{align*} 
{^\hbar (e^L; \xi)} 
&= \frac{1}{n!} (L^{\cdot n}) + \frac{1}{(n+1)!} {^\hbar (L^{\cdot n+1}; \xi)} + \dotsb 
\\
{^\hbar (e^{\bar{\mathcal{L}}}; \xi)} 
&= \frac{1}{(n+1)!} (\bar{\mathcal{L}}^{\cdot n+1}) + \frac{1}{(n+2)!}  {^\hbar (\bar{\mathcal{L}}^{\cdot n+2}; \xi)} + \dotsb 
\\
{^\hbar (L. e^L; \xi)} 
&= \frac{1}{(n-1)!} (L^{\cdot n}) + \frac{1}{n!} (L^{\cdot n+1}; \xi) + \dotsb 
\\
{^\hbar (\bar{\mathcal{L}}. e^{\bar{\mathcal{L}}}; \xi)} 
&= \frac{1}{n!} (\bar{\mathcal{L}}^{\cdot n+1}) + \frac{1}{(n+1)!} {^\hbar (\bar{\mathcal{L}}^{\cdot n+2}; \xi)} + \dotsb 
\end{align*}
gives the Taylor series, we compute 
\begin{align*} 
\frac{d}{d\tau}\Big{|}_{\tau = -0} \frac{ {^\hbar (\bar{\mathcal{L}}. e^{\bar{\mathcal{L}}}; \xi^\tau)} }{ {^\hbar (e^L; \xi^\tau)} } 
&= \frac{1}{n+1} \frac{ {^\hbar (\bar{\mathcal{L}}^{\cdot n+2}; \xi_{\mathrm{ext}})} \cdot (L^{\cdot n}) - (\bar{\mathcal{L}}^{\cdot n+1}) \cdot {^\hbar (L^{\cdot n+1}; \xi_{\mathrm{ext}})} }{ (L^{\cdot n})^2 }
\\
\frac{d}{d\tau}\Big{|}_{\tau = -0} \frac{ {^\hbar (L. e^L; \xi^\tau)} }{ {^\hbar (e^L; \xi^\tau)} }
&= \frac{1}{n+1} \frac{{^\hbar (L^{\cdot n+1}; \xi_{\mathrm{ext}})} }{(L^{\cdot n})}
\\
\frac{d}{d\tau}\Big{|}_{\tau = -0} \frac{ {^\hbar (e^{\bar{\mathcal{L}}}; \xi^\tau)} }{ {^\hbar (e^L; \xi^\tau)} } 
&=\frac{ \frac{1}{(n+1) (n+2)}  {^\hbar (\bar{\mathcal{L}}^{\cdot n+2}; \xi_{\mathrm{ext}})} \cdot (L^{\cdot n}) - \frac{1}{(n+1)^2} (\bar{\mathcal{L}}^{\cdot n+1}) \cdot {^\hbar (L^{\cdot n+1}; \xi_{\mathrm{ext}})} }{ (L^{\cdot n})^2 }.
\end{align*}
It follows that 
\begin{align*} 
\lim_{\lambda \to -\infty}
& \lambda \frac{ {^\hbar (\bar{\mathcal{L}}_T. e^{\bar{\mathcal{L}}_T}; \xi_\lambda)} \cdot {^\hbar (e^{L_T}; \xi_\lambda)} - {^\hbar (L_T. e^{L_T}; \xi_\lambda)} \cdot {^\hbar (e^{\bar{\mathcal{L}}_T}; \xi_\lambda)} }{{^\hbar (e^{L_T};\xi_\lambda )}^2}  
\\
&= \frac{d}{d\tau}\Big{|}_{\tau = - 0} \left[ \frac{ {^\hbar (\bar{\mathcal{L}}_T. e^{\bar{\mathcal{L}}_T}; \xi^\tau)} \cdot {^\hbar (e^{L_T}; \xi^\tau)} - {^\hbar (L_T. e^{L_T}; \xi^\tau)} \cdot {^\hbar (e^{\bar{\mathcal{L}}_T}; \xi^\tau)} }{{^\hbar (e^{L_T};\xi^\tau)}^2} - \frac{{^\hbar (e^{\bar{\mathcal{L}}_T}; \xi^\tau)}}{{^\hbar (e^{L_T}; \xi^\tau)}} \right]
\\
&= \frac{1}{n+1} \frac{ {^\hbar (\bar{\mathcal{L}}^{\cdot n+2}; \xi_{\mathrm{ext}})} \cdot (L^{\cdot n}) - (\bar{\mathcal{L}}^{\cdot n+1}) \cdot {^\hbar (L^{\cdot n+1}; \xi_{\mathrm{ext}})} }{ (L^{\cdot n})^2 }
\\
&\qquad + \frac{1}{(n+1)^2} \frac{{^\hbar (L^{\cdot n+1}; \xi_{\mathrm{ext}})} }{(L^{\cdot n})} \cdot \frac{(\bar{\mathcal{L}}^{\cdot n+1})}{(L^{\cdot n})} 
\\
&\qquad \quad - \frac{ \frac{1}{n+2}  {^\hbar (\bar{\mathcal{L}}^{\cdot n+2}; \xi_{\mathrm{ext}})} \cdot (L^{\cdot n}) - \frac{1}{n+1} (\bar{\mathcal{L}}^{\cdot n+1}) \cdot {^\hbar (L^{\cdot n+1}; \xi_{\mathrm{ext}})} }{ (L^{\cdot n})^2 }
\\
&= \frac{1}{(n+1)(n+2)} \frac{ {^\hbar (\bar{\mathcal{L}}^{\cdot n+2}; \xi_{\mathrm{ext}})} }{ (L^{\cdot n}) } + \frac{1}{(n+1)^2} \frac{{^\hbar (L^{\cdot n+1}; \xi_{\mathrm{ext}})} }{(L^{\cdot n})} \cdot \frac{(\bar{\mathcal{L}}^{\cdot n+1})}{(L^{\cdot n})}, 
\end{align*}
which proves the claim. 
\end{proof}

\subsubsection{Basic property of $\mu$K-stability}

Now we are ready to apply the results in section \ref{absolute equivariant intersection theory}. 

\begin{thm}
\label{fundamental lemma}
~
\begin{enumerate}
\item A $T$-polarized normal variety $(X, L)$ is $\mu^\lambda_\xi$K-semistable (resp. $\mu^\lambda_\xi$K-polystable, $\mu^\lambda_\xi$K-stable) with respect to general test configurations iff it is $\mu^\lambda_\xi$K-semistable (resp. $\mu^\lambda_\xi$K-polystable, $\mu^\lambda_\xi$K-stable) with respect to normal ample test configurations. 

\item A $T$-polarized manifold $(X, L)$ is $\mu^\lambda_\xi$K-semistable with respect to general test configurations iff it is $\mu^\lambda_\xi$K-semistable with respect to smooth ample test configurations with reduced central fibre. 
\end{enumerate}
\end{thm}

\begin{proof}
Pick a semiample test configuration $(\mathcal{X}, \mathcal{L})$ of $(X, L)$. 
As $L$ is ample and $\mathcal{L}$ is relatively semiample, we have a unique ample test configuration $(\mathcal{X}^{\mathrm{amp}}, \mathcal{L}^{\mathrm{amp}})$ of the same $(X, L)$ associated to $(\mathcal{X}, \mathcal{L})$ as in \cite[Definition 2.16]{BHJ1}. 
The associated morphism $\mu: \mathcal{X} \to \mathcal{X}^{\mathrm{amp}}$ is an isomorphism away from a codimension one subscheme of the central fibre, which is a codimension two subscheme of the total space. 
It follows that $\cFut^\lambda_{\hbar. \xi} (\mathcal{X}^{\mathrm{amp}}, \mathcal{L}^{\mathrm{amp}}) = \cFut^\lambda_{\hbar. \xi} (\mathcal{X}, \mathcal{L})$ by Proposition \ref{projection formula for equivariant intersection}. 
Thus we may assume $\mathcal{L}$ is relatively ample. 
By the above remark, we may further assume that $\mathcal{L}$ is ample. 
We apply Proposition \ref{projection formula for equivariant intersection} (2) to the normalization $\nu: \mathcal{X}^\nu \to \mathcal{X}$ and obtain the first claim. 

Now, we may assume $(\mathcal{X}, \mathcal{L})$ is a normal ample test configuration to prove the second claim. 
Since $\mathcal{X}$ is normal, there is a $T \times \mathbb{G}_m$-equivariant resolution $\beta: \tilde{\mathcal{X}} \to \mathcal{X}$ of singularities which is isomorphism away from a codimension two subscheme in $\mathcal{X}$. 
Then we have $\cFut^\lambda_{\hbar. \xi} (\tilde{\mathcal{X}}, \beta^* \mathcal{L}) = \cFut^\lambda_{\hbar. \xi} (\mathcal{X}, \mathcal{L})$ by Proposition \ref{projection formula for equivariant intersection}. 

By the reduced fibre theorem, there is a positive integer $d$ such that the normalized base-change $\nu_d: \tilde{\mathcal{X}}_d^\nu \to \tilde{\mathcal{X}}$ along the morphism $m_d (z) = z^d: \mathbb{A}^1 \to \mathbb{A}^1$ has the reduced central fibre. 
Let $\mu_d: \tilde{\mathcal{X}}_d \to \tilde{\mathcal{X}}$ be the (non-normalized) base change morphism along $m_d$. 
We note these morphisms are all equivariant with respect to the group homomorphism $\varphi_d: \mathbb{G}_m \to \mathbb{G}_m: \tau \mapsto \tau^d$. 
Since $\D_{\hbar. \xi} \cmu^\lambda_{T \times \mathbb{G}_m} (\tilde{\mathcal{X}}_d/\mathbb{A}^1, \mu_d^* \beta^* \mathcal{L}) = m_d^* \varphi_d^\# \D_{\hbar. \xi} \cmu^\lambda_{T \times \mathbb{G}_m} (\mathcal{X}/\mathbb{A}^1, \mathcal{L})$ and $\mu_d^* \varphi_d^\# \bm{x} = d. \bm{x}$, we compute 
\begin{align*} 
\cFut^\lambda_{\hbar. \xi} (\tilde{\mathcal{X}}_d, \mu_d^* \beta^* \mathcal{L}). \bm{x} 
&= - \D_{\hbar. \xi} \cmu^\lambda_{T \times \mathbb{G}_m} (\tilde{\mathcal{X}}_d/\mathbb{A}^1, \mu_d^* \beta^* \mathcal{L}) 
\\
&= - m_d^* \varphi_d^\# \D_{\hbar. \xi} \cmu^\lambda_{T \times \mathbb{G}_m} (\tilde{\mathcal{X}}/\mathbb{A}^1, \beta^* \mathcal{L}) = d. \cFut^\lambda_{\hbar. \xi} (\mathcal{X}, \mathcal{L}). \bm{x} 
\end{align*}
in $H^2_{\mathbb{G}_m} (\mathbb{A}^1, \mathbb{R})$. 
Thus we obtain $\cFut^\lambda_{\hbar. \xi} (\tilde{\mathcal{X}}_d^\nu, \nu_d^* \beta^* \mathcal{L}) \le d. \cFut^\lambda_{\hbar. \xi} (\mathcal{X}, \mathcal{L})$ by Proposition \ref{projection formula for equivariant intersection} (2). 

Since the resolution of singularities can be obtained by a sequence of blowing-ups (though this seems not true for complex analytic space, the following claim is still valid as noted in \cite{DR}), 
there is an effective exceptional $\beta$-anti-ample divisor $\Sigma$ supporting precisely on $\mathrm{Exc} (\beta)$, so that $\tilde{\mathcal{L}}_\epsilon := \nu_d^* (\beta^* \mathcal{L} - \epsilon \Sigma^{T \times \mathbb{G}_m})$ is ample for every small $\epsilon > 0$. 
Therefore we get a smooth test configuration $(\tilde{\mathcal{X}}_d^\nu, \tilde{\mathcal{L}}_\epsilon)$ with reduced central fibre and ample $\tilde{\mathcal{L}}_\epsilon$ (cf. \cite[Chapter IV, Section 3]{KKMS}). 
Since $\tilde{\mathcal{L}}_\epsilon \to \nu_d^* \beta^* \mathcal{L}$ in $H^2_{\mathbb{G}_m} (\tilde{\mathcal{X}}_d^\nu, \mathbb{R})$ as $\epsilon \to 0$, we have $d^{-1} \cFut^\lambda_{\hbar. \xi} (\tilde{\mathcal{X}}_d^\nu, \tilde{\mathcal{L}}_\epsilon) \to \cFut^\lambda_{\hbar. \xi} (\mathcal{X}, \mathcal{L})$ from the integration expression of equivariant intersection in the proof of Proposition \ref{absolute intersection}. 
Now suppose $(X, L)$ is $\mu$K-semistable with respect to smooth ample test configurations with reduced central fibre. 
Then since $d^{-1} \cFut^\lambda_{\hbar. \xi} (\tilde{\mathcal{X}}_d, \tilde{\mathcal{L}}_\epsilon) \ge 0$, we get $\cFut^\lambda_{\hbar. \xi} (\mathcal{X}, \mathcal{L}) \ge 0$ for every normal test configuration $(\mathcal{X}, \mathcal{L})$. 
Thus $(X, L)$ is $\mu$K-semistable with respect to general test configurations. 
The reverse implication is obvious. 
\end{proof}

The above theorem reduces Theorem A to Lahdili's result \cite[Theorem 2]{Lah}. 
Thus we obtain the following. 

\begin{cor}
If a smooth K\"ahler manifold $(X, L)$ admits a $\mu^\lambda_\xi$-cscK metric, then $(X, L)$ is $\mu^\lambda_\xi$K-semistable. 
\end{cor}

\begin{rem}
As in \cite[Proposition 4]{Lah}, we can express $\mu^\lambda_\xi$-Futaki invariant of toric test configurations of a polarized toric variety by an integration on polytope, similarly to the usual Futaki invariant. 
For the toric test configuration $(\mathcal{X}_q, \mathcal{L}_q)$ associated to a convex piecewise linear function $q$ on the polytope $P$, we have 
\[ {^\hbar \bar{s}^\lambda_\xi} = 2\pi \frac{\int_{\partial P} e^{- \langle x, \hbar \xi \rangle} d\sigma}{\int_P e^{- \langle x, \hbar \xi \rangle} d\mu} + \lambda \frac{\int_P \langle x, \hbar \xi \rangle e^{- \langle x, \hbar \xi \rangle} d\mu}{\int_P e^{- \langle x, \hbar \xi \rangle} d\mu} \]
and
\[ \cFut^\lambda_{\hbar. \xi} (\mathcal{X}_q, \mathcal{L}_q) = 2\pi \frac{\int_{\partial P} q e^{- \langle x, \hbar \xi \rangle} d\sigma}{\int_P e^{- \langle x, \hbar \xi \rangle} d\mu} + \lambda \frac{\int_P \langle x, \hbar \xi \rangle q e^{- \langle x, \hbar \xi \rangle} d\mu}{\int_P e^{- \langle x, \hbar \xi \rangle} d\mu} - {^\hbar \bar{s}^\lambda_\xi} \frac{\int_P q e^{- \langle x, \hbar \xi \rangle} d\mu}{\int_P e^{- \langle x, \hbar \xi \rangle} d\mu} \]
for some uniform positive constant $c$. 
The expression holds also for singular $(X, L)$, in which case we modify the $\mu^\lambda_\xi$-Futaki invariant by replacing $K_{\bar{\mathcal{X}}/\mathbb{P}^1}$ with $K_{\bar{\mathcal{X}}/\mathbb{P}^1}^{\log}$. 
\end{rem}

\section{Appendix: Equivariant cohomology and the Cartan model}
\label{section: Appendix}

Here we (re)arrange background materials on equivariant cohomology and equivariant locally finite homology to fix our notations and sign conventions in equivariant cohomology. 
The sign arrangement is crucial when computing the right sign of ($\mu$-)Futaki invariant via equivariant cohomology. 

We also briefly explain some advantage of Cartan model, which employs differential forms as its chains. 
While there is an analogous equivariant theory for Chow group, which works also for schemes  not even over arbitrary characteristic field but also over $\mathbb{Z}$, we prefer to use the singular / de Rham cohomology with $\mathbb{R}$-coefficient to benefit from the Cartan model when proving the convergence of some sequences in equivariant cohomology. 

\subsection{Equivariant singular cohomology \& locally finite homology}
\label{section: Equivariant singular cohomology and locally finite homology}

\subsubsection{Equivariant singular cohomology}

We firstly review equivariant \textit{singular} cohomology and locally finite homology as these work also for singular spaces. 
Let $X$ be a topological space with a continuous right action of a topological group $G$. 
Using a classifying principal bundle $EG \to BG$ of $G$, which can be realized for instance by Milnor construction, the \textit{$G$-equivariant singular cohomology} $H^*_G (X, \mathbb{Z})$ is defined to be the singular cohomology of the Borel construction $X_G := EG \times_G X = (EG \times X) /G = \{ [p, x] ~|~ [p, x] = [p g, x g], ~\forall g \in G \}$: for $k = 0, 1, \ldots$, we put 
\begin{equation} 
\label{Borel construction of equivariant cohomology}
H^k_G (X, \mathbb{Z}) := H^k (X_G, \mathbb{Z}). 
\end{equation}
For another choice of the classifying space $E' G \to B' G$, we have a classifying homotopy equivalence $B' G \to BG, BG \to B' G$, by which we get a natural isomorphism $H^* (X_G, \mathbb{Z}) \cong H^* (X'_G, \mathbb{Z})$. 
We usually define the equivariant cohomology in this way among other possible candidates for `equivariant cohomology' in order to ensure the equivariant homotopy invariance. 

Let $Y$ be another topological space with a continuous action of a topological group $H$, $\varphi: H \to G$ be a topological group morphism and $f: Y \to X$ be a $H$-equivariant continuous map, i.e. $f (x. g) = f (x). \varphi (g)$ for $g \in H$. 
Then we have a pullback map $f^* \varphi^\#: H^*_G (X, \mathbb{Z}) \to H^*_H (Y, \mathbb{Z})$. 

In general $EG$ is infinite dimensional, so we usually have $H_G^k (X, \mathbb{Z}) \neq 0$ for infinitely many $k \ge 0$ even if $X$ is finite dimensional. 
We put $\hat{H}^{\mathrm{even}}_G (X, \mathbb{Z}) := \prod_{k=0}^\infty H^{2k}_G (X, \mathbb{Z})$ and denote by $\alpha^{\langle k \rangle} \in H^{2k}_G (X, \mathbb{Z})$ the degree $2k$-part for an element $\alpha \in \hat{H}_G^{\mathrm{even}} (X, \mathbb{Z})$. 
When the action is free, we have a natural isomorphism $H^*_G (X, \mathbb{Z}) \cong H^* (X/G, \mathbb{Z})$, so that $H^k_G (X, \mathbb{Z}) = 0$ for $k \ge \dim (X/G)$ in this case. 

For an almost connected locally compact group $G$ (i.e. the quotient $G/G^0$ by the identity component is compact), we have a maximal compact subgroup $\varphi: K \hookrightarrow G$ by Iwasawa's theorem. 
Since it admits a $K$-equivariant deformation retract $H_t: G \to K$, the natural map $\tilde{\varphi} :EK \to EK \times_K G \to EG$ induced from the homotopical universality of $BG$ is a $K$-equivariant homotopy equivalence. 
Thus we obtain the induced isomorphism $\varphi^\#: H^*_G (X, \mathbb{Z}) \xrightarrow{\sim} H^*_K (X, \mathbb{Z})$. 
As this isomorphism is independent of the choice of the maximal compact subgroup $\varphi$, we often identify these two equivariant cohomologies. 

For a $G$-equivariant complex vector bundle $E \to X$, we have the associated descent vector bundle $E_G := EG \times_G E$ over $X_G$. 
We define the \textit{equivariant Chern class} $c_G (E) \in H^{\mathrm{even}}_G (X, \mathbb{Z})$ to be the Chern class of $E_G$: $c_* (E_G) \in H^* (X_G, \mathbb{Z}) = H^*_G (X, \mathbb{Z})$. 
The \textit{equivariant Chern charatcer} $\mathrm{ch}_G (L)$ of a $G$-equivariant complex line bundle $L$ is defined as 
\begin{equation} 
\mathrm{ch}_G (L) := e^{c_{G, 1} (L)} = \sum_{p=0}^\infty \frac{1}{p!} (c_{G, 1} (L))^{\smile p}, 
\end{equation}
which lives in the formal product $\hat{H}^{\mathrm{even}}_G (X, \mathbb{Z}) := \prod_{k=0}^\infty H^{2k}_G (X, \mathbb{Z})$.

\subsubsection{Equivariant locally finite homology}
\label{section: equivariant locally finite homology}

\subsubsection*{$\bullet$ Locally finite homology}
For a locally compact Hausdorff space $X$, the \textit{locally finite homology} $H^{\mathrm{lf}}_* (X, \mathbb{Z})$ is defined to be the homology of the chain complex $C^{\mathrm{lf}}_*$ of the locally finite chains, i.e. 
\[ C^{\mathrm{lf}}_k := \Big{\{} \sigma: \mathrm{Map} (\Delta^k, X) \to \mathbb{Z} ~\Big{|}~  \begin{matrix} \forall K \subset X: \text{ compact set } 
\\
\# \{ c \in \sigma^{-1} (\mathbb{Z} \setminus \{ 0 \}) ~|~ c^{-1} (K) \neq \emptyset \} < \infty 
\end{matrix} \Big{\}}, \]
where $\mathrm{Map} (\Delta^k, X)$ denotes the set of continuous maps. 
We usually denote its chain by a formal expression $\sum_{c \in \mathrm{Map} (\Delta^k, X)} \sigma (c). c$. 
The boundary map $\partial: C^{\mathrm{lf}}_k \to C^{\mathrm{lf}}_{k-1}$ is given similarly as the usual homology. 
The locally finite homology $H^{\mathrm{lf}}_* (-, \mathbb{Z})$ gives a covariant functor from the category of locally compact Hausdorff spaces with \textit{proper continuous maps} to the category of $\mathbb{Z}$-modules. 
The functor is not a homotopy functor, but only invariant under \textit{proper homotopy}. 
For example, $H^{\mathrm{lf}}_k (X, \mathbb{Z}) \ncong H^{\mathrm{lf}}_k (X \times \mathbb{R}^q, \mathbb{Z})$ while we have $H^{\mathrm{lf}}_k (X, \mathbb{Z}) \cong H^{\mathrm{lf}}_{k+q} (X \times \mathbb{R}^q, \mathbb{Z})$. 

We have the cap product: 
\begin{equation}
\frown: H^{\mathrm{lf}}_k (X, \mathbb{Z}) \otimes H^l (X, \mathbb{Z}) \to H^{\mathrm{lf}}_{k-l} (X, \mathbb{Z}), 
\end{equation}
by which the anti-graded module $H^{\mathrm{lf}}_{-*} (X, \mathbb{Z})$ becomes a module over the algebra $(H^* (X, \mathbb{Z}), \smile)$. 
We have the following projection formula: 
\begin{equation}
f_* (\sigma \frown f^* \varphi) = f_* \sigma \frown \varphi 
\end{equation}
for every proper continuous map $f: X \to Y$ and $\sigma \in H^{\mathrm{lf}}_k (X, \mathbb{Z})$, $\varphi \in H^l (Y, \mathbb{Z})$. 

When $X$ is a real $n$-dimensional connected topological oriented manifold, we have an orientation preserving triangulation $\sum_{\alpha \in A} \Delta_\alpha^n$ of $X$ and obtain a generator $[X] \in H^{\mathrm{lf}}_n (X, \mathbb{Z})$ called the \textit{fundamental class} of $X$, independent of the choice of the triangulation. 
The map $([X] \frown \cdot): H^k (X, \mathbb{Z}) \to H^{\mathrm{lf}}_{n-k} (X, \mathbb{Z})$ gives an isomorphism of $\mathbb{Z}$-modules for each $q \in \mathbb{Z}$. 
We denote its inverse $([X] \frown \cdot)^{-1}$ by $\mathrm{PD}^X: H^{\mathrm{lf}}_k (X, \mathbb{Z}) \to H^{n-k} (X, \mathbb{Z})$.

\subsubsection*{$\bullet$ Equivariant locally finite homology}

The following are key properties for the well-definedness of the equivariant version of locally finte homology: 
\begin{itemize}
\item For any closed subset $Y \subset X$, we have a long exact sequence 
\begin{align} 
\label{lf exact sequence}
\dotsb \to H^{\mathrm{lf}}_k (Y, \mathbb{Z}) \to H^{\mathrm{lf}}_k (X, \mathbb{Z}) \to H^{\mathrm{lf}}_k (X \setminus Y, \mathbb{Z}) \to H^{\mathrm{lf}}_{k-1} (Y, \mathbb{Z}) \to \dotsb. 
\end{align}
In particular, when $\dim Y < d$, we have the isomorphism $H^{\mathrm{lf}}_k (X, \mathbb{Z}) \cong H^{\mathrm{lf}}_k (X \setminus Y, \mathbb{Z})$ for $k > d$. 

\item For a real vector bundle $\pi: E \to X$ of real rank $r$, we have an isomorphism 
\begin{align}
\pi^*: H^{\mathrm{lf}}_* (X, \mathbb{Z}) \xrightarrow{\sim} H^{\mathrm{lf}}_{*+r} (E, \mathbb{Z}). 
\end{align}
\end{itemize}

Now we introduce the equivariant version of locally finite homology (cf. \cite{EG1}). 
Let $X$ be a real $n$-dimensional locally compact space with a continuous action of an almost connected Lie group $G$. 
For an almost connected Lie group $G$, we have a collection of principal $G$-bundles $\{ E_l G \to B_l G \}_{l \in \mathbb{N}}$ which enjoys the following properties: 
\begin{enumerate}
\item For each $l \in \mathbb{N}$, $E_l G$ is $G$-equivariantly proper homotopy equivalent to a $G$-invariant Zariski open set $V^\circ$ of a complex $G$-representation $V$ satisfying $\dim_{\mathbb{R}} (V \setminus V^\circ) > l+1$. 

\item The group $G$ acts on $E_l G$ freely and $E_l G \to B_l G$ is the quotient. 
\end{enumerate}
We call such a collection a \textit{finite dimensional approximation of classifying bundle} $EG \to BG$. 

For $l' \ge l$, $E_{l'} G \to B_{l'} G$ can serve as $E_l G \to B_l G$. 
For a maximal compact group $K \subset G$, $E_l G \to E_l G/K$ can serve as $E_l K \to B_l K$. 
For example, when $G = \mathbb{G}_m$, $E_l \mathbb{G}_m := \mathbb{C}^{l+1} \setminus \{ 0 \}$ endowed with the diagonal $\mathbb{G}_m$-action gives such a finite dimensional approximation; $B_l \mathbb{G}_m = \mathbb{C}P^l$. 
For the maximal compact group $K = U (1)$, we have $E_l U (1) = E_l \mathbb{G}_m$ and $B_l U (1) = \mathbb{C}P^l \times \mathbb{R}_{>0}$. 

We put $X^k_G := E_{n-k} G \times_G X$. 
For $k' \le k$, $X^{k'}_G$ serve as $X^k_G$. 
Using such a finite dimensional approximation, we define the \textit{$G$-equivariant locally finite homology} $H^{\mathrm{lf}, G}_k (X, \mathbb{Z})$ of degree $k \in \mathbb{Z}$ (negative degree allowed) by 
\begin{equation} 
H^{\mathrm{lf}, G}_k (X, \mathbb{Z}) := H^{\mathrm{lf}}_{\dim_\mathbb{R} (X^k_G/X) + k} (X^k_G, \mathbb{Z}), 
\end{equation}
where we put $\dim_\mathbb{R} (X^k_G/X) := \dim_\mathbb{R} X^k_G - \dim_\mathbb{R} X = \dim_\mathbb{R} B_{n-k} G$. 

For example, since $\mathrm{pt}^k_{\mathbb{G}_m} = \mathbb{C}P^{-k}$ we have 
\[ H^{\mathrm{lf}, \mathbb{G}_m}_k (\mathrm{pt}, \mathbb{Z}) = H^{\mathrm{lf}}_{-2k + k} (\mathbb{C}P^{-k}, \mathbb{Z}) = 
\begin{cases} 
0 & k > 0 \text{ or } k \text{ odd } 
\\ 
\mathbb{Z} & k \le 0 \text{ and } k \text{ even } 
\end{cases}. \]

By the key properties of locally finite homology, we can show the above construction is independent of the choice of a finite dimensional approximation of classifying space similarly as \cite{EG1}. 
For a $G$-equivariant proper continuous map $f: X \to Y$, we have the proper pushforward $f_*: H^{\mathrm{lf}, G}_k (X) \to H^{\mathrm{lf}, G}_k (Y)$ induced from the map $f: E_{n-k} G \times_G X \to E_{n-k} G \times_G Y$ with $n = \max \{ \dim X, \dim Y \}$. 

For a maximal compact group $\varphi: K \subset G$, $\tilde{\varphi}: X^k_K := E_{n-k} G \times_K X \to X^k_G$ is a fibre bundle with contractible fibres $G/K \cong \mathfrak{g}/\mathfrak{k}$. 
Then since $\dim_\mathbb{R} (X^k_K/X) = \dim_\mathbb{R} (G/K) + \dim_\mathbb{R} (X^k_G/X)$, the following gives an isomorphism 
\[ \varphi^\#: H^{\mathrm{lf}, G}_k (X, \mathbb{Z}) = H^{\mathrm{lf}}_{\dim_\mathbb{R} (X^k_G/X) + k} (X^k_G, \mathbb{Z}) \xrightarrow{\tilde{\varphi}^*} H^{\mathrm{lf}}_{\dim_\mathbb{R} (X^k_K/X) + k} (X^k_K, \mathbb{Z}) = H^{\mathrm{lf}, K}_k (X, \mathbb{Z}) \]
by the above key property of locally finite homology. 

Since we have $H^l_G (X, \mathbb{Z}) \cong H^l (E_m G \times_G X, \mathbb{Z})$ for $m \ge l$, we have the equivariant cap product: 
\begin{align}
\frown: H^{\mathrm{lf}, G}_k (X, \mathbb{Z}) \otimes H^l_G (X, \mathbb{Z}) \to H^{\mathrm{lf}, G}_{k-l} (X, \mathbb{Z}), 
\end{align}
which is also independent of the choice of the finite dimensional approximation. 
The projection formula for the usual case yields the equivariant version of the projection formula.

\subsubsection*{$\bullet$ Equivariant cycle map}

Let $X$ be a pure $n$-dimensional complex space with a holomorphic $G$-action. 
The $k$-th \textit{$G$-equivariant Chow group} $A^G_k (X)$ is defined as the $(\dim_{\mathbb{C}} (X^{2k}_G/X) + k)$-th Chow group of $X^{2k}_G$, for which we use the Zariski open set $E_l G = V^\circ \subset V$ . 
It is independent of the choice of the finite dimensional approximation $E_{2(n-k)} G \to B_{2(n-k)} G$ of the classifying bundle as shown in \cite{EG1}. 
A \textit{$G$-equivariant $k$-cycle} on a pure $n$-dimensional complex space $X$ with a $G$-action means a $\dim_{\mathbb{C}} (X^{2k}_G/X) + k$-cycle on $X^{2k}_G = E_{2(n-k)} G \times_G X$. 
The $G$-equivariant cycle map 
\begin{equation} 
\label{equivariant cycle map}
cl^G: A^G_k (X) \to H^{\mathrm{lf}, G}_{2k} (X, \mathbb{Z}) 
\end{equation}
is given by the usual cycle map $cl: A_{\dim_{\mathbb{C}} (X^{2k}_G/X) + k} (X^{2k}_G) \to H^{2(\dim_{\mathbb{C}} (X^{2k}_G/X) + k)} (X^{2k}_G, \mathbb{Z})$ where we assign the fundamental class of the regular part for each irreducible cycle. 

A $G$-invariant $k$-cycle $C$ on $X$ gives a $G$-equivariant $k$-cycle by $E_{2(n-k)} G \times_G C \subset E_{2(n-k)} G \times_G X$. 
We denote the corresponding class in $H_{2k}^G (X, \mathbb{Z})$ as $[C]^G$. 
Since $X$ itself is $G$-invariant, we always have the equivariant fundamental class $[X]^G$: 
\begin{equation} 
\label{cas fundamental class}
[X]^G := \sum_{i \in I} m_i [X_i^\circ]^G \in H^{\mathrm{lf}, G}_{2n} (X, \mathbb{Z}), 
\end{equation}
where $X_i^\circ$ denotes the regular part of the reduction $X^{\mathrm{red}}_i$, and $m_i$ denotes the length of $\mathcal{O}_{X_i}/\mathcal{O}_{X_i^{\mathrm{red}}}$ at a general point of the the irreducible component $X_i \subset X$.

When $X$ is smooth, the map $([X]^G \frown \cdot): H^k_G (X, \mathbb{Z}) \to H^{\mathrm{lf}, G}_{2n-k} (X, \mathbb{Z})$ gives an isomorphism for each $k \in \mathbb{Z}$. 
Since we can take $E_l G$ as a manifold, this follows by the usual Poincare duality. 
We denote its inverse $([X]^G \frown \cdot)^{-1}$ by 
\[ \mathrm{PD}^X_G: H^{\mathrm{lf}, G}_k (X, \mathbb{Z}) \to H^{2n-k}_G (X, \mathbb{Z}). \]
For a proper map $f: X \to Y$, we abbreviate $\mathrm{PD}^Y_G \circ f_*$ as $f_\diamond$ when it makes sense, while we write $f_* \circ ([X]^G \frown \cdot)$ and $\mathrm{PD}^Y_G \circ f_* \circ ([X]^G \frown \cdot)$ by the same symbol $f_*$. 

\subsection{Equivariant deRham cohomology \& current homology}
\label{section: Cartan model of equivariant cohomology and locally finite homology}

\subsubsection{Cartan model of equivariant cohomology}

Now we turn to the Cartan model. 
The Cartan model of equivariant cohomology behaves well when the action is proper, in which case we have a slice of the action. 
Let $X$ be a smooth manifold with a smooth action of a \textit{compact} Lie group $K$ and $\mathfrak{k}$ be the Lie algebra of $K$. 
Put $C^{p, q} := S^p \mathfrak{k}^\vee \otimes \Omega^{q-p} (X)$. 
Identifying elements of the symmetric product $S^p \mathfrak{k}^\vee$ with the degree $p$-homogeneous polynomial maps on $\mathfrak{k}$, we regard $C^{p,q}$ the space of $p$-homogeneous polynomial maps from $\mathfrak{k}$ to $\Omega^{q-p} (X)$. 
Consider the subspace consisting of such maps which are $K$-equivariant with respect to the adjoint right action on $\mathfrak{k}$ and the induced right action on $\Omega^* (X)$: 
\begin{equation}
C^{p, q}_K := (S^p \mathfrak{k}^\vee \otimes \Omega^{q-p} (X) )^K. 
\end{equation}
For an element $\varphi = \sum_i P_i \otimes \varphi_i \in C^{p, q}_K$ and $\xi \in \mathfrak{k}$, we put $\varphi_\xi = \sum_i P_i (\xi) \cdot \varphi_i$. 
Then $C^{p,q}_K$ becomes a double complex by giving the differentials $d: C^{p, q} \to C^{p, q+1}$, ${^\hbar \delta}: C^{p, q} \to C^{p+1, q}$ by 
\begin{equation} 
(d \varphi)_\xi = d (\varphi_\xi), \quad ({^\hbar \delta} \varphi)_\xi = i_{\hbar \xi^\#} (\varphi_\xi). 
\end{equation}
Indeed, we have $((d {^\hbar \delta} + {^\hbar \delta} d) \varphi)_\xi = L_{\hbar \xi^\#} (\varphi_\xi) = \varphi_{\hbar [\xi, \xi]} = 0$ as the map $\xi \mapsto \varphi_\xi$ is $K$-equivariant. 
The \textit{Cartan model ${^\hbar H_{\mathrm{dR}, K}^*} (X)$ of equivariant cohomology} is defined to be the cohomology of the total complex 
\[ (\Omega_K^* (X), {^\hbar d_K}) := (\bigoplus_{p+q = *} C^{p,q}_K, d + {^\hbar \delta}) = \bigoplus_{2i+j = *} (S^i \mathfrak{k}^\vee \otimes \Omega^j (X))^K \] 
of the double complex $C^{p, q}_K$. 
We call elements of $\Omega_K^k (X)$ \textit{$K$-equivariant $k$-forms}. 

This cohomology ${^\hbar H_{\mathrm{dR}, K}^*} (X)$ is known to be naturally isomorphic to the equivariant cohomology $H_K^* (X, \mathbb{R})$ for any (non-compact) $X$ and compact Lie group $K$ (cf. \cite[Section 2.5 and 4.2]{GS}). 
It follows that for a almost connected Lie group $G$ and a maximal compact subgroup $G$, we have a natural isomorphism: 
\[ {^\hbar \Phi}: H_G^* (X, \mathbb{R}) \to {^\hbar H_{\mathrm{dR}, K}^*} (X). \]

We have a chain-level pullback map $f^*: \Omega_K^k (Y) \to \Omega_K^k (X)$ along a $K$-equivariant smooth map $f: X \to Y$ which induces a map $f^*: {^\hbar H_{\mathrm{dR}, K}^k} (Y) \to {^\hbar H_{\mathrm{dR}, K}^k} (X)$ and is compatible with $f^*: H_K^k (Y, \mathbb{R}) \to H_K^k (X, \mathbb{R})$. 
We also have a chain-level cup product $\wedge: \Omega_K^k (X) \otimes \Omega_K^l (X) \to \Omega_K^{k+l} (X)$ which induces a map $\wedge: {^\hbar H_{\mathrm{dR}, K}^k} (X) \otimes {^\hbar H_{\mathrm{dR}, K}^l} (X) \to {^\hbar H_{\mathrm{dR}, K}^{k+l}} (X)$. 
This is compatible with the cup product $\smile: H_G^k (X, \mathbb{R}) \otimes H_G^l (X, \mathbb{R}) \to H_G^{k+l} (X, \mathbb{R})$ under the isomorphism ${^\hbar \Phi}: H_G^* (X, \mathbb{R}) \to {^\hbar H^*_{\mathrm{dR}, K}} (X)$. 

When $X$ is a smooth $n$-dimensional oriented compact manifold, we have an integration map $\int_X: \Omega^{n+k}_K (X) \to S^{k/2} \mathfrak{k}^\vee$ for even $k \ge 0$, which is given by the integration of the component in $S^{k/2} \mathfrak{k}^\vee \otimes \Omega^n (X)$. 

\begin{eg}
A ${^\hbar d_K}$-closed equivariant 2-form is a pair (formal sum) $(\omega, \hbar \mu) = \omega + \hbar \mu$ of a $K$-invariant 2-form $\omega$ and a $K$-equivariant smooth map $\mu: X \to \mathfrak{k}^\vee$ satisfying $-d\mu_\xi = i_{\xi^\#} \omega$ for every $\xi \in \mathfrak{k}$. 
On a $2n$-dimensional $X$, the integration $\int_X (\omega+\mu)^{n+k}$ of the equivariant $2(n+k)$-form $(\omega + \mu)^{n+k}$ is by definition the intgeration $\binom{n+k}{k} \int_X \mu^k \omega^n \in S^k \mathfrak{k}^\vee$, which gives a degree $k$ homogeneous function on $\mathfrak{k}$. 
Similarly for instance we can identify the integration $\int_X e^{\omega + \hbar \mu} = \int_X e^{\hbar \mu} \omega^n$ a real analytic function on $\mathfrak{k}$, which gives an element of the formal completion $\hat{S} \mathfrak{k}^\vee = \prod_{k=0}^\infty S^k \mathfrak{k}^\vee$. 
\end{eg}

\subsubsection*{$\bullet$ The uniqueness of equivariant Chern class}

Here we give the proof of Lemma \ref{equivariant Chern uniqueness}. 

\begin{proof}[Proof of Lemma \ref{equivariant Chern uniqueness}]
We show the uniqueness of the equivariant Chern class. 
We may assume the Lie group $G$ and the manifold $M$ is compact by the properties (1) and (2) of $\tilde{c}$. 
Moreover, for a fixed $M$ and $N$, taking large $l > \dim M, \dim N$, we may assume the property (1) holds for $C^l$ maps. 
Endow $TM$ and $L$ with $K$-invariant smooth metrics, then we get a unitary representation $L^2_k (M, L)$ of $K$. 
By Peter--Weyl theorem, for every $s \in L^2_k (M, L)$ and $\epsilon > 0$, there are $\tilde{s} \in L^2_k (M, L)$ and a finite dimensional $K$-invariant subspace $V$ such that $\| s - \tilde{s} \| < \epsilon$ and $\tilde{s} \in V$. 
Take a finite collection of smooth sections $(s_i)$ of $L$ so that it has no common zero. 
Then for $k \gg n/2$, by perturbing $s_i$ in $L^2_k$ as above, we obtain a collection of $C^l$ sections $(\tilde{s}_i)$ and a finite dimensional $K$-invariant subspace $V \subset C^0 (M, L)$ such that $(\tilde{s}_i)$ has no common zero and $\tilde{s}_i \in V$. 
Thus we obtain a $K$-equivariant $C^l$ classifying map $f: M \to \mathbb{P} (V)$. 
Since $V$ is a unitary representation of $K$, we have a group homomorphism $K \to U (V)$. 
By construction, the $K$-equivariant line bundle $L$ is the pullback of the $U (V)$-equivariant line bundle $\mathcal{O} (1)$. 
Therefore it suffices to show $\tilde{c} (\mathcal{O} (1)_{U (V)}) = c_{U (V), 1} (\mathcal{O} (1))$. 
By the group isomorphism $U (V) \cong SU (V) \times U (1)$, we can further reduce the equality to $\tilde{c} (\mathcal{O} (1)_{SU (V)}) = c_{SU (V), 1} (\mathcal{O} (1))$ and $\tilde{c} (\mathcal{O} (1)_{U (1)}) = c_{U (1), 1} (\mathcal{O} (1))$. 
As for $SU (V)$, since we have $(\mathfrak{su}^\vee)^{SU} = 0$, the forgetful map $H^2_{SU (V)} (\mathbb{P} (V), \mathbb{R}) \to H^2 (\mathbb{P} (V), \mathbb{R})$ is injective, so that we must have $\tilde{c} (\mathcal{O} (1)_{SU (V)}) = c_{SU (V), 1} (\mathcal{O} (1))$. 
On the other hand, the kernel $\mathfrak{u} (1)^\vee$ of $H^2_{U (1)} (\mathbb{P} (V), \mathbb{R}) \to H^2 (\mathbb{P} (V), \mathbb{R})$ maps injectively into $H^2_{U (1)} (\mathrm{pt}, \mathbb{R})$ by the pullback along a map $i: \mathrm{pt} \to \mathbb{P} (V)$. 
Therefore, it suffices to show $\tilde{c} (\mathbb{C}_{U (1)}) = c_{U (1), 1} (\mathbb{C})$ for the pullback $\mathbb{C} =i^* \mathcal{O} (1)$, which is endowed with the scaling $U (1)$-action. 
There is a group homomorphism $\varphi: U (1) \to SU (2)$ and a map $j: \mathrm{pt} \to \mathbb{P}^1 \circlearrowleft SU (2)$ to a $U (1)$-invariant point such that the $U (1)$-equivariant line bundle $j^* \mathcal{O} (1)$ over the point is $\mathbb{C}$ endowed with the scaling action. 
Since $\tilde{c} (\mathcal{O} (1)_{SU (2)}) = c_{SU (2), 1} (\mathcal{O} (1))$, we obtain $\tilde{c} (\mathbb{C}_{U (1)}) = j^* \varphi^\# \tilde{c} (\mathcal{O} (1)_{SU (2)}) = j^* \varphi^\# c_{SU (2), 1} (\mathcal{O} (1)) = c_{U (1), 1} (\mathbb{C})$, which proves the claim. 
\end{proof}

\subsubsection{Cartan model of equivariant locally finite homology (current homology)}
\label{section: Cartan model of equivariant locally finite homology (current homology)}

\subsubsection*{$\bullet$ Current homology}

Next we consider the dual construction, which corresponds to the equivariant locally finite homology. 
Firstly we review the current homology. 
Let $X$ be a connected $n$-dimensional smooth manifold. 
For a compact set $D \subset X$, let $\Omega^k_D$ denote the space of smooth $p$-forms supported on $D$ with the $C^\infty$-topology. 
We denote by $\mathcal{D}^k (X)$ the space of compactly supported smooth $k$-forms on $X$ endowed with the weakest topology which makes the natural inclusions $\Omega^k_D (X) \hookrightarrow \mathcal{D}^k (X)$ continuous for all compact sets $D \subset X$. 
Then $\mathcal{D}^k (X)$ is an LF-space. 
Let $\mathcal{D}'_k (X)$ denote the space of continuous linear functionals on $\mathcal{D}^k (X)$. 
We have a boundary map $\partial: \mathcal{D}'_k (X) \to \mathcal{D}'_{k-1} (X)$ adjoint to the differential $d: \mathcal{D}_{k-1} (X) \to \mathcal{D}_k (X)$ (with an appropriate sign) and get the homology group $H^{\mathrm{cur}}_* (X)$ of this complex $(\mathcal{D}'_* (X), \partial)$. 

A smooth $p$-chain $c: \Delta^k \to X$ defines an element of $\mathcal{D}'_k (X)$ by the integration $\varphi \mapsto \int_{\Delta^k} c^* \varphi$ and this gives a linear map $C^{\mathrm{lf}}_k \to \mathcal{D}'_k (X)$. 
It is known by \cite{deR} that the homology $H^{\mathrm{cur}}_* (X)$ is isomorphic to the locally finite homology $H^{\mathrm{lf}}_* (X, \mathbb{R})$ via the map $C^{\mathrm{lf}}_k \to \mathcal{D}'_k (X)$ given as above. 

For a proper smooth map $f: X \to Y$, we have a chain-level pushforward map $f_*: \mathcal{D}'_k (X) \to \mathcal{D}'_k (Y)$ adjoint to the proper pullback $f^*: \mathcal{D}^k (Y) \to \mathcal{D}^k (X)$. 
This induces the pushforward map $f_*: H^{\mathrm{cur}}_k (X) \to H^{\mathrm{cur}}_k (Y)$. 
The cap product 
\[ \frown: H^{\mathrm{cur}}_k (X) \otimes H_{\mathrm{dR}}^l (X) \to H^{\mathrm{cur}}_{k-l} (X) \]
is induced from the chain-level map 
\[ \mathcal{D}'_k (X) \otimes \Omega^l (X) \to \mathcal{D}'_{k-l} (X): \sigma \otimes \varphi \mapsto \sigma (\varphi \wedge \cdot). \]
When $X$ is oriented and the action is orientation preserving, the closed current $\int_X: \mathcal{D}^n (X) \to \mathbb{R}$ gives the fundamental class $[X] \in H^{\mathrm{cur}}_n (X)$. 
All of these constructions are compatible with those counterpart of the locally finite homology $H^{\mathrm{lf}}_* (X, \mathbb{R})$. 

\subsubsection*{$\bullet$ Equivariant current homology}

In the equivariant setup, we consider the double complex 
\begin{align}
C^K_{p, q} := (S^{-p} \mathfrak{k}^\vee \otimes \mathcal{D}'_{q-p} (X))^K
\end{align}
with the differentials ${^\hbar \delta}: C^K_{p, q} \to C^K_{p-1, q}$, $\partial: C^K_{p, q} \to C^K_{p, q-1}$ defined by 
\begin{align*} 
\langle ({^\hbar \delta} \sigma)_\xi, \varphi \rangle 
'&:= (-1)^{n-(q-p)} \langle \sigma_\xi, i_{\hbar \xi^\#} \varphi \rangle, 
\\
\langle (\partial \sigma)_\xi, \varphi \rangle 
&:= (-1)^{n-(q-p)} \langle \sigma_\xi, d\varphi \rangle. 
\end{align*}
Here we put $S^{-p} \mathfrak{k}^\vee = 0$ for $p > 0$. 
These are compatible with $(C_K^{p, q}, \delta, d)$ under the natural inclusion $C_K^{p, q} \hookrightarrow C^K_{-p, n-q}$ for oriented $X$. 
We define the \textit{Cartan model ${^\hbar H^{\mathrm{cur}, K}_k} (X)$ of equivariant current homology} to be the homology of the total complex $(\mathcal{D}')^K_k (X) := \bigoplus_{p+q = k} C^K_{p, q} = \bigoplus_{2i+ (n-j)=n-k} (S^i \mathfrak{k}^\vee \otimes \mathcal{D}'_j (X))^K$ with the differential $\partial + {^\hbar \delta}$. 
For a $K$-equivariant proper smooth map $f: X \to Y$, we have a chain-level pushforward map $f_*: (\mathcal{D}')^K_k (X) \to (\mathcal{D}')^K_k (Y)$, which induces the pushforward map $f_*: {^\hbar H^{\mathrm{cur}, K}_k} (X) \to {^\hbar H^{\mathrm{cur}, K}_k} (Y)$. 

The equivariant cap product $\frown: {^\hbar H^{\mathrm{cur}, K}_k} (X) \otimes {^\hbar H_{\mathrm{dR}, K}^l} (X) \to {^\hbar H^{\mathrm{cur}, K}_{k-l}} (X)$ and the equivariant fundamental class $[X]^K \in {^\hbar H^{\mathrm{cur}, K}_n} (X)$ are given similarly as the non-equivariant case and are compatible with those of locally finite homology. 
We also have the evaluation map $\rho^T_\xi: S^p (\mathfrak{t} \times \mathfrak{k})^\vee \otimes \mathcal{D}'_{q+p} (X) \to \bigoplus_{0 \le r \le p} S^r \mathfrak{k}^\vee \otimes \mathcal{D}'_{q+p} (X)$. 

When $X$ is oriented, the inclusion $C_K^{p, q} \hookrightarrow C^K_{-p, n-q}$ gives the isomorphism $([X]^K \frown \cdot): {^\hbar H_{\mathrm{dR}, K}^k} (X) \to {^\hbar H^{\mathrm{cur}, K}_{n-k}} (X)$. 
We can check this using the spectral sequence associated to the double complexes $C_K^{p,q}$ and $\check{C}_K^{p,q} := C^K_{-p, n-q}$ (cf. \cite[Section 10.10 and 6.5]{GS}).

\end{document}